\documentclass{amsart}

\usepackage{amssymb, amsmath}

\usepackage{xcolor}
\usepackage{enumitem}
\usepackage{verbatim}
\usepackage{hyperref}

\usepackage{fontenc}
\usepackage{mathtools}

\mathtoolsset{showonlyrefs=true}

\newtheorem{theorem}{Theorem}
\newtheorem{prop}[theorem]{Proposition}
\newtheorem{lemma}[theorem]{Lemma}
\newtheorem{cor}[theorem]{Corollary}

\theoremstyle{definition}
\newtheorem{defn}[theorem]{Definition}
\newtheorem{remark}[theorem]{Remark}

\renewcommand{\[}{\begin{equation}\begin{aligned}}
      \renewcommand{\]}{\end{aligned}\end{equation}}

\author{J.M. Daniels-Holgate}
\address{School of Mathematical Sciences, Queen Mary University of London, London, UK}
\email{j.danielsholgate@qmul.ac.uk}

\author{Or Hershkovits}
\address{Department of Mathematics, University of Maryland 
  MD, USA \newline \hspace*{3mm} Einstein institute of Mathematics, The Hebrew University of Jerusalem, Israel}
\email{orher@umd.edu}

\author{G\'abor Sz\'ekelyhidi}
\address{Department of Mathematics, Northwestern University, Evanston,
  IL, USA}
\email{gaborsz@northwestern.edu}

\title{Generic Mean Curvature Flow has finitely many singularities}
\date{}

\begin{document}

\begin{abstract}
We show that for a generic compact initial surface in $\mathbb{R}^3$, and any generic compact initial mean convex hypersurface in higher dimensions, the mean curvature flow encounters only finitely many singularities. This resolves the full generic Morse regularity conjectures of Ilmanen and
Colding-Minicozzi-Pedersen in the above two cases. More generally, we show that if a mean curvature flow of a compact hypersurface has only multiplicity one spherical and cylindrical singularities, then after a small
perturbation of the initial condition only non-degenerate cylindrical and spherical singularities appear.  
\end{abstract}

\maketitle

\section{Introduction}
A Mean Curvature Flow (MCF) of hypersurfaces in $\mathbb{R}^{n+1}$ is a
one parameter family $\{N_t\}_{t\in[0,T)}$, $T\in(0,\infty]$, of hypersurfaces $N_t \subset \mathbb{R}^{n+1}$ evolving according to the
equation
\begin{equation}\label{flow_eq}
    \left(\frac{\partial x}{\partial t}\right)^\perp = \mathbf{H}(x,t).
\end{equation}  
That is, points move in the direction of their mean curvature vector, $\mathbf{H}$. The mean curvature flow can be thought of as the (negative)
$L^2$-gradient flow of the area functional. It is a natural extrinsic analog of the Ricci flow, and it is perhaps the most natural evolution equation for submanifolds of Riemannian manifolds. Geometric evolution equations have had many important applications to geometry and topology, most famously the resolution of the Poincar\'e conjecture using the Ricci flow~\cite{Perelman}. For a selection of other applications of Ricci flow, mean curvature flow, and its variants, see for instance \cite{BrendleSchoen, BK23b, HI01, HK19, LM23, FS08}.

Every MCF starting from a closed, compact hypersurface encounters a singularity in finite time, at which point the flow need not disappear altogether, but at which \eqref{flow_eq} becomes ill-defined in a smooth sense. To address that, notions of weak solution of \eqref{flow_eq} were developed in \cite{Brakke, ES91, CGG91, Ilmanen94}. The analysis of the singular behavior, and the regularity theory of that weak flow has been studied extensively in the last 4 decades \cite{White00,White02,White15,BK23,Hui90,Wang16,Wang2014,Wang16b,CS,CM15,CM15_1, Brendle16, CHH22,BL25,HS99,HK16,HW20,SWX25_2,CHN13,ADS,BC,CDHS}, the basic question being ``how similar is such weak MCF to an honest MCF''. Simple examples suggest that at the very least one has to contend with the presence of singularities that are modelled on collapsing spheres $\mathbb{S}^n$ or generalised cylinders $\mathbb{S}^{k}\times \mathbb{R}^{n-k}, 1\leq k<n$ (from now on, we drop the `generalised'). Indeed, Huisken predicted that for a generic initial data, only multiplicity one spherical and cylindrical
singularities appear, see Ilmanen~\cite[Problem 8]{IlmanenProb}.

In some important cases this prediction has now been justified.  Under the assumption of mean-convexity ($H>0$) on the initial data, Huisken--Sinestrari~\cite{HS99} showed that at the first singular time, all singularities are weakly convex, and therefore can only be spherical or cylindrical. Later, under the same mean-convex assumption on the initial data, White~\cite{White02, White15} showed the flow through singularities encounters only multiplicity one spherical and cylindrical singularities. A recent breakthrough, Bamler and Kleiner's resolution of the multiplicity one conjecture, \cite{BK23}, along with Brendle's classification of genus zero self-shrinkers, \cite{Brendle16}, implies that the same holds for flows of embedded spheres in $\mathbb{R}^3$. 

More generally, in \cite{CM12}  Colding and Minicozzi laid out a program suggesting that, while arbitrary MCF in all dimensions might be wild, it should be the case that for generic initial data, all singularities are spherical or cylindrical. For generic MCF in $\mathbb{R}^3$ this was shown by \cite{BK23}, relying on the seminal work of Chodosh-Choi-Mantoulidis-Schulze~\cite{CCMS24, CCMS24_2} which was conditioned on the multiplicity one conjecture. Assuming a strong form of the  multiplicity one conjecture in higher dimensions, this result was also extended to dimensions 4 and 5 in \cite{CMS}.  Note that from Angenent's construction of a genus one self-shrinker, \cite{Ang92}, already in $\mathbb{R}^3$ there do exist MCF which encounter singularities that are neither cylindrical nor spherical.

The appearance of only cylindrical and spherical singularities for generic (or constrained) initial data has thus been justified in many cases, however, it was
conjectured by Ilmanen~\cite[Problem 9]{IlmanenProb} (see also
Colding-Minicozzi-Pedersen~\cite{CMP15} or Sun-Xue~\cite[Conjecture
1.1]{SX22}), that for a generic initial condition the singularities are also isolated in spacetime. The so-called `marriage ring' example indicates that cylindrical singularities need not be isolated in general. It is worth noting, that by contrast, compact and asymptotically conical singularities are isolated in space-time\footnote{Asymptotically conical singularities are  known to be isolated for the `canonical' Brakke flows (formally, the inner and outer flows of Hershkovits--White \cite{HW20}), and work of the first-named author shows that in cases with non-unique evolution, there can be a `persistence of singularities': see the examples in \cite{DH25b}.}, see the work of Chodosh--Schulze and their work with the first-named author \cite{CS, CDHS}.

One heuristic reason which made Ilmanen's conjecture plausible from its inception is the following: MCF of smooth hypersurfaces is a (quasilinear) heat equation, so Ilmanen's conjecture can be thought of as being  
 analogous to the generic Morse regularity of
the (linear) heat equation. Indeed it is known that, for generic initial data, the singular set, $\{(x,t)\in\mathbb{R}^{n}\times (0,\infty) \ | \  u(x,t)=|\nabla u(x,t)|=0\}$, of solutions to the heat equation is a collection of isolated points, at which the flow is locally modelled on non-degenerate, second order, homogeneous, caloric polynomials.  This follows from the Thom transversality theorem and the Sard--Smale theorem, see \cite{Damon95, Damon97}. Analogous results have been obtained for elliptic problems \cite{SautTemam, DuboisDufour}. One new difficulty in extending this to the case of MCF is that the evolving hypersurface can change topology when passing through singularities. Nevertheless, Bamler-Schulze-Wang~\cite{BSW26} showed that the space-time track of any MCF with only multiplicity one spherical and cylindrical singularities can be perturbed (no longer satisfying an equation) to have isolated singularities with Morse type behavior.

The last 4 years has seen substantial progress on the generic isolatedness conjecture. A key step was obtained by Sun-Xue ~\cite{SX22},
Sun-Wang-Xue~\cite{SWX25}, where it was shown that non-degenerate cylindrical singularities (which were first introduced by Angenent-Vel\'azquez~\cite{AV97}) are isolated in space-time. In \cite{Sz26, Sz26_2}, the third author was able to obtain the genericity of non-degenerate cylindrical singularities \textit{at the first singular time}, whenever multiplicity one cylindrical singularities are generic.  One can then perturb later time slices to ensure that
all of the singularities that appear are non-degenerate, leading to the
notion of ``almost mean curvature flow'' used in \cite{Sz26_2}. Note that Colding-Minicozzi~\cite{CM12} introduced a similar notion of ``piece-wise'' mean curvature flow, and showed that it can be used to perturb away singularities modelled on non-spherical compact shrinkers.

While the results of \cite{Sz26_2} and \cite{BSW26} are probably sufficient for many potential applications of MCF to geometry and topology, the basic question of whether a generic MCF has only isolated non-degenerate singularities has remained open, even for flows in $\mathbb{R}^3$, and it was only known that the singular set has parabolic Hausdorff dimension  at most $(n-1)$, due to White~\cite{White97}.  Addressing this question is the main contribution of this work.  

\bigskip

To state the main result, we let $N_t$ be a weak mean curvature flow emanating from $N_0$. It follows from \cite{BL25} (see also \cite{HW20}) that if one such flow emanating from $N_0$ encounters only multiplicity one cylindrical and spherical singularities, then this flow is unique. Our main result is the following.

\begin{theorem}\label{thm:main}
Let $N_t\subset\mathbb{R}^{n+1}$ be a mean curvature flow of
compact hypersurfaces, with only multiplicity one spherical and cylindrical singularities. Then 
  there are arbitrarily small $C^2$ perturbations $\tilde{N}_0$ of the initial
  condition $N_0$ such that the resulting flows $\tilde{N}_t$ only have
  non-degenerate cylindrical and spherical singularities. In particular, $\tilde{N}_t$ has only finitely many singular points.
\end{theorem}

Using this, the results of \cite{BK23} and \cite{CCMS24} discussed above imply the following. 

\begin{cor}
  For a generic initial compact surface in $\mathbb{R}^3$ the mean
  curvature flow  has only isolated singularities in space-time. 
\end{cor}

Similarly, the flow of a generic mean convex hypersurface encounter only finitely many singularities, and the same would be true for a generic hypersurface in ambient dimensions $4$ and $5$, assuming a strong form of the multiplicity one conjecture (see \cite[Def A.1 and Theorem A.6]{CMS}). An interesting question, going further than these results, is what we could say about a generic $k$-parameter family of MCFs.    

\subsection*{Outline of the proof}
We now give a brief outline of the proof, focusing on the new ingredients. As mentioned above, in \cite{Sz26,Sz26_2} the third author showed that for generic flows with multiplicity one cylindrical and spherical singularities, the singularities are non-degenerate at the first singular time. The basic mechanism for this is a result that says the following: suppose that $N_t$ is a mean curvature flow of hypersurfaces, smooth for $t\in [0,T)$, with a cylindrical singularity at $(x_0, T)$. Then there exist arbitrarily small perturbations $\tilde{N}_0$ of the initial condition, such that the perturbed flow $\tilde{N}_t$ only has non-degenerate cylindrical singularities in a parabolic neighborhood $B_\delta(x_0) \times (T-\delta^2, T+\delta^2)$, where $\delta$ may depend on the flow $N_t$, but crucially does not depend on the size of the perturbation. 

In proving this result it is important to be able to construct perturbations $\tilde{N}_0$ of $N_0$ which lead to some predetermined perturbations $\tilde{N}_{T'}$ of $N_{T'}$ at some $T' < T$ slightly before the singularity. This in turn boils down to a controllability result for the linearized mean curvature flow equation along $N_t$ for $t\in [0,T']$, which follows from the backwards uniqueness property of the conjugate linearized equation (see \cite[Lemma 5.3]{SX22}, \cite[Lemma 5.2]{HV92}). 

In this paper the main difficulty is to extend this controllability result, as well as its nonlinear version along the flow, to the setting where the flow $N_t$ encounters some non-degenerate cylindrical, and spherical singularities in the interval $(0, T)$. At the heart of this is a new backwards uniqueness theorem for
the conjugate linearized equation at a time slice with only non-degenerate
cylindrical singularities. 
\begin{theorem}\label{thm:linearbackwarduniqueness}
  Suppose that $N_t\subset \mathbb{R}^{n+1}$ is a mean curvature flow of compact hypersurfaces
  that has only non-degenerate cylindrical singularities at time
  $t=t_0$, and which is smooth for $t\in (t_0, t_1]$. Suppose that
  $w(x,t)$ solves the conjugate linearized equation
  \[ (\partial_t + \Delta + |A|^2 - H^2) w = 0, \]
  for $t\in (t_0, t_1]$, and as $t\to t_0$ we have $w(x,t)\to 0$
  locally smoothly away from the singular set. Then $w\equiv 0$. 
\end{theorem}

Somewhat related singular backwards uniqueness theorems have been shown by the first two authors \cite{DH25} (which itself relies on ideas from \cite{Wang2014}). A key new
difficulty in the present setting is that the flow does not satisfy
Type I curvature bounds past the singularity. Another challenge is that in lieu of conical behavior past the singularity, one either has to find a regime in which a remnant of conicality exists (as we eventually do), or one has to work with a more degenerate Carleman estimate, suitable for cylindrical singularities (which we could do to treat neck singularities, but not more generally). Compare \cite{Wang2014,Wang16}. Given this result, we
are able to show that initial perturbations of the flow can give rise
to desired perturbations of the flow even past non-degenerate
cylindrical singularities. Combining this with the techniques of
\cite{Sz26, Sz26_2} then leads to Theorem~\ref{thm:main}.  

\bigskip

 As the paper naturally decomposes into three parts: estimates, perturbations and backwards uniqueness, the relevant proof ideas and comparative discussions appear in the subsequent sections.

\subsection*{AI Disclosure}
We  used ChatGPT to help derive some of the estimates regarding the post singularity behavior of the MCF. Explicitly, we  used it to help derive Theorem \ref{metric_changed_upper_bound} (sharp metric change upper bound) and Lemma \ref{lem:almostconical} (almost conicality away from the tip). The almost conicality estimate, Lemma \ref{lem:almostconical}, which is key for our ability to treat non-degenerate cylindrical singularities in arbitrary dimension (and not just the neck case), was one we did not expect. ChatGPT also suggested the estimate in Lemma~\ref{lem:31} which we could not find in the literature. The contribution of ChatGPT was not merely a proof, but also the statements and an explanation of their relevance to the problem.

None of the text in this article was written by generative AI.

\subsection*{Acknowledgments}
We thank Ao Sun for useful discussions. JMDH acknowledges support from the EPSRC through the grant EP/Y017862/1, Geometric Flows and the Dynamics of Phase Transitions. OH acknowledges that this project has received funding from the European Research Council (ERC) under the European Union’s Horizon 2020 research and innovation programme, grant agreement No. 101116390. GSz was supported in part by NSF grant DMS-2506325.

\section{Preliminaries}
In this section we will recall some preliminary results and
notation. To fix the notation, we will reserve $N_t$ for the original
mean curvature flow that we started with. We will need to consider
various small perturbations of this flow, as well as rescalings, and
we will typically use $L_t$ or $S_t$ for such flows. In analyzing
singularities it is useful to consider the corresponding rescaled mean
curvature flows (see below), which we will primarily denote by
$M_\tau$, using the rescaled time parameter $\tau$. Various estimates
in the paper will depend, even if not stated explicitly, on an entropy
bound for $N_t$, which equivalently gives area ratio bounds for balls
in any rescalings of the $N_t$. 

Suppose that $L_t$ is a mean curvature flow with only cylindrical and
spherical singularities. More precisely, $L_t$ is a unit-regular
cyclic integral Brakke flow, all of whose tangent flows at singular
points are multiplicity one shrinking spheres or cylinders. 

We write
\[ \mathcal{C}_{n,k} = \mathbb{R}^k \times S^{n-k}(\sqrt{2(n-k)})
  \subset \mathbb{R}^{n+1} \]
for the standard cylinder, so that $\sqrt{-t}\mathcal{C}_{n,k}$ is a
self shrinking solution of the mean curvature flow. We will use $y_i$
for coordinates on $\mathbb{R}^k$, $z_\alpha$ for coordinates on the
remaining $R^{n+1-k}$ factor, and $x$ for the position in
$\mathbb{R}^{n+1}$.

If the flow has a cylindrical singularity at a point $X_0 = (x_0,
t_0)$, then we consider the rescaled flow $M_\tau$ centered at $X_0$,
given by
\[ M_\tau = e^{\tau/2} (L_{t_0 - e^{-\tau}} - x_0). \]
Having a cylindrical singularity means that along a sequence of times,
and up to a rotation, we have $M_\tau\to \mathcal{C}_{n,k}$ for some
$k$. It was shown by Colding-Minicozzi~\cite{CM15_1} that in this case
the convergence $M_\tau \to Q\mathcal{C}_{n,k}$ holds for a fixed
rotation $Q$, even without choosing a sequence. We will say in this
case that $M_\tau$ has an $(n,k)$-cylindrical singularity at $\infty$,
and $L_t$ has an $(n,k)$-cylindrical singularity at $X_0$. If we
actually have $M_\tau\to \mathcal{C}_{n,k}$, then we will say that the
flows have $\mathcal{C}_{n,k}$-singularities at infinity, or at $X_0$
respectively.

It was shown by Sun-Xue~\cite{SX22} that if $M_\tau \to
\mathcal{C}_{n,k}$, then after rotating the coordinates we can write $M_\tau$ as the graph of a function $u$ over
$\mathcal{C}_{n,k}$ on expanding balls, where
\[ u(y, z, \tau) = \frac{\sqrt{2(n-k)}}{4\tau} \sum_{i=1}^\ell (y_i^2
  - 2) + o(\tau^{-1}), \]
for a unique $0\leq \ell \leq k$. We say that the singularity is non-degenerate  if
$\ell=k$, and it is $\ell$-degenerate if $\ell < k$. An important
property of non-degenerate singularities is the following stability and
isolatedness property, due to
Sun-Wang-Xue~\cite{SWX25}. 

\begin{theorem}\label{thm:isolated}
  Suppose that $L_t$ has a non-degenerate cylindrical singularity at
  $(x_0, t_0)$. There exists a $\delta_0 > 0$, such that for all $\delta \in (0, \delta_0)$ if $\tilde{L}_t$ is a sufficiently small
  perturbation of $L_t$, then $\tilde{L}_t$ has a non-degenerate
  cylindrical singularity at a point $(x_0', t_0')$ in the parabolic
  $\delta$-neighborhood of $(x_0,t_0)$, and no other singularities
  there. 
\end{theorem}

Going further, Sun-Wang-Xue~\cite{SWX25} proved that as the flow
passes through a non-degenerate $(n,k)$-cylindrical singularity, the
hypersurfaces go through an $(n-k)$-surgery. More precisely we have the
following.
\begin{theorem}\label{thm:nondegsurgery}
  Suppose that the flow $L_t$ has a non-degenerate
  $(n,k)$-cylindrical singularity at the point $(x_0, t_0)$. Then
  there are nearby times $t_- < t_0 < t_+$ and a small $r > 0$ such
  that $L_{t_+}\cap B_r(x_0)$ is obtained by removing a region
  diffeomorphic to $S^{n-k}\times B^k$ from $L_{t_-}\cap B_r(x_0)$,
  and gluing in $B^{n-k+1}\times S^{k-1}$ along the boundary. Here
  $B^k$ is the $k$-dimensional ball.
\end{theorem}

In addition we will need to use the recent resolution of the generalised mean convex
neighborhood conjecture by Bamler-Lai~\cite{BL25} (see also the work of Choi--Haslhofer--Hershkovits \cite{CHH22} for
the neck-pinch in $\mathbb{R}^3$, and the same authors with White \cite{CHHW} for neck-pinches in higher dimensions), which states that if a mean
curvature flow $L_t$ has a cylindrical singularity at $X$, then in a
parabolic neighborhood of $X$ the flow is mean convex (for a suitable
orientation). In particular this has the following consequence (see
\cite[Theorem 1.15]{BL25}).
\begin{theorem}
  Suppose that the flow $N_t$ has only cylindrical and spherical
  singularities, and $L_t^i$ is a sequence of flows converging to
  $N_t$ in the Brakke sense, with support contained in a common compact set. Then for sufficiently large $i$ the flows
  $L_t^i$ also only have cylindrical and spherical singularities.
\end{theorem}

\subsection{Refined normal forms}
In order to perturb away degenerate cylindrical singularities,
\cite{Sz26_2} developed quantitative results for the normal forms of
such singularities, building on work of Sun-Xue~\cite{SX22}. 

The normal form is stated relative to certain model flows. In order to
write these down, we first let $\mathcal{M}^{n', k'}_\tau$ be rescaled
mean curvature flows with non-degenerate $\mathcal{C}_{n',
  k'}$-singularities at infinity (these exists by \cite[Proposition
45]{Sz26_2}). Then we let 
\[ \mathcal{M}^{n,k,\ell}_\tau = \mathbb{R}^{k-\ell}\times
  \mathcal{M}^{n-(k-\ell), \ell}_\tau,  \]
which has an $\ell$-degenerate
$\mathcal{C}_{n,k}$-singularity at infinity. 

It was shown in \cite{Sz26_2} that any rescaled flow $M_\tau$ with an
$\ell$-degenerate $\mathcal{C}_{n,k}$-singularity is graphical on
balls  $B_{K_0\tau^{1/2}}$ over $\mathcal{M}^{n,k,\ell}_\tau$
for any $K_0$, and sufficiently large $\tau$ (depending on $K_0$). The
following definition and result quantify this. The constants $\epsilon_0,
\kappa_0, \gamma_0 > 0$ in the definition can be chosen arbitrarily close to
zero. The $L^2$-norms here, and below, always refer to the Gaussian weighted norm
\[ \Vert u\Vert_{L^2(M)}^2 = \int_M u^2\, e^{-|x|^2/4}\, dV_M, \]
unless otherwise stated. 

\begin{defn}\label{defn:goodasympt}
  Suppose that $M_\tau$ is a rescaled mean curvature flow with an
  $\ell$-degenerate $\mathcal{C}_{n,k}$-singularity at infinity. We
  say that $M_\tau$ has good asymptotics at time $T > 0$, if for all
  $\tau \geq T$, $M_\tau$ can be written as the graph of $u(\cdot,
  \tau)$ over the model flow $\mathcal{M}^{n,k, \ell}_\tau$ on
  $B_{10\tau^{1/2}}$, satisfying:
  \begin{itemize}
    \item $\Vert u(\cdot, \tau)\Vert_{C^4} < \epsilon_0$ on
      $B_{10\tau^{1/2}}$, while $\Vert u(\cdot, \tau)\Vert_{C^4} <
      \tau^{-3/2}$ on the $\tau^{\gamma_0}$-ball.
    \item $\Vert u(\cdot, \tau)\Vert_{L^2(B_{10\tau^{1/2}})} < \tau^{-19/10}$,
    \item $\Vert\partial_{y_q} u(\cdot, \tau)\Vert_{L^2(B_{10\tau^{1/2}})}
      < e^{-\kappa_0 \tau}$ for the degenerate directions $q=\ell+1,
      \ldots, k$.
    \end{itemize}

    We define a corresponding notion for a mean curvature flow $L_t$ as
well: if $L_t$ has an $\ell$-degenerate cylindrical singularity at
$(x_0, t_0)$, then we say that this singularity is $\ell$-degenerate
at scale $r > 0$, if the corresponding rescaled flow $M_\tau$,
centered at $(x_0, t_0)$, has good asymptotics at time $-2\ln r$ after
a suitable rotation.  
 \end{defn}

 We have the following, from \cite{Sz26_2}.
 \begin{prop}
   Suppose that $M_\tau$ is a rescaled mean curvature flow with an
   $\ell$-degenerate $\mathcal{C}_{n,k}$-singularity at infinity. Then
   there exists a $T > 0$ such that $M_\tau$ has good asymptotics at
   $T$. 
\end{prop}

The importance of this notion is the following, showing that if an
$(n,k)$-cylindrical 
singularity is $\ell$-degenerate at some scale $r > 0$, then nearby
$\ell$-degenerate $(n,k)$-cylindrical singularities must also be
$\ell$-degenerate at a controlled scale. See \cite[Proposition
35]{Sz26_2}. 
\begin{prop}\label{prop:nearbydegen}
  For any $r > 0$ there exist $r', \delta > 0$ satisfying the
  following. Suppose that $L_t$ has an $\ell$-degenerate
  $(n,k)$-cylindrical singularity at $X_0 = (x_0, t_0)$, which is
  $\ell$-degenerate at scale $r$. Let $L'_t$ be another mean curvature
  flow, which is $\delta$-close to $L_t$ in a metrization of Brakke
  convergence, and suppose that $L'_t$ has an $\ell$-degenerate
  $(n,k)$-cylindrical singularity at $X_0'$ in the parabolic
  $\delta$-neighborhood of $X_0$. Then $X_0'$ is $\ell$-degenerate at
  scale $r'$. 
\end{prop}

While from Theorem~\ref{thm:isolated} we know that non-degenerate
cylindrical singularities are isolated and cannot be perturbed away, it is certainly
possible (and expected) that after a perturbation, a degenerate cylindrical
singularity becomes a different type of cylindrical or spherical
singularity, or even many such singularities. Estimates for
the entropy restrict which types of singularities can appear after a
perturbation, and even for cylindrical singularities of the same type,
a degenerate singularity cannot be perturbed into a more degenerate
one. To state this precisely, we use the following order on
the types of degenerate singularities:

\begin{defn}\label{defn:singorder}
Let $(k,  \ell)$ and $(k', \ell')$ be types of degenerate cylindrical
singularities, i.e. $0 < k, k' < n$ and $0 \leq \ell < k$, $0 \leq
\ell' < k'$. We say that $(k', \ell') \succ (k, \ell)$
if either $k' > k$, or $k'=k$ and $\ell' < \ell$. In this case we
think of an $\ell'$-degenerate $(n,k')$-cylindrical singularity as
``worse'' than an $\ell$-degenerate $(n,k)$-cylindrical singularity.

We also write $\Sigma_{k,\ell}(L_t)$ for the set of $\ell$-degenerate
$(n,k)$-cylindrical singularities of the flow $L_t$, and let
$\Sigma^>_{k, \ell}(L_t)$ denote the union of all $\Sigma_{k',
  \ell'}(L_t)$ with $(k', \ell') \succ (k, \ell)$. We emphasize that
non-degenerate cylindrical and spherical singularities are not included
in this notation. 
\end{defn}

We
have the following, from \cite{Sz26_2}.
\begin{prop} \label{prop:singularityorder}
  Suppose that $L^\infty_t$ has only cylindrical and spherical singularities,
  and $L^i_t$ is a sequence of flows converging to $L_t^\infty$ in the Brakke
  sense. Let $X_i$ be $\ell_i$-degenerate $(n, k_i)$-cylindrical
  singularities of the flows $L^i_t$, such that $X_i \to
  X_\infty$. Then $X_\infty$ is an $\ell_\infty$-degenerate $(n,
  k_\infty)$-cylindrical singularity of the flow $L^\infty_t$, where
  $k_\infty \geq \lim\sup_i k_i$. If in addition $k_i=k_\infty$ for
  all (sufficiently large) $i$, then we also have $\ell_\infty \leq
  \lim\inf_i \ell_i$. 
\end{prop}

This result can be applied in particular to translations of
$L^\infty_t$ in space-time. In that case an important consequence is that if
$\Sigma^>_{k, \ell}(L_t) = \emptyset$, then $\Sigma_{k, \ell}(L_t)$ is
closed. 

\subsection{Perturbing away degenerate singularities locally}\label{sec:perturb}
Next we will recall the results from \cite{Sz26_2} that we
will use in order to perturb away degenerate singularities.

Let us assume that $M_\tau$ is a rescaled mean curvature flow with an
$\ell$-degenerate $\mathcal{C}_{n,k}$-singularity at infinity. Suppose
that $M_\tau$ has good asymptotics at time $T_0$ according to
Definition~\ref{defn:goodasympt}. The following is \cite[Proposition
39]{Sz26_2}.

\begin{prop}\label{prop:perturb1}
Let $c_1\in (0,1)$, and let $\kappa_1 = \kappa_0/4$ for the $\kappa_0$
in Definition~\ref{defn:goodasympt}. Suppose that $\epsilon_0,
T_0^{-1}$ are sufficiently small. For any $T_1 \geq
  T_0$ there exists $r_0=r_0(c_1, T_1)$ with the following property. 
  Suppose that the rescaled flow $N_\tau$ satisfies the following conditions
  for some $\gamma > 0$: 
  \begin{itemize}
     \item[(i)] For $\tau \geq T_1$, as long as $c_1^{-1}\gamma T_1 e^{\tau-T_1} < \epsilon_0$,
       $N_\tau$ is the graph of $v$ over $M_\tau$ on the
       $10\tau^{1/2}$-balls, and $\Vert v\Vert_{C^4} \leq
       c_1^{-1} T_1\gamma e^{\tau-T_1}$. Suppose that $\gamma$ is small
       enough so that this holds in
       particular for $\tau\in [T_1, 3T_1]$, i.e. $\gamma < \epsilon_0
       c_1 T_1^{-1} e^{-2T_1}$. 
     \item[(ii)] $\Vert v(\cdot, T_1)\Vert_{10T_1^{1/2}} \geq c_1 \gamma$.
     \item[(iii)] $\Vert v(\cdot, T_1+ 1)\Vert_{10(T_1+1)^{1/2}} \geq
         e^{\frac{1}{2}-\kappa_1} \Vert v(\cdot, T_1)\Vert_{10T_1^{1/2}}$.
   \end{itemize}
  Let $x_0\in \mathbb{R}^{n+1}$ be a vector with no $z_\alpha$ or
  $y_i$ component for $i\leq \ell$ (i.e. it is only nonzero in the
  degenerate directions) such that $|x_0| < r_0 \gamma^{1-\kappa_1}$.
  Then the translated flow $N'_\tau = N_\tau + e^{(\tau-T_1)/2}x_0$ does not
  have an $(n,k)$-cylindrical singularity at infinity (i.e. it does
  not converge to a rotation of $\mathcal{C}_{n,k}$). 
\end{prop}

We briefly recall how this result was used in \cite[Proposition
40]{Sz26_2} to perturb away degenerate cylindrical singularities in a
parabolic neighborhood. Suppose that $L_t$ is a mean curvature flow
with an $\ell$-degenerate $\mathcal{C}_{n,k}$-singularity at $(0,0)$.
In practice this flow is a small perturbation and translation and
rotation of the original mean curvature flow $N_t$ that we
started with. Suppose that the singularity at $(0,0)$ is $\ell$-degenerate at
scale $r > 0$. By Proposition~\ref{prop:nearbydegen}, there are
$\delta, r' > 0$ such that if $L_t'$ is $\delta$-close to $L_t$, and
has an $\ell$-degenerate $(n,k)$-cylindrical singularity in the
$\delta$-neighborhood of $(0,0)$, then this singularity is 
$\ell$-degenerate at scale $r'$.  We let $T_1 = -2\ln r'$, so that the
rescaled flows corresponding to all of these nearby $\ell$-degenerate
singularities have good asymptotics at time $T_1$.

Let $M_\tau$ denote the rescaled flow corresponding to $L_t$, centered
at the origin. For sufficiently small $\mathbf{a}\in
\mathbb{R}^{k-\ell}$, in \cite{Sz26_2} we considered perturbations
$M^{\mathbf{a}}_{T_1+1}$ of $M_{T_1+1}$ by taking the graph of the
function $\chi(|x| T_1^{-\beta/2}) \mathbf{a}\cdot y$. Here $\beta >
0$ is a small constant (can be chosen arbitrarily small), $\chi$ is a
standard cutoff function such that $\chi(s)=0$ for $s > 3/2$ and
$\chi(s)=1$ for $s < 1$, and we write
\[ \mathbf{a}\cdot y = \sum_{j = \ell+1}^k a_{j-\ell} y_j, \]
in terms of the components $a_i$ of $\mathbf{a}$. We then let
$M^{\mathbf{a}}_\tau$ be the rescaled flows, for $\tau\geq T_1 + 1$,
with initial condition $M^{\mathbf{a}}_{T_1+1}$. This defines a
$(k-\ell)$-dimensional family of rescaled flows, and we write
$L^{\mathbf{a}}_t$ for the corresponding (unrescaled) mean curvature flows, that are
defined for $t \geq -e^{-T_1-1}$.
 In \cite{Sz26_2},  Proposition~\ref{prop:perturb1} was used to show that if
  $L^{\mathbf{a}_1}_t$ and $L^{\mathbf{a}_2}_t$ both have
  $\ell$-degenerate $(n,k)$-cylindrical singularities in a small
  neighborhood of $(0,0)$, then the degenerate $y$-components  of
  these singularities must be at least $r_0|\mathbf{a}_1 -
  \mathbf{a}_2|^{1-\kappa_1}$ apart, for the $r_0, \kappa_1$ in that
  proposition. With a simple covering argument (see \cite[Proposition 41]{Sz26_2}) this implies that for
  almost every value of the perturbation parameter $\mathbf{a}$ the
  flow $L^{\mathbf{a}}_t$ has no $\ell$-degenerate $(n,k)$-cylindrical
  singularities in a fixed small neighborhood of $(0,0)$. 

In this paper our goal is to perturb away degenerate singularities
using perturbations of the \emph{initial condition} of our original flow,
rather than a time slice such as $M_{T_1+1}$. There are two aspects of
this:
\begin{enumerate}
  \item It is not possible to find initial perturbations of the flow
    that will lead to perturbations $M^{\mathbf{a}}_{T_1+1}$ of
    $M_{T_1+1}$ that have an exact graphical representation as above,
    supported on a ball where $M_{T_1+1}$ is cylindrical. Indeed, any
    initial perturbation will lead to a global perturbation at any
    later time. This difficulty was also encountered in \cite{Sz26},
    and as in that paper, it is enough to find initial perturbations
    that at time $T_1+1$ along the rescaled flow behave approximately
    like the exact graphical perturbations described above. 
  \item The much more substantial difficulty is that in the setting of
    \cite{Sz26} we were only concerned with the first time singularity
    of the flow. In other words the flow that we were perturbing was
    smooth from the initial time until the time slice corresponding to
    $M_{T_1+1}$. In the current paper this is not the case, and we
    will need to construct initial perturbations that have the desired
    effect at a later time, even after passing through
    singularities. This will be the content of Section~\ref{sec:initialperturb}, using the backwards uniqueness property shown in Sections~\ref{sec:backwards1} and \ref{sec:backwardsunique}. 
  \end{enumerate}
  
In this section we consider perturbations of $M_{T_1+1}$
that are approximately of the form as the exact graphical perturbation
above, and show how these can still be used to perturb away degenerate
singularities. These ideas are mostly already contained in
\cite{Sz26}. We summarize the types of perturbations of $M_{T_1+1}$
that we consider in the following definition. 

\begin{defn}\label{defn:kappagood}
  Suppose that, as above, $L_t$ is  a MCF encountering only spherical and cylindrical singularities, that has an $\ell$-degenerate
  $\mathcal{C}_{n,k}$-singularity at $(0,0)$. 
  Let $\kappa > 0$, and suppose that $T_1$ is defined as above. 

  We say that  a continuous  family of flows (in the Brakke flow topology), $L^{\mathbf{a}}_t$
  is a $\kappa$-good family of perturbations of $L_t$ at $(0,0)$, if $L^0_t=L_t$ and if it satisfies the
  following. Let $M^{\mathbf{a}}_\tau$ be the corresponding rescaled
  flows, centered at the origin, and suppose that $M^{\mathbf{a}}_{T_1+1}$
  is smooth for $|\mathbf{a}| < \kappa$. We
  require that $M^{\mathbf{a}_2}_{T_1+1}$ is the graph of a function
  $u$ over $M^{\mathbf{a}_1}_{T_1+1}$ (globally, not just on a ball),
  such that
  \[ \left| u - \chi(|x| T_1^{-\beta/2})
      (\mathbf{a}_2-\mathbf{a}_1)\cdot y\right| \leq \kappa
    |\mathbf{a}_2-\mathbf{a}_1|,\]
  for  all $|\mathbf{a}_1|, |\mathbf{a}_2| < \kappa$.
\end{defn}

Our goal in this section is to prove that if $\kappa$ is sufficiently
small (depending on $L_t$), then outside of a set of $\mathbf{a}$ with
measure zero, the flows $L^{\mathbf{a}}_t$ have no $\ell$-degenerate
$(n,k)$-cylindrical singularities near $(0,0)$. To set this up more
precisely, we follow the notation of \cite{Sz26_2}. Let  $\beta\in
\mathbb{R}^{k-\ell}, \alpha_1\in \mathbb{R}^\ell, \alpha_2\in
\mathbb{R}^{n-k+1}$ and $s\in \mathbb{R}$. Write $\Phi = (\alpha_1,
\alpha_2, s)$, and define
\[ L^{\mathbf{a}, \Phi, \beta}_t = L^{\mathbf{a}}_{t+s} + (\alpha_1,
  \beta, \alpha_2). \]
Note that $\beta$ encodes the translation in the degenerate $y$-directions. Note further that by definition and Proposition \ref{prop:nearbydegen}, all the rescaled flows around singularities of the $L^{\mathbf{a}, \Phi, \beta}_t$ in a neighborhood of $(0,0)$ that we encounter, have good asymptotics at time
  $T_1$, if they have an $\ell$-degenerate
  $\mathcal{C}_{n,k}$-singularity at infinity.

The key result is the  following, similar to
\cite[Proposition 40]{Sz26_2}.

\begin{prop}\label{prop:perturb22}
  Suppose that $T_1$ is sufficiently large (depending on the dimension
  and entropy bound), $\kappa$ is sufficiently small, depending on
  $L_t$, and $L^{\mathbf{a}}_t$ is a $\kappa$-good family of
  perturbations.

  There is an $r_2=r_2(T_1) > 0$, with the
  following property. Suppose that $|\mathbf{a}_1|, |\mathbf{a}_2| <
  \kappa$, and $L^{\mathbf{a}_1}_t$, $L^{\mathbf{a}_2}_t$ have
  $\ell$-degenerate $(n,k)$-cylindrical singularities at the points
  $(y_i, z_i, t_i)$, in the $r_2$-neighborhood of
  $(0,0,0)$. Decompose $y_i = (y_i', y_i'')$ into the non-degenerate and
  degenerate components according to $\mathbb{R}^k = \mathbb{R}^\ell
  \times \mathbb{R}^{k-\ell}$. Then the degenerate components satisfy
  $|y_1'' - y_2''| \geq r_2 |\mathbf{a}_1 -
  \mathbf{a}_2|^{1-\kappa_1}$, for the $\kappa_1$ in Proposition~\ref{prop:perturb1}. 
\end{prop}
  The proof will follow the proof in \cite{Sz26_2} closely, with one
  important difference. In the current setting we are not assuming
  that our flows are mean convex, so we cannot simply use time
  translated copies of our flow to obtain barriers. Instead we use the 
  more general barrier construction from \cite{Sz26}. In order to
  state the result, suppose that $S_t$ is a mean curvature flow with
  only cylindrical and spherical singularities. By the work of
  Bamler-Lai~\cite{BL25} (see also \cite{CHH22}), each such
  singularity has a mean convex neighborhood. By a covering argument
  we can find open sets $\mathcal{U}_1\subset\subset
  \mathcal{U}_2\subset \mathbb{R}^{n+1}\times[0,\infty)$ such that the flow is mean convex
  (with suitable orientations) on $\mathcal{U}_2$, and is smooth
  outside of $\mathcal{U}_1$. From \cite[Proposition 14]{Sz26} we have
  the following. 

  \begin{prop}\label{prop:barrier}
 There are $\xi, \sigma > 0$ depending on the flow $S_t$, such
  that for all $T > 0$ we have the following. We can choose open sets
  $U_1 \subset\subset U_2\subset \mathbb{R}^{n+1}$ such that on the time interval
  $I := [T-2\xi, T+2\xi]$ we have
  \[\label{eq:Ucontain} \mathcal{U}_1 \cap (\mathbb{R}^{n+1}\times I)\subset U_1 \times I \subset U_2 \times I\subset
    \mathcal{U}_2. \]
  Suppose that $t_0\in
  [T-\xi, T+\xi]$, $t_0' < t_0$, and $S'_t$ is a mean curvature
  flow on $[t_0', T+2\xi]$ that is
  sufficiently close to $S_t$. Suppose in addition that $S''_t$ is
  a mean curvature flow on $[t_0, T+2\xi]$ satisfying that for some
  $\epsilon < \min\{\sigma, t_0-t_0'\}$ we have
  \begin{enumerate}
    \item Outside of $U_1$, the surface $S''_{t_0}$ is in the
      $\sigma\epsilon$-neighborhood of $S'_{t_0}$,
    \item In $U_2$, the surface $S''_{t_0}$ lies between $S'_{t_0 \pm
        \epsilon}$.
    \end{enumerate}
  Then for $t\in [t_0, T+\xi]$, on $U_1$, the surface $S''_t$ lies
  between $S'_{t\pm \epsilon}$. 
\end{prop}

To be a little more precise about $S''_t$ lying between $S'_{t\pm
  \epsilon}$ note that by the mean convexity property due to
Bamler-Lai~\cite{BL25}, we can find domains $\Omega_t'$ such that
$S'_t\cap U_2 = \partial \Omega'_t \cap U_2$, and $\Omega'_{t_1}
\subset \Omega'_{t_2}$ for $t_1 < t_2$. We say that $S''_t$ lies
between $S'_{t\pm \epsilon}$ in $U_2$ if $S''_t \cap U_2\subset
\Omega'_{t+\epsilon}$ and $S''_t\cap U_2$ is disjoint from the interior
of $\Omega'_{t-\epsilon}$. 

We can now prove Proposition~\ref{prop:perturb22}, mostly following the
proof of \cite[Proposition 40]{Sz26_2}. 
\begin{proof}[Proof of Proposition~\ref{prop:perturb22}]
 Recall that $L_t$ has an $\ell$-degenerate
 $\mathcal{C}_{n,k}$-singularity at the origin. As in the proof of
 \cite[Proposition 40]{Sz26_2} we know that if $L^{\mathbf{a}}_t$ has
 an $\ell$-degenerate $(n,k)$-cylindrical singularity at $(x_0, t_0)$
 with $|x_0| + |t_0|^{1/2} < r_2$, then there is a rotation $Q$ with
 $|Q - Id| < \Psi(r_2, |\mathbf{a}|\, |\, T_1)$ such that
 $QL^{\mathbf{a}}_t$ has an $\ell$-degenerate
 $\mathcal{C}_{n,k}$-singularity (i.e. the tangent flow is actually
 $\mathcal{C}_{n,k}$, not just a rotation). So, once $\kappa, r_2$ are
 small enough, the $\mathbb{R}^{k-\ell}$ subspaces corresponding to the
 degenerate directions of the $\ell$-degenerate singularities that
 appear are all small perturbations of each other.

 After a $\Psi(r_2, \kappa\,|\, T_1)$ rotation and translation we can
 assume that $L^{\mathbf{a}_1}_t$ has an $\ell$-degenerate
 $\mathcal{C}_{n,k}$-singularity at $(0,0)$. Suppose that
 $L^{\mathbf{a}_2}_t$ also has an $\ell$-degenerate
   $(n,k)$-cylindrical singularity at a
   point $(y_0', y_0'', z_0, s_0)$ with 
   \[ |y_0'| + |y_0''| + |z_0| + |s_0|^{1/2} < 2r_2. \]
  Our goal is to show that $|y_0''| \geq r_2 |\mathbf{a}_2 -
  \mathbf{a}_1|^{1-\kappa_1}$. For this, we define the translated flow
  \[ \widetilde{L}^{\mathbf{a}_2}_t = L^{\mathbf{a}_2}_{t+s_0} -
    (y_0', 0, z_0), \]
  which has an $(n,k)$-cylindrical singularity at $(0, y_0'',
  0,0)$. In order to use Proposition~\ref{prop:perturb1}, let us
  denote by $M_\tau$ the rescaled flow corresponding to
  $L^{\mathbf{a}_1}_t$, and $\widetilde{N}_\tau$ the rescaled flow
  corresponding to $\widetilde{L}^{\mathbf{a}_2}_t$ (both centered at
  the origin). We will show that the hypotheses of
  Proposition~\ref{prop:perturb1} hold for $\widetilde{N}_\tau$ as a
  graph over $M_\tau$. In
  applying Proposition~\ref{prop:perturb1} we will shift time by 2. So
  for instance conditions (ii), (iii)  will be verified at times
  $T_1+2$ and $T_1+3$ rather than at times $T_1, T_1+1$.

  For the parameter $\gamma$ in Proposition~\ref{prop:perturb1} we
  will choose
  \[ \gamma = T_1^{-1}e^{T_1/2}|y_0'| + |\mathbf{a}_2-\mathbf{a}_1| +
     e^{T_1/2}|z_0| + e^{T_1} |s_0|, \] 
    as in \cite{Sz26_2}. Let us consider the condition (i), which is
    where we need to use the barrier construction of
    Proposition~\ref{prop:barrier} rather than using the mean
    convexity as in \cite{Sz26_2}.

  First we compare $L^{\mathbf{a}_1}_t$ to $L^{\mathbf{a}_2}_t$, and
  we denote by $N_\tau$ the rescaled flow (centered at the origin)
  corresponding to $L^{\mathbf{a}_2}_t$. By the assumption, and
  Definition~\ref{defn:kappagood} we know that $N_{T_1+1}$ is the
  graph of a function $u$ over $M_{T_1+1}$, where
  \[ \label{eq:graph10} \left| u - \chi(|x| T_1^{-\beta/2})
      (\mathbf{a}_2-\mathbf{a}_1)\cdot y\right| &\leq \kappa  |\mathbf{a}_2-\mathbf{a}_1|.
  \]

  We will apply Proposition~\ref{prop:barrier} with $L^{\mathbf{a}_1}$
  being $S'_t$ (i.e. a small perturbation of our original flow), and
  $L^{\mathbf{a}_2}$ being $S''_t$. The time $t_0$ in the Proposition
  corresponds to $T_1+1$ along our rescaled flows here. Let us first
  consider the condition (2), on the open set $U_2$ (where the flow is
  mean convex). For small $s$, consider the flow
  $\widetilde{M}^s_\tau$ defined by
  \[\label{eq:Mtildedefn10} \widetilde{M}^s_\tau = (1 - e^{\tau-T_1-1}s)^{1/2} M_{\tau -
      \log(1 - e^{\tau-T_1-1}s)}, \]
    which is the rescaled flow corresponding to the translation of
    $L^{\mathbf{a}_1}_t$ in time by $e^{-T_1-1}s$. 
  We know from \cite[Proposition 13]{Sz26_2} that $M_\tau$ is
  $\tau^{-\beta}$-graphical over $\mathcal{C}_{n,k}$ on the
  $\tau^\beta$-ball. From this it follows that if $|s|$ is small, for $\tau\in [T_1,
  T_1+3]$, on $B_{\tau^\beta}$, the flow $\widetilde{M}^s_\tau$ is on
  the negative side of the graph of $-c_n |s|$ over $M_\tau$ if $s > 0$, and on the
  positive side of the graph of $c_n |s|$ if $s < 0$. Here $c_n > 0$
  is  a dimensional constant. 

  Note that we have
  \[ |\chi(|x| T_1^{-\beta/2})(\mathbf{a}_2 - \mathbf{a}_1) \cdot y |
    \leq 2T_1^{\beta/2} |\mathbf{a}_2 - \mathbf{a}_1| \leq
    2T_1^{\beta/2} \gamma,  \]
  and this function is supported inside the
  $T_1^\beta$-ball. Using \eqref{eq:graph10}, if we choose $\kappa$
  sufficiently small (depending on $T_1$) and we set $\overline{s} =
  c_n^{-1} 3T_1^{\beta/2}\gamma$, then it follows that on the $T_1^\beta$-ball,
  the hypersurface $N_{T_1+1}$ is contained between the two
  hypersurfaces $\widetilde{M}^{\pm \overline{s}}_{T_1+1}$.
  In Proposition~\ref{prop:barrier} we therefore set $\epsilon =
  \overline{s}$. Note that the set $U_2$ in condition (2) will be larger than the
  $T_1^{\beta}$-ball, but since the flow is mean convex on $U_2$ it
  follows from \eqref{eq:graph10} that if we choose $\kappa$
  sufficiently small (depending on the flow), the hypersurface
  $N_{T_1+1}$ will be contained between $\widetilde{M}^{\pm
    \overline{s}}_{T_1+1}$ on all of $U_2$.

  Regarding condition (1) from Proposition~\ref{prop:barrier}, outside
  of the set $U_1$ (this will be outside of the $T_1^\beta$-ball), by
  \eqref{eq:graph10}, the hypersurface $N_{T_1+1}$ is in the
  $\kappa\gamma$-neighborhood of $M_{T_1+1}$. After scaling back to
  the unrescaled flows, this implies that $L^{\mathbf{a}_2}_{t_0}$ is
  in the $e^{-(T_1+1)/2} \kappa\gamma$-neighborhood of
  $L^{\mathbf{a}_1}_{t_0}$ (outside of $U_1$). By choosing $\kappa$
  sufficiently small (depending on $T_1$), we can ensure that
  $e^{-(T_1+1)/2}\kappa\gamma < \sigma\epsilon$ in terms of the
  $\sigma$ in Proposition~\ref{prop:barrier}, and our choice of
  $\epsilon = \overline{s} = c_n^{-1}3T_1^{\beta/2}\gamma$.

  We can therefore arrange that the conditions (1), (2) both hold. The
  conclusion of Proposition~\ref{prop:barrier} is that at a later time
  $t$ we also have that on the set $U_1$, $L^{\mathbf{a}_2}_t$ lies
  between the time translated flows $L^{\mathbf{a}_1}_{t\pm
    \overline{s}}$. In terms of our rescaled flows $M_\tau, N_\tau$
  the set $U_1$ contains the $10\tau^{1/2}$-balls for $\tau\geq T_1$,
  and so on these balls we have that $N_\tau$ is contained between
  $\widetilde{M}^{\pm \overline{s}}_\tau$. From \eqref{eq:Mtildedefn10}
  we know that as long as $e^{\tau-T_1-1}\overline{s}$ is sufficiently
  small, $\widetilde{M}^{\pm \overline{s}}_\tau$ lie between the
  graphs of $\pm c_n^{-1} e^{\tau-T_1}\overline{s}$ over $M_\tau$,
  decreasing $c_n$ if needed. Therefore, for this range of $\tau$, the
  hypersurface $N_\tau$ is the graph of $v$ over $M_\tau$ on the
  $10\tau^{1/2}$-ball, where
  \[ |v| \leq c_n^{-1} e^{\tau-T_1} \overline{s} =
    3c_n^{-2}T_1^{\beta/2} e^{\tau-T_1}\gamma. \]
  We can now proceed as in \cite{Sz26_2}, comparing
  $\widetilde{N}_\tau$ to $N_\tau$. In conclusion we find that as long
  as $c_n^{-2} T_1 e^{\tau-T_1} \gamma$ is sufficiently small, on the
  $10\tau^{1/2}$-ball, $\widetilde{N}_\tau$ is the graph of
  $\tilde{v}$ over $M_\tau$, where
  \[ |\tilde{v}| \leq 2c_n^{-2} T_1 e^{\tau-T_1}\gamma, \]
  as required by condition (i) of Proposition~\ref{prop:perturb1}.

  The conditions (ii), (iii) follow exactly as in \cite{Sz26_2}, as
  long as we choose $\kappa$ sufficiently small. The conclusion of
  Proposition~\ref{prop:perturb1} is then that the translated flow
  $\widetilde{N}'_\tau = \widetilde{N}_\tau + e^{(\tau-T_1-2)/2} x_0$
  does not have an $(n,k)$-cylindrical singularity at infinity if
  $x_0$ is only nonzero in the degenerate $y$-directions, and $|x_0| <
  r_0 \gamma^{1-\kappa_1}$. By our assumption the flow
  $\widetilde{L}^{\mathbf{a}_2}_t$ (which corresponds to the rescaled
  flow $\widetilde{N}_\tau$) has an $(n,k)$-cylindrical singularity at
  $(0,y_0'', 0, 0)$, so $\widetilde{N}_\tau - e^{\tau/2}(0, y_0'', 0)$
  has an $(n,k)$-cylindrical singularity at infinity. The conclusion
  is that
  \[ e^{(T_1+2)/2} |y_0''| \geq r_0 \gamma^{1-\kappa_1} \geq r_0
    |\mathbf{a}_2 - \mathbf{a}_1|^{1-\kappa_1}. \]
  By choosing $r_2 = e^{-(T_1+2)/2}r_0$, it follows that $|y_0''| \geq
  r_2 |\mathbf{a}_2 - \mathbf{a}_1|^{1-\kappa_1}$, as required. 
\end{proof}

As a consequence we have the following result, whose proof is identical to \cite[Proposition 41]{Sz26_2}. 

\begin{prop}\label{prop:genericperturb}
    There exists $\underline{r}_0 > 0$ depending on the dimension, and entropy bounds, satisfying the following. Suppose that $L_t$ has an $\ell$-degenerate $\mathcal{C}_{n,k}$-singularity at the origin, which is $\ell$-degenerate at scale $r$ for some $r < \underline{r}_0$. There exists an $\epsilon = \epsilon(r)$, and for every $\kappa$-good family of perturbations $L^{\mathbf{a}}_t$ at $(0,0)$, there exists an $a_0 > 0$ (depending on $L^{\mathbf{a}}_t$), such that for almost every $\mathbf{a}\in \mathbb{R}^{k-\ell}$ with $|\mathbf{a}| < a_0$ the perturbed flow $L^{\mathbf{a}}_t$ has no $\ell$-degenerate $(n,k)$-cylindrical singularity in the $\epsilon$-neighborhood of the origin. 
\end{prop}

\section{Geometry of the flow after a singularity}
In this section we will prove some geometric properties of a mean
curvature flow passing through a non-degenerate cylindrical
singularity. Several important properties were obtained by
Sun-Wang-Xu~\cite{SWX25} in the course of proving their
Theorem~\ref{thm:nondegsurgery} on surgery. Recall that by \cite{SWX25} we know that the flow is mean convex and non-collapsing in a neighborhood of a non-degenerate cylindrical singularity. As a consequence estimating the full curvature is equivalent to estimating the mean curvature, by Haslhofer-Kleiner~\cite{HK16}, as follows. 
\begin{prop}\label{prop:HK16}
There is a parabolic neighborhood of any non-degenerate singularity, and a constant $C > 0$ (depending on the flow), such that on this neighborhood we have $|A| \leq C|H|$ and $|\nabla A| \leq C|H|^2$. 
\end{prop}

\subsection{Asymptotic expansion up to the singularity}
We first recall the result of \cite[Theorem 1.1 (vi)]{SWX25} about the profile of the
hypersurface at the singular time. 
\begin{theorem}\label{thm:SWXMain}
  At time $t=0$, in a neighborhood of the origin, $N_0$ is given by the
  graph of $u(\theta, y)$ over $S^{n-k}(\sqrt{2(n-k)}) \times
  \mathbb{R}^k$, with 
  \[ u(\theta, y) = \frac{\sqrt{2(n-k)}|y|}{2 |\log |y||^{1/2}} (1 +
    o(1)) - \sqrt{2(n-k)}, \]
  where $\Vert o(1)\Vert_{C^0} \to 0$ as $y\to 0$.
\end{theorem}
We will use $r=|y|$ for the distance along the axis of the cylinder. 
In other words along $N_0$, near the origin, we can write
\[ |x| = \rho(r) ( 1+ o(1)), \qquad \rho(r) = \frac{ \sqrt{2(n-k)}
    r}{2 |\log r|^{1/2}}. \]
The clearing-out lemma from Colding-Minicozzi~\cite[Theorem 6.1]{CM15}
(see \cite[Lemma 4.1]{SWX25}) says that $|x| + \sqrt{t} = o(r)$ for $t
> 0$ in a neighborhood of $(0,0)$. In addition the various possible
blowups were classified in \cite[Theorem 4.2]{SWX25}: if $(x_j, t_j)
\to (0,0)$ with $t_j \geq 0$, then there are scales $\lambda_j \to 0$
such that the rescalings
\[ s\mapsto \lambda_j^{-1}(N_{t_j + \lambda_j^2s} - x_j) \]
converge either to a cylinder $S^{n-k}\times\mathbb{R}^k$, or a bowl
translator times $\mathbb{R}^{k-1}$.

We will repeatedly use the
pseudolocality result, \cite[Theorem 2.5]{SX22}, which states the
following: for any $\tau_0 \in (-1,0)$ and $\delta_0, R_0 > 0$, 
if a mean curvature flow $L_t$ is $\epsilon$-close to
$S^{n-k}(\sqrt{2(n-k)}) \times\mathbb{R}^k$ at $t=-1$ on
$B_{\epsilon^{-1}}$, and $\epsilon$ is sufficiently small, depending
on $\tau_0, \delta_0, R_0$ and a bound for the area ratios, then $L_t$
is a $\delta_0$-graph over $\sqrt{-t} S^{n-k}(\sqrt{2(n-k)})
\times\mathbb{R}^k$ on $B_{R_0}$ for $t\in [-1,\tau_0)$.

Let us record the behavior of the rescaled mean curvature flow
$M_\tau = e^{\tau/2} N_{-e^{-\tau}}$, which is a slight amplification of \cite[Theorem 1.2, Theorem 1.3]{SX22}.

\begin{theorem}[\cite{SX22}]\label{C^2_crude_lemma}
Let $N_t$ and $M_{\tau}$ be as above, and denote by $u$ the profile of $M_{\tau}$. Denoting
\[
u_0(y,\theta,\tau)=\sqrt{2(n-k)}\left(\sqrt{1+\frac{\sum_{i=1}^{k}(y_i^2-2)}{2\tau}}-1\right).
\]
For every $K<\infty$ there exists $\tau_0<\infty$ and a constant $C<\infty$ such that, setting $w=u-u_0$ we have
\begin{equation}\label{SX1_est}
|D^jw|(x,\tau) \leq \frac{C}{\tau^{3/4}}, \qquad j=0,1, \;\tau \geq \tau_0,\; |x|\leq K\sqrt{\tau},
\end{equation}
and
\begin{equation}\label{SX2_est}
|D^jw|(x,\tau) \leq \frac{C}{\tau^{3/2}}\exp(|x|^2/6), \qquad j=0,1,2, \;\tau \geq \tau_0,\; |x|\leq K\sqrt{\tau}.
\end{equation}
Note that \eqref{SX2_est} gives better decay in $\tau$ at the cost of the exponential factor in $|x|$.
\end{theorem}
\begin{remark}
In particular, for any $K > 0$, on the
expanding balls $B_{K\tau^{1/2}}(0)$ the transverse radius of the flow
is
\[ \label{eq:rescale40} |x| = \sqrt{2(n-k)} \left(1 + \frac{|y|^2 - 2k}{2\tau}\right)^{1/2}
  + O(\tau^{-3/4}). \]
The error is valid in $C^1$.     
\end{remark}
\begin{proof}[Proof of theorem \ref{C^2_crude_lemma}]
Estimate \eqref{SX1_est} is just \cite[Theorem 1.3]{SX22}. Estimate \eqref{SX2_est} is resulted from combining Theorem \cite[Theorem 1.2]{SX22} and standard parabolic estimates, as we now show.
Considering the unrescaled profile and model profile function 
\[
U(x,\theta,t)=\sqrt{-t}u\left(\frac{x}{\sqrt{-t}},\theta,-\log(-t)\right),\; U_0(x,\theta,t)=\sqrt{-t}u_0\left(\frac{x}{\sqrt{-t}},\theta,-\log(-t)\right), 
\]
then letting $\mathcal{M}$ be the MCF operator in the graph over $\mathbb{R}^{k}\times S^{n-k}$, a direct calculation shows that in the region $|x|\leq K\sqrt{-t}|\log(t)|$ we have
\begin{equation}\label{non_lin_U_0}
\mathcal{M}[U_0]=\frac{k\sqrt{2(n-k)}}{2\sqrt{-t}\log^2(-t)}+O_1\left(\frac{1+\frac{|x|^4}{t^2}}{\sqrt{-t}|\log^3(-t)|}\right),
\end{equation}
where $O_1$ describes the behavior including first derivatives. Now, take some point $(x_0,t_0)$ and rescale the mean curvature flow, as well as the model profile $U_0$ to unit size around that point. Namely, taking $\lambda=\sqrt{2(n-k)}\sqrt{t_0}+U_0(x_0,0,t_0)$, 
\begin{equation}
V(x,\theta,t)= \lambda^{-1}U(x_0+\lambda x,\theta,t_0+\lambda^2 t), \qquad V_0(x,\theta,t)= \lambda^{-1}U_0(x_0+\lambda x,\theta,t_0+\lambda^2 t),
\end{equation} 
the rescaling of \eqref{non_lin_U_0} implies
\begin{equation}\label{non_lin_V_0}
\mathcal{M}[V_0]=\lambda \frac{k\sqrt{2(n-k)}}{2\sqrt{-t_0-\lambda^2t}\log^2(-t_0-\lambda^2t)}+\lambda O_1\left(\frac{1+\frac{|x_0+\lambda x|^4}{(t_0+\lambda^2 t)^2}}{\sqrt{-t_0-\lambda^2 t}|\log^3(-t_0-\lambda^2 t)|}\right).
\end{equation}
It follows from smooth convergence to cylinders and the fundamental theorem of calculus that the difference  $\tilde{W}=V-V^0$ satisfies a linear parabolic  equation
\[
\tilde{W}_t-a^{ij}\partial_i\partial j \tilde{W}-b^i\partial_i \tilde{W}-c\tilde{W}=-\mathcal{M}[V_0]
\]
with uniformly bounded ellipticity and coefficient. Therefore, applying standard parabolic estimates, we get that for $j=0,1,2$
\begin{equation}\label{int_est}
|D^j\tilde{W}|(0,0) \leq C(\varepsilon)\left(\|\tilde{W}\|_{L^2(B(0,\varepsilon)\times (-\varepsilon,0))} +\frac{1}{\log^2(t_0)}\right).
\end{equation}
On the other hand, using \cite[Theorem 1.2]{SX22}, for every $|x|\leq K\sqrt{\tau}$ we have
\begin{equation}\label{w_L2_est}
\|w(\cdot,\tau)\|_{L^2(B(x,1))} \leq C\tau^{-3/2}\exp(|x|^2/7). 
\end{equation}
As ball sizes at rescaled times (for $u$ and $u_0$) and the times for ($V$ and $V_0$) are comparable we can plug \eqref{w_L2_est} into \eqref{int_est} to obtain a Hessian bound for $\tilde{W}$, which translated back into $w$ gives \eqref{SX2_est}.    
\end{proof}

\subsection{Curvature and position bounds after the singularity}

\begin{lemma}\label{very_cyl_away_lemma}
    For every $\varepsilon>0$ and every $t>0$ sufficiently small, all points in $N_t\cap (B(0,\varepsilon))$ with $|x|\geq (\sqrt{2}+\varepsilon)\sqrt{t|\log t|}$ are $\varepsilon$ cylindrical.
\end{lemma}

\begin{proof}
Observe that the precise asymptotics \eqref{SX1_est} imply that for every $\varepsilon>0$, taking time $t_1>0$ sufficiently close to zero and considering
\begin{equation}
\tilde{N}_s=\frac{1}{\sqrt{t_1}}N_{t_1s},
\end{equation}
we have that $\tilde{N}_s\cap \left(B(0,(\sqrt{2}+2\varepsilon)\sqrt{|\log t_1|})-B(0,(\sqrt{2}+\varepsilon)\sqrt{|\log t_1|})\right)$ evolve as a cylinder for the entire interval $s\in [-1,1]$. Indeed,  the classification of blowups \cite[Theorem 4.2]{SWX25} implies that after recentering at a point in this annulus, and taking the limit as $t_1\rightarrow 0$  the flow $\tilde{N}_s$ converges to a cylinder of initial radius at least $\sqrt{(4+\varepsilon/2)(n-k)}$, that will evolve for a time interval strictly containing $[-1,1]$). Taking $t_1$ sufficiently small and considering time $s=1$ yields the result.   
\end{proof}

We now prove the key curvature estimates for the flow after the non-degenerate singularity. We start with defining the `tip', and then estimate its position, from which we can prove our first estimate on the mean curvature.
\begin{defn}
   Suppose the flow $N_t$ has a non-degenerate singularity at the origin modelled on the  cylinder $\mathbb{S}^{n-k}\times \mathbb{R}^k$. Let $\pi_k:\mathbb{R}^{n+1}\to \mathbb{R}^k$ be the projection onto the $\mathbb{R}^k$ factor. For any direction $v\in \mathbb{S}^{k-1}\subset \mathbb{R}^k$, we define
   \begin{align*}
      p_t(v):=\operatorname{Argmin}_{\left\{x\in N_t, \frac{\pi_k(x)}{|\pi_k(x)|}=v\right\}}\left( \left\lvert{x}\right\rvert \right)
   \end{align*}
\end{defn}
\begin{prop}\label{positionest}
    Let $\mathcal{N}$ be a closed mean curvature (unit-regular cyclic mod 2 Brakke) flow with an isolated non-degenerate singularity at the origin, modelled on the cylinder $\mathbb{S}^{n-k}\times \mathbb{R}^k$. Then, there exists an $r_0>0$ and $c,C \in(0,\infty), c<C$,  such that in the parabolic cylinder $P(0,r)=B(0,r)\times(0,r^2)$, with $r\in(0,r_0)$ we have
    
    \begin{align*}
        c\sqrt{-t\log t}\leq |p_t(v)|\leq C\sqrt{-t\log t}
    \end{align*}
\end{prop}

Before giving the proof we recall the definition of parabolic dilation, and of being $(\delta,(n,k),R)$-cylindrical.
\begin{defn}
    Let $\mathcal{N}=\{N_t\}_{t\in I}$, be a Mean curvature flow (or Brakke flow) on some interval $I\subseteq \mathbb{R}$. For $\lambda>0$ and a point $(x_0,t_0)\in\mathbb{R}^{n+1}\times \mathbb{R}$, we define the parabolic dilation as
    \begin{align*}
        D_\lambda(\mathcal{N}):= \{\lambda (N_{\lambda^{-2} \tilde{t}}-x_0)\}_{\tilde{t}\in \tilde{I}}
    \end{align*}
    where $\tilde{t} = \lambda^2(t-t_0)$, and $\tilde{I}=\lambda^2(I-t_0)$.
\end{defn}
\begin{defn}
   For $\delta>0,R<\infty$ and $k,n$ positive integers, $k<n$, we say a mean curvature flow $\mathcal{N}^n\subset \mathbb{R}^{n+1}\times\mathbb{R}$ is $(\delta,(n,k),R)$-cylindrical (at the scale $\lambda\in(0,\infty$)), if the flow $\mathcal{D}_\lambda(\mathcal{N})$ (after suitable translation) is $\delta$-close to the round $(n,k)$-cylinder terminating at time 1, in $C^{\lfloor 1/\delta \rfloor}(B(0,R)\times[0,1-\delta])$).
\end{defn}
\begin{proof}
   
     We show that there exist $c,C$ such that
     \begin{align*}
        c\sqrt{-t\log t}\leq p_t(v)\leq C\sqrt{-t\log t}
    \end{align*}
    where $v\in \mathbb{S}^{k-1}$ is a unit vector in the `flat direction' of the generalised cylinder $\mathbb{S}^{n-k}\times \mathbb{R}^{k}$. 
   
    Recalling the Sun--Wang--Xue asymptotic, Theorem \ref{thm:SWXMain}, at time $t=0$ the radius at distance $|x|$ has the asymptotic form
      \begin{align*}
         r(x)=\frac{\sqrt{2(n-k)}|x|}{2\sqrt{-\log(|x|)}}(1+o_x(1))
      \end{align*}
     Fix $v\in S^{k-1}$ and let $x_i$ be a sequence of points $x_i\to 0, \frac{x_i}{|x_i|}=v$. Consider the dilation of the flow about $x_i$:
     \begin{align*}
        \mathcal{N}_i := \mathcal{D}_{\lambda_i}(\mathcal{N}-(x_i,0)), \ \lambda_i = \frac{2\sqrt{-\log(|x_i|)}}{|x_i|}.
     \end{align*}
     
     Taking a (sub-sequential) Brakke limit as $i\to\infty$ and applying the results of Bamler--Lai yields an ancient asymptotically cylindrical flow $\mathcal{N}_\infty$. In fact, from Sun--Wang--Xue, \cite[Remark 2.11, Theorem 4.2]{SWX25}, the limit is either the $k$-cylinder or the $k$-bowl times a plane translating in the direction  $v$.  It is easy to see that the limit is the $(n,k)$-cylinder from the asymptotic of Sun--Wang--Xue, Theorem \ref{thm:SWXMain}: the $t=0$ time-slice of the limit $\mathcal{N}_\infty$ must extend to both $\pm\infty$ in the $v$ direction, as the distance to the singularity goes to infinity along the sequence, and the remaining $k-1$ directions follow immediately from the asymptotic. Alternatively, if we saw the $k$-bowl cross a plane, the Brakke convergence could be promoted to smooth convergence, in which case we would find a `tip' at some finite distance from the origin in each of the $\mathcal{N}_i$ at time 0, violating our assumption of a singularity. 
     Now, by the normalisation $\sqrt{2(n-k)}$, we know that the limit flow is extinct by time $t=1$. Thus, for any $R>0$ and $\delta>0$, the convergence is smooth on $B(0,R)\times\left[0,1-\delta\right]$ and we have Brakke convergence on $B(0,R)\times\left[0,1+\delta\right]$. Translated back to the un-dilated flow: for sufficiently large $i$, the flow is $(\delta,k,R)$-cylindrical at scale $\lambda_i$. Thus, the flow evolves like a cylinder in some spatial ball 
     \begin{align*}
     B(x_i,R\frac{|x_i|}{2\sqrt{-\log(|x_i|)}})
     \end{align*} on the time interval 
     \begin{align*}
        \left[0,\frac{|x_i|^2}{-4(\log|x_i|)}(1-\delta)\right],
    \end{align*} 
     in particular, the flow must vacate this ball for some time 
    \begin{align*}
    t_i\in \left[\frac{|x_i|^2}{-4(\log|x_i|)}(1-\delta),\frac{|x_i|^2}{-4(\log|x_i|)}(1+\delta)\right]
    .\end{align*}
    
      It stands to reason that, since the singularity is isolated and the flow is monotonic (in a set theoretic sense), the only way the flow can clear-out from this neighborhood for large $i$, is for a `tip' to pass through this neighborhood of $x_i$ during the time interval $\left[\frac{|x_i|^2}{4(-\log|x_i|)}(1-\delta),\frac{|x_i|^2}{4(-\log|x_i|)}(1+\delta)\right]$.

     Setting $t_i^-:=\frac{|x_i|^2}{4|\log|x_i||}(1-\delta)$ and $t_i^+:=\frac{|x_i|^2}{4|\log|x_i||}(1+\delta)$, we can calculate the values of $c,C$ by considering $\sqrt{-t_i^\pm\log(t_i^\pm)}$. 

     From the above discussion, at time $t_i^-$ we known that the `tip', $P_{t_i^-}(v)$, must satisfy 
     \begin{align*}
        P_{t_i^-}(v)\leq \left\lvert{x_i}\right\rvert-R\lambda_i^{-1}
     \end{align*}
     Note, $\lambda_i^{-1}=(1-\delta)^{-1/2}\sqrt{t_i}$. Next, we observe:
     \begin{align}\label{upperbound1}
        \sqrt{-t_i^-\log(t_i^-)}&=\frac{(1-\delta)^{1/2}}{2}\left\lvert{x_i}\right\rvert\sqrt{2+\frac{\log(1-\delta)-\log(-4\log(\left\lvert{x_i}\right\rvert))}{\log(\left\lvert{x_i}\right\rvert)}}\\&\geq \frac1{C_I}\left\lvert{x_i}\right\rvert
     \end{align}
     where we write $C_I$ to highlight that the constant depends on the choice of sufficiently large $i$.
     We now deduce:
     \begin{align}\label{upperbound2}
        P_{t_i^-}(v)\leq C_I\sqrt{-t_i^-\log(t_i^-)}-R(1-\delta)^{-1/2}\sqrt{t_i^{-}}
     \end{align}
     Again, for sufficiently large $i$, we may absorb the second term into the first by increasing the constant. This yields the choice of $C$. A similar argument can be used to show the existence of $c$, using $t_i^+$.
     The sequence $x_i$ and direction $v$ was chosen arbitrarily, it follows (invoking a diagonalisation argument if necessary) that we may choose $r_0>0$ sufficiently small that the claimed tip position estimates hold.
\end{proof}

\begin{remark}\label{rmk:SharpcC}
   Examining equations \eqref{upperbound1} and \eqref{upperbound2}, we see  $C=(1+o(r_0))\sqrt{2}$, as taking $r_0$ smaller allows us to make the errors smaller. Similarly, $c=(1-o(r_0))\sqrt{2}$.
\end{remark}

Even though a non-degenerate cylindrical singularity is manifestly not asymptotically conical, the precise asymptotics of \cite{SWX25} show that at the singular time, it does open up conically at comparable scales. Because of this, it stands to reason that the flow behaves like a self expander next to its tips. This is the content of the next theorem.

\begin{theorem}[Expander behavior after the singularity]\label{quo_after21}
For every $K<\infty$ and $\varepsilon>0$ there exist some physical time $t_0>0$ such that for every $t\in (0,t_0)$
\begin{equation}
\left|2+\frac{\langle x,\nu \rangle}{tH}\right|<\varepsilon ,\qquad x\in N_{t}\cap B(0,K\sqrt{t|\log(t)|}).
\end{equation}
\end{theorem}

We defer the proof of this result to Theorem \ref{quo_after2} in Section \ref{near_tip_exp_sec}.

\begin{cor}\label{tipcurvest}
   Let $r_0$ be as in Proposition \ref{positionest}. There exist $c,C\in(0,\infty)$ such that, for all $v\in\mathbb{S}^{k-1}$, we have
   \begin{align*}
        c\sqrt{\frac{-\log t}{t}}\leq|H|(p_t(v),t)\leq C\sqrt{\frac{-\log t}{t}}, \qquad t\in(0,r_0)
    \end{align*}
\end{cor}

\begin{proof}
   For $v\in \mathbb{S}^{k-1}$, note that $\langle p_t(v),\nu(p_t(v)) \rangle=\|p_t(v)\|$. By Theorem \ref{quo_after2} with $\varepsilon=1$ and $K$ sufficiently large, this gives
   \[
   \frac{\|p_t(v)\|}{3t}\leq H(p_t(v),t) \leq \frac{\|p_t(v)\|}{t}
   \]
   Using the tip position estimate of Proposition \ref{positionest} the result follows.
\end{proof}

\begin{remark}
   As before, we can take  $C=(1+o(r_0))\frac1{\sqrt{2}}$, as taking $r_0$ smaller allows us to make the errors smaller. Similarly, we can take $c=(1-o(r_0))\frac1{\sqrt{2}}$. 
\end{remark}

\begin{cor}\label{initialcurvatureest}
   Under the same assumptions as Proposition \ref{positionest}, there exists an $r_1>0$ such that, for $r\in (0,r_1)$, in the parabolic cylinder $P(0,r)$ we have
   \begin{align*}
        |H|(\cdot,t)\leq 2C\sqrt{\frac{-\log t}{t}}
    \end{align*}
    where $C$ is the constant from Corollary \ref{tipcurvest}.
\end{cor}
\begin{proof}
   The corollary is the same as asserting that the maximum of the curvature at time $t$ occurs at (or, more accurately, near) one of the tip points $p_t(v)$.
   Thus, we show that there exists an $r_1$ such that there cannot be points in $B(0,r_1)$ with mean curvature larger than this `tip'.

   Suppose for contradiction such points exist, that is, for a sequence of times $t_i\to0$, we find points $x_i$ such that $H(x_i,t_i)>2C\sqrt{\frac{-\log t_i}{t_i}}$, where $C$ is the constant from Corollary \ref{tipcurvest}. Without loss of generality, we assume that $(x_i,t_i)$ is the maximum of the mean curvature at time $t_i$ (at least in the neighborhood of the singularity we are considering), and there is a fixed $v\in \mathbb{S}^{n-k-1}$ such that $\frac{\pi(x_i)}{|\pi(x_i)|}=v$ (this can be achieved by taking a subsequence, as $\mathbb{S}^{n-k-1}$ is compact). Clearly, these points must accumulate at the space-time origin, as if they accumulated elsewhere, we would violate the isolated nature of the singularity.
   We now blow-up along this sequence by the mean curvature. By the canonical neighborhood theorem of Bamler--Lai, \cite{BL25}, and \cite[Remark 2.11, Theorem 4.2]{SWX25}, the (sub-sequential) limit is either the $(n,k)$-cylinder or the $k$-bowl cross a plane, translating in the direction $v$.
   
   If we saw the bowl, we would have an immediate contradiction: since we are blowing up at the maximum of the mean curvature, the limit would have the tip of the bowl at the space-time origin. However, this is the same limit we see after blowing up at the point $P_{v}(t_i)$ by the maximum of the mean curvature at time $t_i$. Recall, $P_{v}(t_i)$ is characterised as being (in a suitable sense) the closest point to the origin in the direction $v$ and $(x_i,t_i)$ as the maximum of the mean curvature. Both of these points are in some sense the `tip' in the direction $v$ at time $t_i$. Since there cannot be two tips (by the set monotonicity of mean convex MCF), and we know the mean curvature of the point $P_{v}(t_i)$, we must in fact have that the difference of mean curvatures $|H(x_i,t_i)-H(P_{v}(t_i))|$ converges to 0. Thus, for sufficiently large $i$, the ratio of mean curvatures $H(x_i,t_i)/H(P_{v_1}(t_i),t_i)$ cannot be larger than 2, a contradiction. (In fact, for any $c\in(1,2)$, we can find an $I$ such that for $i>I$, $H(x_i,t_i)/H(P_{v_1}(t_i),t_i)<c$).
   
   Hence, the limit must be the cylinder. This also leads to contradiction: If we see the cylinder, then the argument presented in Proposition \ref{positionest} means the flow must vanish from a neighborhood of $x_i$ far before the tip passes through, as the scale of this cylinder is much smaller than the scale of the tip. This again contradicts the set monotonicity of the mean convex MCF.
\end{proof}

Using our knowledge of the geometry after the non-degenerate singularity, we can now improve the initial curvature estimate to get more refined estimate involving the distance from the tip.

\begin{prop}\label{curvdecay}
   Let $K>0$, and suppose $ B(0,K\sqrt{-t\log t})\cap N_t\neq \emptyset$. Then, there exists $r,C_1, C_2$ such that, for $t\in[0,r^2]$ and every $x\in B(0,K\sqrt{-t\log t})\cap N_t$, we have
   \begin{align*}
      H(x,t)\leq C_1\sqrt{\frac{-\log t}{t}} \left({1+C_2\sqrt{\frac{-\log t}{t}}\left|{\left\langle {x-p_t(v)},{v}\right\rangle}\right|}\right)^{-1/2} 
   \end{align*}
   where $v=\frac{\pi(x)}{|\pi(x)|}$.
\end{prop}

\begin{proof}
Note that, in light of Corollary \ref{initialcurvatureest} and \cite[Theorem 4.2]{SWX25}, given $C<\infty$,  it suffices to prove the estimate for $|x-p_t(x)|\geq C\sqrt{|\log(t)|t^{-1}}$, as close to the tip, in the tip curvature scale, the curvature remains comparable to the tip curvature. Given any $x$ and denoting $d=\|x-p_t(v)\|$ we can therefore assume that
\[
\frac{d}{2} \leq |\langle {x-p_t(v)},{v}\rangle|\leq 2d,
\]
and prove the estimate with $d$ instead of $|\langle {x-p_t(v)},{v}\rangle |$.

Denote by $t_1$ the time at which the tip arrives at $x$. More precisely $t_1$ is such that
\[
\pi_k(p_{t_1}(v))=\pi_k(x)
\]
By Theorem \ref{tipcurvest} , we know that
\begin{equation}
H(p_{t_1}(v),t_1) \leq C_1\sqrt{\frac{|\log(t_1)|}{t_1}}, \qquad  \frac{1}{C_1}d\sqrt{\frac{t}{|\log(t)|}} \leq t_1-t.
\end{equation}
where $C_1$ depends on $K$. The second estimate is just a motion bound, in which we have used the fact that $t_1$ is comparable to $t$ (by the tip position bound), hence the curvatures are also comparable. 

\bigskip

Considering the flow $\tilde{N}_s=\lambda(N_{\lambda^{-2}s+t_1}-p_{t_1}(v))$ where $\lambda=H(p_{t_1}(v),t_1)$, this flow is $\varepsilon$ cylindrical for all time $s\in [\lambda^2(t-t_1),-\Lambda]$, where $\Lambda=\Lambda(n-k,\varepsilon)$ is a constant computed on the bowl soliton. Therefore taking $C$ large enough in the initial paragraph of the proof, denoting 
\[
(\tilde{x},\tilde{t})=(\lambda (x-p_{t_1}(v)),\lambda^2(t-t_1)),
\]
we have
\begin{equation}
    H_{\tilde{N}}(\tilde{x},\tilde{t}) \leq C_2\lambda^{-1}|t-t_1|^{-1/2} \leq C_3\lambda^{-1} d^{-1/2}\left(\frac{t}{|\log(t)|}\right)^{-1/4}.
\end{equation}
So 
\[
H_{N_t}(x,t) \leq C_3d^{-1/2}\left(\frac{t}{|\log(t)|}\right)^{-1/4}.
\]
This completes the proof.

\end{proof}

We now record the following lemma, which says if we are sufficiently far away from the origin at the length scale $\sqrt{-t\log(t)}$, then we are actually in the cylindrical regime. 
\begin{lemma}\label{cylregioncurvdecay}
   For every small $\delta,\epsilon>0$, there exists $r>0$ with the following significance:
   
    For $(x,t)\in N_t\cap P(0,r)$, with $|x|>R(t):= \sqrt{\frac{2}{(1-\epsilon)(1-\delta)}}\sqrt{-t\log(t)}$, we have
   \begin{align*}
      H(x,t)\leq \sqrt{2}(1+\delta)\sqrt{n-k}\frac{\sqrt{-\log |x|}}{|x|}\left[\sqrt{\delta}\right]^{-1} 
   \end{align*}
\end{lemma}
\begin{proof}
Recall, from the argument in the proof of Proposition \ref{positionest}, for every $\delta>0$, we can find $r,R>0$ such that, for $x\in B(0,r)$, the flow is $(\delta, k,R)$-cylindrical. From which, we can conclude
     for $ y\in B(x,R\frac{|x|}{\sqrt{-4\log(|x|)}})$ and $ t\in \left(0,\frac{|x|^2}{-4(\log|x|)}(1-\delta)\right]$, we have
     \begin{equation}\label{eqn:cylindricalest}
      H(y,t)\leq \sqrt{2}(1+\delta)\frac{\sqrt{-\log(|x|)}}{|x|}\left(1-\frac{-4(\log|x|)}{|x|^2}t\right)^{-1/2}
     \end{equation}

     For every $\epsilon$, by potentially reducing $r$, we can assume $r$ is such that $x\in B(0,r)$ implies $\left\lvert{\frac{\log(-4\log |x|)}{2\log|x|}}\right\rvert\leq\epsilon$ and $r^2<1/e$.

     Now, fix $t$ and a point $x$ such that $|x|\geq \sqrt{\frac{2}{(1-\epsilon)(1-\delta)}}\sqrt{-t\log(t)}$. Then, 
     \[
     \begin{aligned}
      \sqrt{-\frac{(1-\delta)|x|^2}{-4\log|x|}\log\left(\frac{|x|^2}{-4\log|x|}\right)}&=|x|\sqrt{1-\delta}\sqrt{\frac{2\log(|x|)-\log(-4\log(|x|))}{4\log|x|}}\\
      &=\frac{|x|\sqrt{1-\delta}}{\sqrt{2}}\sqrt{1-\frac{\log(-4\log(|x|))}{2\log|x|}}\\
      &\geq \sqrt{\frac{(1-\delta)(1-\epsilon)}{2}}|x| \\
      &\geq \sqrt{-t\log(t)}
     \end{aligned} \]
     which implies 
     \begin{align*}
      -(1-\delta)a\log(a)\geq -b\log(b)
     \end{align*}
     and therefore
     \begin{align*}
      -(1-\delta)a\log((1-\delta)a)\geq -b\log(b)
     \end{align*}
     where $a=\frac{|x|^2}{-4\log|x|}$, $b=t$. By the monotonicity of $-y\log(y)$, for $y<1/e$, we deduce $t\leq(1-\delta)\frac{|x|^2}{-4\log|x|}$. 
     In particular, this means that the point $(x,t)$ is in the region in which \eqref{eqn:cylindricalest} holds.
     The lemma follows from the cylindrical curvature estimate.
\end{proof}
\begin{lemma}\label{curvcomp} There exists an $r>0$ such that, for points $(x,t)\in P(0,r)$ such that $c\sqrt{t|\log t|}\leq|x|\leq C\sqrt{t|\log t|}$ we have

    \begin{align*}
        \frac{|\log|x||}{|x|}\asymp \sqrt{\frac{|\log t|}{t}}
    \end{align*}
\end{lemma}
\begin{proof} For sufficiently small $r$, we know for points $x,y\in B(0,r), |y|<|x|$ $|\log|x||<|\log|y||$.

    We consider points $c\sqrt{t|\log t|}\leq|x|\leq C\sqrt{t|\log t|}$. We have
    \begin{align*}
        \left|\log C +\frac{\log t}{2}+\frac{\log|\log t|}{2}\right|\leq|\log|x||\leq\left|\log c +\frac{\log t}{2}+\frac{\log|\log t|}{2}\right|
    \end{align*}
    Dividing by $|x|$, we conclude the claim by shrinking $r$ such that we bound the additional terms by a constant times $\sqrt{\frac{|\log t|}{t}}$.
\end{proof}
\begin{prop}\label{prop:curvxest}
    There exists an $r, C>0$ such that in $B(0,r)$, we have
    \begin{align*}
        H(x,t)\leq  C\frac{|\log|x||}{|x|}
    \end{align*}
\end{prop}
\begin{proof}
    This follows immediately from our tip position estimate Proposition \ref{positionest}, Lemma \ref{cylregioncurvdecay} and Lemma \ref{curvcomp}.
\end{proof}

\subsection{Almost conical behavior at the tip location}

Let $r_0, \lambda > 0$ and $y_0 = r_0 \mathbf{e}\in \mathbb{R}^k$ be a vector along
the $\mathbb{R}^k$ factor of the
cylinder, where $|\mathbf{e}|=1$. Consider the scaled and translated flow
\[ \tilde{N}_s = \lambda^{-1} (N_{\lambda^2 s} - y_0). \]
Define $T$ by $\lambda = e^{-T/2}$. Then the
slice $\tilde{N}_{-1}$ agrees with the translation $M_T - e^{T/2}y_0$ in terms of the
rescaled flow $M_T$
corresponding to $N_t$. By the above
discussion, near the origin $\tilde{N}_0$ is close to a cylinder of
radius $\lambda^{-1}\rho(r_0)$. To ensure that this is roughly of a
fixed size, we want to choose $\lambda \sim \rho(r_0)$. Motivated by
this, it is convenient to let
\[ r_0^2 = 2\lambda^2 |\log \lambda| =
  T e^{-T}. \]
Note that with this choice we have
\[ \frac{\rho(r_0)^2}{\lambda^2} = \frac{(n-k) r_0^2}{2|\log r_0|
    \lambda^2} \to n-k,\]
as $T\to \infty$ (or equivalently as $\lambda\to 0$). Therefore, for
large $T$, $\tilde{N}_0$ is close to a cylinder of radius
$\sqrt{n-k}$. Note that a cylinder of radius exactly $\sqrt{n-k}$
would disappear in time $1/2$. The following result says that we have
good control of the flow $\tilde{N}_s$ for $s\in [-1, s_0]$, once $T$
is large enough, for any given $s_0 < 1/2$. 

\begin{lemma} \label{lem:linearterm10} 
  Fix $s_0 < 1/2$. For any $R > 0$, if $T$ is sufficiently large,
  depending on $s_0, R$, then the translated
  and rescaled flow $\tilde{N}_s$ can be parametrized by
  \[ \label{eq:p21} (\theta, y) \mapsto (F_T(\theta, y, s) \theta, y), \]
  for $s\in [-1,s_0]$ on the ball $B_R$. Here $\theta\in S^{n-k}$ and
  $y\in \mathbb{R}^k$, and
  \[\label{eq:p22} F_T(\theta, y, s) = a(s) + \frac{n-k}{a(s) T^{1/2}}
    \mathbf{e}\cdot y+
    \mathcal{E}_T(\theta, y, s). \]
  For $s\in [-1/2, s_0]$ we have $\Vert \mathcal{E}_T\Vert_{C^4} = o(T^{-1/2})$, and $a(s) =
  \sqrt{(n-k)(1-2s)}$. 
\end{lemma}
\begin{proof}
   Recall that $\tilde{N}_{-1} = M_T - e^{T/2}y_0$, and $y_0 = r_0
   \mathbf{e}$. Since $r_0 = T^{1/2} e^{-T/2}$, this implies that
   $\tilde{N}_{-1} = M_T - T^{1/2}\mathbf{e}$. From
   \eqref{eq:rescale40} we have that on a ball of radius $T^{1/2}$,
   the hypersurface $\tilde{N}_{-1}$ can be parametrized as in
   \eqref{eq:p21}, with
   \[ F_{T}(\theta, y, -1) &= \sqrt{2(n-k)} \left( 1 + \frac{ |y+
         T^{1/2}\mathbf{e}|^2 - 2k}{2T}\right)^{1/2} + O(T^{-3/4}) \\
     &= \sqrt{2(n-k)} \left( \frac{3}{2} + \frac{y \cdot \mathbf{e}}{T^{1/2}}
       + \frac{|y|^2 - 2k}{2T}\right)^{1/2} + O(T^{-3/4}). 
   \]
   From now we will work on a ball of radius $R_T = CT^{1/8}$, with $C > 2$,
   but possibly shrinking $C$ when we pass to interior estimates. For
   $|y| < R_T$, we have
   \[ F_T(\theta, y, -1) = \sqrt{3(n-k)} + \frac{
       n-k}{\sqrt{3(n-k)T}} y\cdot \mathbf{e} + O(T^{-3/4}), \]
   where the error holds in $C^1$. Note here that $|y|^2 T^{-1} <
   CT^{-3/4}$, so the additional errors from the expansion can all be
   absorbed. Since $a(-1) = \sqrt{3(n-k)}$, this shows the required
   expansion when $s=-1$.

   To obtain estimates for $s\in [-1,s_0]$, note first that if we let
   $L_s = S^{n-k}(a(s)) \times \mathbb{R}^k$, then $L_s$ satisfies the
   mean curvature flow. It is a flow of cylinders collapsing at time
   $s=1/2$ and for fixed
   $s_0 < 1/2$, the flow $L_s$ has uniformly bounded geometry on $s\in
   [-1,s_0]$, and so we can approximate the flow $\tilde{N}_s$ using
   the linearization, as $T\to \infty$. To make this precise, let
   $T^{(i)} \to \infty$, and consider the corresponding flows
   $\tilde{N}^{(i)}_s$. On the larger and larger balls
   $B_{R_{T^{(i)}}}$ we can write $\tilde{N}^{(i)}_s$ as the graph of
   $u^{(i)}$ over the cylinders $L_s$. By a barrier argument, and the
   expansion at $s=-1$, we see that on any compact set,
   $(T^{(i)})^{1/2} u^{(i)}$ is bounded for sufficiently large
   $i$. Up to choosing a subsequence we can extract a limit $U$ over
   the cylinder $L_s$ for $s\in [-1,s_0]$, and
   \[ U(\theta, y, -1) = \frac{n-k}{\sqrt{3(n-k)}} \mathbf{e}\cdot
     y. \]
   The function $U$ satisfies the linearized equation 
   \[ \partial_s U = \Delta U + |A|^2 U \]
   over $L_s$, so we find that
   \[ U(\theta, y, s) = \frac{n-k}{a(s)} \mathbf{e}\cdot y. \]
   Using that the $\sqrt{T^{(i)}} u^{(i)}$ converge to $U$ locally
   smoothly, it follows that on any fixed region $B_R \times [-1/2,
   s_0]$ we have
   \[ u(\theta, y, s) = \frac{n-k}{a(s) T^{1/2}} \mathbf{e}\cdot y +
     o(T^{-1/2}), \]
   where the error estimate holds in $C^{\ell}$ for any $\ell$.  This
   implies the result. 
\end{proof}

The following is the crucial consequence of the presence of the linear
$y\cdot \mathbf{e}$ term in the expansion~\eqref{eq:p22}.
\begin{lemma}\label{lem:support21}
 Let $\epsilon > 0$. There exists $s_0 > 0$ depending on
 $\epsilon$, and for any $R > 0$ there exists $T_0$ depending on
 $\epsilon, R$, with the following property. Following the notation of
 Lemma~\ref{lem:linearterm10},  on the translated and rescaled flow
 $\tilde{N}_s$ we have
 \[ \label{eq:supp22} |\langle \nu, x + T^{1/2}\mathbf{e} \rangle A| \leq \epsilon\]
 on the ball $B_R$, for $s\in [0, s_0]$. 
\end{lemma}
\begin{proof}
  We fix a time $s\in [0, s_0]$ for small $s_0$, and work on the
  time slice $\tilde{N}_{s}$. We write
  \[ X : (\theta, y) \mapsto (F_T(\theta, y, s) \theta, y) \]
  for the parametrization of $\tilde{N}_{s}$ in the ball $B_R$, and
  for simplicity write $F$ for $F_T$. The
  unit normal is given by
  \[ \nu = J^{-1} (\theta - F^{-1} \nabla_{S^{n-k}} F, -\nabla_y F), \]
  where
  \[ J = (1 + F^{-2}|\nabla_{S^{n-k}} F|^2 + |\nabla_y F|^2)^{1/2} = 1
    + O(T^{-1}). \]
  To write down the components of the second fundamental form $A$, we
  use coordinates $\theta_a$ on the sphere factor, and $y_i$ on the
  $\mathbb{R}^k$ factor. We write $\sigma_{ab}$ for the metric on the
  unit sphere $S^{n-k}$. We have
  \[ A_{ab} &= J^{-1}(-\nabla^2_{ab} F + F\sigma_{ab} + 2F^{-1} F_a
    F_b) = a(s) \sigma_{ab} + O(T^{-1/2}), \\
    A_{ai} &= J^{-1} (-\nabla^2_{ai} F + F^{-1} F_a F_i) = O(T^{-1/2}), \\
    A_{ij} &= -J^{-1} \nabla_{ij} F = O(T^{-1/2}). \]
  while the metric is $g = a(s)^2 \sigma + |dy|^2 + O(T^{-1/2})$. 

  From these formulas, and the expansion \eqref{eq:p22}, we have
  \[ \langle \nu, x\rangle &= J^{-1}\langle (\theta - F^{-1}
    \nabla_{S^{n-k}} F, -\nabla_y F), (F \theta, y)\rangle \\
    &= J^{-1}(F - y\cdot \nabla_y F) \\
  &= a(s) + O(T^{-1/2}). \]
We also have (note that here the linear term in \eqref{eq:p22} is important)
\[ \langle \nu, T^{1/2}\mathbf{e}\rangle &= J^{-1}\langle (\theta -
  F^{-1}\nabla_{S^{n-k}} F, -\nabla_y F), (0,
  T^{1/2}\mathbf{e})\rangle \\
  &= -\frac{n-k}{a(s)} + o(1), \]
and
\[ |A| = \frac{\sqrt{n-k}}{a(s)} + O(T^{-1/2}). \]
Combining these, we get
\[ |\langle \nu, x + T^{1/2}\mathbf{e}\rangle A| &=
  \frac{\sqrt{n-k}}{a(s)} \left(a(s) - \frac{n-k}{a(s)}\right) +
  o(1) \\
  &= \frac{\sqrt{n-k}}{a(s)^2}(-2s(n-k)) + o(1). \]
Note that $a(s)\to \sqrt{n-k}$ as $s\to 0$. So if $T$ is sufficiently
large, and $s_0$ is sufficiently small, then \eqref{eq:supp22} holds. 
\end{proof}

We will need the following conclusion of this, phrased in terms of the
original flow $N_t$. 
\begin{lemma}\label{lem:almostconical}
  For any $\epsilon > 0$, there exists a large $K > 0$ and small $\delta
  > 0$ such that on the region after the singularity where $|x| +
  t^{1/2} < \delta$, and $|x| > K | t |\log t||^{1/2}$, we have
  \[ |\langle \nu, x\rangle A| \leq \epsilon. \]
\end{lemma}
\begin{proof}
  Consider a point $(x_0, t_0)$ in the given region. We have already
  seen that the $y$-projection $y_0$ of $x_0$ satisfies $|y_0| >
  \frac{1}{2}|x_0|$, if $\delta$ is small enough. Let $y_0 = r_0
  \mathbf{e}$, and consider the rescaling and translation
  $\tilde{N}_s$ as above. Recall that by definition $Te^{-T} = r_0^2$
  and $\lambda = e^{-T/2}$ is the scaling factor. As a result, the
  point $(x_0, t_0)$ along the flow $N_t$ will correspond to the point
  $(e^{T/2}x_0 - T^{1/2} \mathbf{e}, e^T t_0)$ along
  $\tilde{N}_s$. Moreover the point $e^{T/2}x_0 - T^{1/2}\mathbf{e}$ is
  contained in a ball of a fixed radius, say
  $\sqrt{2(n-k)}$.

  We apply Lemma~\ref{lem:support21}, with $R = \sqrt{2(n-k)}$ and our
  given $\epsilon$. Choosing $\delta$, and therefore $r_0$ small
  ensures that $T$ is very large. It remains to ensure that $e^T t_0 <
  s_0$ for the $s_0$ given by Lemma~\ref{lem:support21}. By assumption
  we have
  \[ Te^{-T} =  r_0^2 > \frac{K^2}{4} t_0 |\log t_0|. \]
  Since the function $x e^{-x}$ is decreasing, if $K^2/4 > 1$, this
  implies that $T < |\log t_0|$. Therefore, rearranging the inequality
  we have
  \[ e^T t_0 < \frac{4}{K^2} \frac{T}{|\log t_0|} < \frac{4}{K^2}. \]
  Choosing $K$ sufficiently large implies that $e^T t_0 < s_0$ for any
  given $s_0$.

  Lemma~\ref{lem:support21} therefore implies that $|\langle \nu, x +
  T^{1/2}\mathbf{e}\rangle A|\leq \epsilon$. Note that these are the
  quantities computed along the flow $\tilde{N}_s$. Denoting by
  $\hat{x}, \hat{A}$ the corresponding quantities on $N_t$ (note that
  the unit normal doesn't change), we have
  \[ \hat{x} &= \lambda x + y_0 = e^{-T/2} (x + T^{1/2}\mathbf{e}), \]
  while the second fundamental form scales by $|\hat{A}| = e^{T/2}
  |A|$. The required result follows. 
\end{proof}

\subsection{Sharp metric change upper bound}
The purpose of this section is to prove the following estimate, which is crucial for the conditional backwards uniqueness theorem (Theorem \ref{main_uniqueness_thm}). By convention, for a symmetric $2$ tensor $T$ and a real number $\lambda$ we write $T\geq \lambda $ if $\lambda_1(T) \geq \lambda$, where $\lambda_1(T)$ is the smallest eigenvalue of $T$. 

\begin{theorem}\label{metric_changed_upper_bound}
For every $\varepsilon>0$ There exists some $t_0>0$ such that for all $t\in (0,t_0)$
\begin{equation}\label{desired_AH_bd}
AH \geq -\frac{1/2+\varepsilon}{t}
\end{equation} 
\end{theorem}

\subsubsection{Sharp $A/H$ lower bound} The main ingredient in proving Theorem \ref{metric_changed_upper_bound} is the following.

\begin{theorem}[$A/H$ lower bound after the singularity]\label{quo_after}
For every $K<\infty$ and $\varepsilon>0$ there exist some physical time $t_0>0$ such that for every $t\in (0,t_0)$
\begin{equation}
\frac{A}{H} \geq -\frac{1+\varepsilon}{|\log(t)|},\qquad x\in N_{t}\cap B(0,K\sqrt{t|\log(t)|}).
\end{equation}
\end{theorem}

In order to derive Theorem  \eqref{quo_after}, we first obtain the following lower bound \textit{prior} to the singularity, that is sharp in time, but Gaussian deteriorating in space.

\begin{lemma}
Under the condition of Lemma \ref{C^2_crude_lemma}, there exist some $\tau_0$ and $C<\infty$ such that for every $x\in M_{\tau}\cap B(0,\varepsilon e^{\tau/2})$ with $\tau \geq \tau_0$ we have
\[
\frac{A}{H} \geq -\mathrm{min}\left\{\left(\frac{1}{\tau}+\frac{Ce^{|x|^2/6}}{\tau^{3/2}}\right),\varepsilon\right\}
\]
\end{lemma}
\begin{proof}
This is a fairly direct application of Lemma \ref{C^2_crude_lemma}, \cite{HK16} and the fact that  the Hessian of the graph of $u_0$ essentially supports a $-\frac{g}{\tau}$ lower bound.

Indeed, observe that by the convexity estimate of \cite{HK16}, $A/H \geq -\varepsilon$ in the above domain, for $\tau_0$ sufficiently large. In particular, the desired estimates holds trivially if $|x|^2 \geq 12\log \tau$, as then the right hand side is bounded from above by $-\min\{C\sqrt{\tau},\varepsilon\} = -\varepsilon$.

For $|x|^2 \leq  12\log \tau$ we have that 
\begin{equation}\label{u_0_2_der}
|u_0|+|\nabla u_0| \leq \frac{1}{\tau^{3/4}}, \qquad \mathrm{Hess}_{\mathbb{R}^k}(u_0)=\frac{\sqrt{2(n-k)}}{2\tau}+O(\tau^{-7/4}).
\end{equation}
This yields that for the graph of the ideal profile 
\[
A_{0} \geq -\frac{\sqrt{2(n-k)}}{2\tau}+O(\tau^{-7/4}),
\]
and combined with the spherical terms we get
\[
H_{0} \geq \frac{(n-k)}{\sqrt{2(n-k)}}+O(\tau^{-3/4}),
\]
so for the ideal profile one gets
\begin{equation}\label{ideal_q}
\frac{A_{0}}{H_{0}} \geq -\frac{1}{\tau}+O(\tau^{-7/4}).
\end{equation}

\bigskip

Finally, by cylindricality
\[
|A|+|A_0| +H^{-1}+H_0^{-1} \leq C,
\]
so, by \eqref{u_0_2_der}, \eqref{SX1_est} and the fact that the expressions for mean curvature and the second fundamental form are quasilinear, we get the point-wise bound
\begin{equation}
|A-A_0|+|H-H_0| \leq C(|w|+|Dw|+|D^2w|),
\end{equation}
where $w$ is as in that Lemma \ref{C^2_crude_lemma}. Combining this, \eqref{SX2_est}, \eqref{ideal_q} and the bound
\begin{equation}
\left|\frac{A}{H}-\frac{A_{0}}{H_{0}}\right| \leq \frac{|A-A_{0}|}{H}+\frac{|A_0|}{HH_0}|H-H_0|,
\end{equation}
yields the desired result.
\end{proof}

Due to the drift term the estimate above can be upgraded to a sharp one inside $B(0,K\sqrt{\tau})$.

\begin{theorem}\label{AH_bd_prior}
For every $K<\infty$ there here exist some $\tau_0,C<\infty $ such that for all $\tau\geq \tau_0$ and $x\in M_{\tau}\cap B(0,K\sqrt{\tau})$
\begin{equation}
\frac{A}{H} \geq -\frac{1}{\tau}-\frac{C}{\tau^{3/2}}
\end{equation}
\end{theorem}
\begin{proof}
Recall that for the rescaled MCF
\begin{equation}\label{q_eq}
L[A/H]=0,\qquad \textrm{for}\qquad L[T]=\partial_{\tau}T-\Delta T+\frac{1}{2}\nabla_{x^{\top}}T-2\left\langle\nabla \log H, \nabla T\right\rangle.
\end{equation}
We will look for barriers of the form
\begin{equation}\label{f_def}
f(x,\tau)=a(\tau)\exp(b(\tau)|x|^2)
\end{equation}
Keeping in mind that $\partial_{\tau}x= H+\frac{x^\perp}{2}$ along a rescaled MCF, observe that
\begin{equation}
L[|x|^2]=2\langle \partial_{\tau}x,x\rangle - 2\langle H, x \rangle-2n+|x^{\top}|^2-4\langle \nabla \log H,  x^{\top} \rangle=-2n+|x|^2-4\langle \nabla \log H,  x^{\top} \rangle,
\end{equation}
and so 
\begin{align}
L[f]&=f\left[\frac{a'}{a}+|x|^2b'+b(-2n+|x|^2-4\langle \nabla \log H,  x^{\top} \rangle)-4b^2|x^{\top}|^2\right] \\
& \geq f(\frac{a'}{a}-Cb)+f|x|^2(b'+3/4b-4b^2).
\end{align}

The solution to the ODE $b'+\frac{3}{4}b-4b^2=0$ with initial value $b(\tau_1)=1/6$ is
\[
b(\tau)=\frac{3}{16+2e^{\frac{3(\tau-\tau_1)}{4}}}.
\] 
Integrating the equation $a'/a=Cb$ with $a(\tau_1)=\frac{C}{\tau_1^{3/2}}$ (which is explicitly solvable) yields an increasing function bounded by $\frac{D}{\tau_1^{3/2}}$, for some larger $D$.  Thus, with this choice, we get that 
\[
L[f]\geq 0.
\]
\bigskip

The initial conditions for $f$ and the previous lemma guarantee that both at $\tau=\tau_1$ and on $\partial B(0,\varepsilon e^{\tau_1/2})\times [\tau_1,\tau_1+2\log(\tau_1)]$ we have 
\begin{equation}\label{bd_cond}
\frac{A}{H}(\cdot, \tau)+\frac{1}{\tau_1}+f(\cdot, \tau) \geq 0.
\end{equation}
Therefore, the tensor maximum principle implies that, taking $\tau_2=\tau_1+2\log(\tau_1)$, we have
\begin{equation}
\frac{A}{H}(\cdot, \tau_2)+\frac{1}{\tau_1} \geq -f(\cdot, \tau_2). 
\end{equation}
Observing that $\frac{1}{\tau_1}=\frac{1}{\tau_2}+O(\tau_2^{-3/2})$ and  that $b(\tau_2)\leq \frac{4}{\tau_2}$, we get that at point with $|x|<K\sqrt{\tau_2}$ 
\[
f(x, \tau_2) \leq \frac{D}{\tau_2^{3/2}},
\]
after increasing $D$. This completes the proof.

\end{proof}

The fact that the estimate of Theorem \ref{AH_bd_prior} is valid at logarithmically expanding scales (in terms of $\tau = |\log |t||$), and the  opening of the non-degenerate neck pinch by Lemma \ref{C^2_crude_lemma}, allows us to push the above estimate from physical time $-t$ to physical time $t$ without a significant loss, allowing to prove the sharp $A/H$ lower bound at \textit{positive} times of Theorem \ref{quo_after}. The details are as follows.  

\begin{proof}[Proof of Theorem \ref{quo_after}]
The proof has two steps: In the first, we show the bound at points at distance precisely $K\sqrt{t|\log t|}$. This is a result of the fact that for a parabolic equation with boundary, a far away spatial boundary has an exponentially decaying effect on the solution at time $1$. This is most classically proven via heat kernel estimates, but in our tensor setting, we will show this via the maximum principle.  Once this is established, we can apply a maximum principle for $A/H$ on the entire domain (keeping in mind the singularity at $(0,0)$) to complete the proof.

For the first step, consider a time $t_1\in (0,t_0)$ and consider the parabolic rescaling of $N_t$ by a factor $t_1$
\begin{equation}
\tilde{N}_s=t_1^{-1/2}N_{t_1s},\qquad s\in [-1,1].
\end{equation}
Denoting 
\[
\mathcal{A}_K=B\Big(0,(K+1)\sqrt{|\log(t_1)|}\Big)-B\big(0,(K-1)\sqrt{|\log(t_1)|}\Big)
\]
Taking $K\geq 3$ it follows from Lemma \ref{very_cyl_away_lemma} and its proof, that for $s\in [-1,1]$, $\tilde{N}_s\cap \mathcal{A}_K$ has controlled geometry, and it can be assumed to be as cylindrical as we want. In particular, for every $s\in [-1,1]$ we have that 
\[
\frac{A}{H}(x,s) \geq -1,\qquad x\in \tilde{N}_s\cap \partial \mathcal{A}_K.
\] 
The analogue of \eqref{q_eq} along MCF is 
\begin{equation}\label{q_eq_MCF}
\tilde{L}\left(\frac{A}{H}\right)=0,\qquad\textrm{for}\;\;\;\;\;\tilde{L}[T]=\partial_{s}T-\Delta T-2\left\langle\nabla \log H, \nabla T\right\rangle.
\end{equation}
By the MCF flow equation $\partial_s-\Delta x=0$ so
\[
|(\partial_s-\Delta)|x|| \geq -\frac{C}{|x|}.
\]
Since $\tilde{M}_s\cap \mathcal{A}_K$ has controlled geometry and $|\nabla |x|| \leq 1$ we therefore get that
\begin{equation}\label{Lx}
\left|\tilde{L} |x|\right| \leq C, \qquad x\in \mathcal{A}
_K\cap \tilde{N}_s,\;s\in [-1,1].
\end{equation}

\bigskip
Now, considering a barrier of the form 
\[
f_1(x,s)=e^{as}e^{-b(|x|-(K-1)\sqrt{\log|t_1|})}, 
\]
\eqref{Lx} implies
\[
\tilde{L} f_1 = f_1\left[a-\tilde{L}[|x|]b-b^2|\nabla |x||^2\right] \geq f_1(a-Cb-b^2).
\]
Take $b=\beta\sqrt{\log (t_1^{-1})}$ and $a=\alpha\log(t_1^{-1})$. We choose $\beta=1/8$, $\alpha=1/32$ (we need $\beta^2 < \alpha$, and later we will need $\beta > 2\alpha$), which implies
\[
\tilde{L}f_1 \geq 0.
\]
Considering a similar barrier which decays away from the outer boundary of $\mathcal{A}_K$, we get that for the above choice of $a$ and $b$ 
\begin{equation}
\tilde{L}[f]\geq 0,\;\;\textrm{for}\;\;f(x,s)=e^{as}\left[e^{-b(|x|-(K-1)\sqrt{|\log t_1|})}+e^{-b((K+1)\sqrt{|\log t_1|}-|x|)}\right]. 
\end{equation}

\bigskip

Note $f$ is non negative, and for every $s\in [-1,1]$ and $x\in \partial \mathcal{A}_K$ we have that  $|f(x,s)|\geq t_1^{-\alpha s}$. By this and by Theorem \ref{AH_bd_prior}, on the parabolic boundary of $\mathcal{A}_K\times [-1,1]$ we have
\begin{equation}
\frac{A}{H} \geq -\frac{1+\varepsilon}{|\log t_1|}-\frac{f}{t_1^{\alpha}}.
\end{equation} 
Thus, by the tensor maximum  principle we have at points $x$ with $|x|=K\sqrt{|\log(t_1)|}$ that 
\begin{equation}
\frac{A}{H}(x,1) \geq -\frac{1+\varepsilon}{|\log t_1|}-\frac{f}{t_1^{\alpha}} \geq  -\frac{1+\varepsilon}{|\log t_1|}-\frac{2t_1^{\beta}}{t_1^{2\alpha}} \geq -\frac{1+2\varepsilon}{|\log t_1|},
\end{equation}
where the last inequality follows from $2\alpha < \beta$ in our parameter choice above.  This completes the first step of the proof.

\bigskip

Having established the bound on $\tilde{N}_s\cap \partial B(0,K\sqrt{|\log(t_1)|})$, for every 
\[
\delta\in (0,(2|\log(t_1)|)^{-1}\varepsilon)
\]
we can apply the tensor maximum principle for the tensor 
\[
Q_{\delta}=\frac{A}{H}+\frac{1+2\varepsilon}{|\log(t_1)|}+\delta(s+1).
\]
and for the flow $\tilde{N}_s\cap B(0,K\sqrt{|\log(t_1)|})$ and $s\in [-1,1]$. Note that the bound $Q_{\delta} >0$ holds true on the parabolic boundary, as well as some two sided parabolic neighborhood of the origin (by \cite{HK16}). Suppose that $s_0$ is the largest time in $[-1,1]$ for which $Q_{\delta}\geq 0$. If $s_0<1$, then there exists some interior point $(x_0,s_0)\neq (0,0)$  where $\lambda_1(Q_{\delta})=0$, but this contradicts the evolution equation for $\frac{A}{H}$. Thus $Q_{\delta}\geq 0$ for every such $\delta$ and we are done.
\end{proof}

\subsubsection{Near expansion by the tip}\label{near_tip_exp_sec}

In this section, we will prove Theorem \ref{quo_after21}, which was already used to derive the tip curvature estimate. For the reader's convenience, we state it again here. 

\begin{theorem}[Expander behavior after the singularity]\label{quo_after2}
For every $K<\infty$ and $\varepsilon>0$ there exist some physical time $t_0>0$ such that for every $t\in (0,t_0)$
\begin{equation}
\left|2+\frac{\langle x,\nu \rangle}{tH}\right|<\varepsilon ,\qquad x\in N_{t}\cap B(0,K\sqrt{t|\log(t)|}).
\end{equation}
\end{theorem} 

Similarly to the proof of the $A/H$ lower bound, we will start with an information prior to the singularity time, and push it forward via (the same) maximum principle to positive times. 
\begin{lemma}[Shrinker behavior before the singular time] 
For every $K<\infty$ and $\varepsilon$ there exists $\tau_0>0$ such that for $\tau\geq \tau_0$, along the rescaled flow we have 
\begin{equation}
\left|\frac{\langle x,\nu \rangle}{H}-2\right|< \varepsilon
\end{equation}
in  $B(0,K\sqrt{\tau})\cap M_{\tau}$
\end{lemma}

\begin{proof}
By \eqref{SX1_est} and everywhere cylindricality, we have that, inside $B(0,K\sqrt{\tau})$, denoting by $M^0_\tau$ the graph of $u_0$ over the cylinder we have
\[
|\langle x^0,\nu^0\rangle-\langle x,\nu \rangle | \leq \frac{C}{\tau^{3/4}},\qquad |H-H^0| \leq \varepsilon
\]
where the identification is done in the domain of the cylindrical graph.
As $H^0$ is bounded and bounded away from zero in the domain, we have
\[
\left|\frac{\langle x,\nu \rangle}{H}-\frac{\langle x^0,\nu^0 \rangle}{H^0}\right| \leq \frac{|\langle x,\nu \rangle-\langle x^0,\nu^0 \rangle|}{H}+|\langle x^0,\nu^0 \rangle|\frac{|H-H_0|}{HH_0}.
\]
Therefore, if we were to show 
\begin{equation}
\left|\frac{\langle x^0,\nu^0 \rangle}{H^0}-2\right|< \varepsilon,
\end{equation}
the result would follow.
\bigskip

For that, from the precise form of $u_0$, a direct calculation shows that  setting $S:=\sqrt{1+\frac{\sum y_i^2}{2\tau}}$ we have

\begin{align*}
H^0&= \frac{\sqrt{(n-k)}}{\sqrt{2}S}+o(1),\\
\nu^0&=\frac{(-\sqrt{2(n-k)}\frac{y_1}{2S\tau}\ldots -\sqrt{2(n-k)}\frac{y_k}{2S\tau},\hat{\theta})}{\sqrt{1+\frac{1}{S^2t^2}\sum y_i^2}},\\
x^0&=(y_1,\ldots y_k, \sqrt{2(n-k)}S\hat{\theta}).
\end{align*}

Therefore, 
\begin{equation}
    \frac{\langle x^0,\nu^0\rangle}{H^0}=2\left(S^2-\sum \frac{y_i^2}{2\tau}\right)+o(1)=2+o(1).
\end{equation}
This completes the proof.

\end{proof}

We can prove Theorem \ref{quo_after21}.
\begin{proof}[Proof of Theorem \ref{quo_after21}]
The proof is exactly the same as the one of Theorem \ref{quo_after}. Taking some time $t_1>0$ and defining $\tilde{N}_s$, $\mathcal{A}_K$ as therein,  the quantity
\[
q=2s+\frac{\langle x,\nu \rangle}{H}
\]
satisfies $|q(\cdot,-1)|\leq \varepsilon$ and $|q|\leq C|\log(t_1)|$ for every $s\in [-1,1]$ and $x\in \tilde{N}_s\cap \partial \mathcal{A}_k$, for some $C=C(K)$.
Moreover, for $\tilde{L}$ as in Theorem \ref{quo_after}
\[
\tilde{L}[q]=0, 
\]
we conclude using the first maximum principle of Theorem \ref{quo_after} (or heat kernel estimates) to get that $|q| \leq 2\varepsilon$ for every $s\in [-1,1]$ and $x\in \tilde{N}_s\cap \mathcal{A}_K$. Since $q$ is small in a two sided neighborhood of the origin, using the second maximum principle argument of Theorem \ref{quo_after} gives, at time $s=1$ 
\[
\left|2+\frac{\langle x,\nu \rangle}{H}\right |\leq 2\varepsilon.
\]
Scaling back to $N_t$ we get that at time $t_1$
\[
\left|2t_1+\frac{\langle x,\nu \rangle}{H}\right|<2\varepsilon t_1.
\]
Dividing by $t_1$ gives the result.
\end{proof}

\subsubsection{Conclusion of the argument}

We are finally ready to prove the main theorem of this section.
\begin{proof}[Proof of Theorem \ref{metric_changed_upper_bound}]
First observe that by \eqref{AlmostCon_bd}, taking $K$ sufficiently large we have that 
\[
|A(x,t)| \leq \frac{\delta}{\sqrt{t}}
\]
if $|x|\geq K\sqrt{t\log|t|}$. Thus for such points $x$ \eqref{desired_AH_bd} is trivial. Fix that $K$.

\bigskip

By Lemma  \ref{very_cyl_away_lemma}, for small enough $t$, given a point $x\in N_t\cap B(0,K\sqrt{t|\log t|})$ with $|x|\geq (\sqrt{2}+\varepsilon)\sqrt{t|\log (t)|}$ we have
\[
|\langle x,\nu \rangle| \leq \frac{1}{K}|x| \leq \sqrt{t|\log t|}. 
\]
In light of Theorem \ref{quo_after2}, this gives
\[
H(x,t) \leq (1/2+\varepsilon)\sqrt{t|\log t|}
\]
which combined with Theorem \ref{quo_after} yields \eqref{desired_AH_bd} at such $x$.

\bigskip

Finally, for $x\in M_t\cap B(0,(\sqrt{2}+\varepsilon)\sqrt{t|\log t|})$,  we have 
\[
|\langle x,\nu \rangle| \leq |x| \leq (\sqrt{2}+\varepsilon)\sqrt{t|\log t|}),
\]
and so Theorem \ref{quo_after2} yields
\[
H \leq \frac{\sqrt{|\log t|})}{t(\sqrt{2}-\varepsilon)}.
\]
Together with Theorem  \ref{quo_after}, this gives \eqref{desired_AH_bd}.
\end{proof}

\section{Finding good initial perturbations}\label{sec:initialperturb}
Given the discussion in Section~\ref{sec:perturb}, and in particular
Proposition~\ref{prop:genericperturb}, our main remaining task is to show 
that if we want to perturb away degenerate cylindrical singularities
of a flow $L_t$
near a space-time point $(x_0, t_0)$, then for any $\kappa > 0$ we can find initial
perturbations of our flow that give rise to a $\kappa$-good family of
perturbations of the rescaled flow $M_\tau$ centered at $(x_0, t_0)$,
at a suitable time $T_1$, as in Definition~\ref{defn:kappagood}. We
will apply successive perturbations to the original flow $N_t$, and in
this discussion we suppose that $L_t$ is already one such small
perturbation, which we seek to further perturb in order to eliminate
further degenerate singularities.

Since
$M_{T_1+1}$ is a rescaling (and translation, and rotation) of $L_{t_0
  - e^{-T_1-1}}$, this amounts to finding a suitable family of perturbations
$L^{\mathbf{a}}_0$, for which $L^{\mathbf{a}}_{t_1}$ has the desired
graphicality property over $L_{t_1}$, where $t_1 = t_0 -
e^{-T_1-1}$. In our application we will assume that $L_{t_1}$ is
smooth, and has only non-degenerate cylindrical singularities in the
interval $[0, t_1]$. 

To simplify the notation we will translate in time, and work with a flow $L_t$ defined for $t\in
[-1,1]$, with a single non-degenerate singularity at $(0,0)$. We want
to perturb the data at $t=-1$ in order to have a prescribed effect at
$t=1$. To set this up more precisely, let $\mathcal{B}\subset
C^{2,\alpha}(L_{-1})$ be a small ball centered at the origin. For any $f\in \mathcal{B}$ we define the
flow $L_{f,t}$ such that $L_{f,-1}$ is the graph of $f$ over
$L_{-1}$. From Sun-Wang-Xue's work~\cite{SWX25} we have the following. 

\begin{prop}\label{prop:curvest20}
  Given $\delta > 0$, if $\mathcal{B}$ is small enough, then from
  Theorem~\ref{thm:isolated} we know that for $f\in \mathcal{B}$ the
  flow $L_{f, t}$ also only has one non-degenerate cylindrical
  singularity, which is in the parabolic $\delta$-neighborhood of $(0,0)$.

\end{prop}

Our goal in this section is to prove the following.
\begin{prop}\label{prop:10}
  Let $h$ be a smooth function on $L_1$, and let $\kappa > 0$. Then
  there exists $f\in C^{2,\alpha}(L_{-1})$, and $\delta > 0$ with the
  following property. For $|a| < \delta$ define the flow $L^a_t :=
  L_{af, t}$ with initial condition $L_{af,-1}$ given by the graph of
  $af$ over $L_{-1}$. Then for any $|a_1|, |a_2| < \delta$, the
  hypersurface $L^{a_2}_1$ is the graph of $F$ over $L^{a_1}_1$, where
  \[ |F - (a_2-a_1)h| \leq \kappa |a_2-a_1|. \]
\end{prop}

The proof will rely on several intermediate steps. A major ingredient, which we
defer to Section~\ref{sec:backwardsunique}, is a backwards uniqueness theorem for the
conjugate linearized operator, which will be used in
Proposition~\ref{prop:11}.
We first construct some barrier functions that will be used when
studying solutions of the linearized equation. For this we will need
the following estimate
\begin{lemma}\label{lem:31}
  In a small neighborhood of the singularity $(0,0)$ we have
  $|\nabla\log H| < |A|$.
\end{lemma}
\begin{proof}
  This follows from the fact that near the singularity the flow is
  modelled on cylinders and the bowl soliton, by
  Sun-Wang-Xue~\cite{SWX25}. 
  The desired inequality is
  scale invariant, so if it were to fail along a sequence of points
  $X_i \to (0,0)$, then after rescaling by the curvature at $X_i$, and taking a limit, we would
  find a point $X_\infty$ on either a cylinder, or on the bowl
  soliton, such that $|\nabla\log H|(X_\infty) \geq |A|(X_\infty)$. On
  a cylinder this is clearly a contradiction. On the bowl soliton we
  can do an explicit calculation to see that this also cannot
  happen. 

  For this let us write the $m$-dimensional bowl soliton as the graph
  of $u(r) = u(|x|)$ over $\mathbb{R}^m$. The translator equation can
  be written as
  \[  \frac{u''}{1 + (u')^2} + \frac{m-1}{r} u' = 1, \]
  which implies that the mean curvature is given by
  \[ H = \frac{1}{(1 + (u')^2)^{1/2}}. \]
  The angular principal curvatures are
  \[ \kappa_\theta = \frac{u'}{r (1 + (u')^2)^{1/2}}. \]
  In terms of this we have
  \[\label{eq:logH} |\nabla\log H| &= \frac{ u' u'' }{(1 + (u')^2)^{3/2}} = \frac{r
      u''}{1 + (u')^2} \kappa_\theta. \]
  We first claim that $u' > r/m$. To see this, let $w=u' - r/m$, and
  note that by expanding the solution $u'$ of the
  translator equation as a power series near the origin, we get $u'(r)
  = \frac{r}{m} + O(r^3)$. Therefore $w(0)=0$. Whenever $w(r) \leq 0$,
  then we have $u' \leq r/m$ and so
  \[ w' &= u'' - \frac{1}{m} \\
    &= \left( 1 - \frac{m-1}{r} u'\right) (1 + u'^2) - \frac{1}{m}\\
    &> \left(1 - \frac{m-1}{m}\right) (1 + u'^2) - \frac{1}{m}> 0. \]
  This shows that in fact $w(r) > 0$ for all $r > 0$, as claimed.

    Let us define $F = ru'' (1 + (u')^2)^{-1}$. Then the translation
    equation implies
    \[ F' = 1 - \frac{(m-1)(1 + (u')^2)}{r} F. \]
    We have $F(0)=0$, and we claim that $F(r) < 1$ for all
    $r$. Indeed, if at some first time $r_0$ we had $F(r_0)=1$, then
    $F'(r_0) \geq 0$, however we have
    \[ F'(r_0) \leq 1 - \frac{2(m-1) u'}{r}, \]
    and at the same time we showed that $u' > \frac{r}{m}$, so
    \[F'(r_0) < 1 - \frac{2(m-1)}{m}. \]
    Since $m \geq 2$, this implies $F'(r_0) < 0$, which is a
    contradiction.

    Since $F < 1$, using \eqref{eq:logH} we get $|\nabla \log H| <
    \kappa_\theta \leq |A|$. This shows that on the bowl soliton we
    cannot have a point $X_\infty$ with $|\nabla \log H|(X_\infty)
    \geq |A|(X_\infty)$, as claimed. 
\end{proof}

\begin{lemma}\label{lem:barrier20}
  There is a function $\Phi > 0$ along the flow $L_t$ away from
  $(0,0)$ satisfying the following properties:
  \[ \begin{aligned}
      (\partial_t - \Delta - |A|^2) \Phi &\geq 0, \\
      C^{-1} (1 + |H|) \leq \Phi &\leq C(1 + |H|), 
    \end{aligned} \]
  where $r = |x| + |t|^{1/2}$. There is an additional positive supersolution
  $\Psi$ satisfying
  \[ (\partial_t - \Delta - |A|^2) \Psi >  0, \qquad \text{ and }
    \frac{\Psi}{\Phi} \to \infty \text{ as } r\to 0. \]
\end{lemma}
\begin{proof}
   The (signed) mean curvature $H$ satisfies $(\partial_t - \Delta -
   |A|^2)H = 0$. Let us define $g(s) = (1+s^2)^{1/2}$. Then
   \[ \begin{aligned}  (\partial_t -\Delta - |A|^2) g(H) &= (H g'(H) -
       g(H)) |A|^2 - g''(H) |\nabla H|^2 \\
       &=  - \frac{|A|^2}{g(H)} - \frac{|\nabla
         H|^2}{g(H)^3}, \end{aligned} \]
   and $g(H)$ is comparable to $1+|H|$. 
   For large $\Lambda > 0$ define
   \[ \label{eq:Phidef31} \Phi = e^{\Lambda t} g(H) = e^{\Lambda t} (1 + H^2)^{1/2}, \]
   which is also comparable to $1 + |H|$. We have
   \[ (\partial_t -\Delta - |A|^2) \Phi = e^{\Lambda t} \left( - \frac{|A|^2}{g(H)} - \frac{|\nabla
         H|^2}{g(H)^3} + \Lambda g(H)\right) \geq 0, \]
   if $\Lambda$ is sufficiently large, using 
   the Haslhofer-Kleiner curvature estimates, Proposition~\ref{prop:HK16}. 

   To construct the second supersolution $\Psi$, we argue similarly, but we
   now consider $g: \mathbb{R} \to (0,\infty)$ which is a smooth
   positive function, that equals $s\log s$ for sufficiently large
   $s$. We again have
   \[ (\partial_t - \Delta - |A|^2) g(H) = (H g'(H) - g(H)) |A|^2 -
     g''(H) |\nabla H|^2. \]
   On a small neighborhood $U$ of $(0,0)$ we have $g(H) = H\log H$, so
   here we have
   \[ (\partial_t - \Delta - |A|^2) g(H) = H (|A|^2 - |\nabla \log
     H|^2) \geq 0, \]
   by Lemma~\ref{lem:31}. Outside of the neighborhood $U$ the
   curvature is uniformly bounded. It follows that if we define $\Psi
   = e^{\Lambda_1 t} g(H)$ for sufficiently large $\Lambda_1$, then
   \[ (\partial_t - \Delta - |A|^2) \Psi = e^{\Lambda_1 t}\left(
       (\partial_t - \Delta - |A|^2) g(H) + \Lambda_1 g(H) \right) \geq
     0. \]
   Since $H\to \infty$ as $r\to 0$, it is clear that $\Psi / \Phi \sim
   \log H \to \infty$ as $r\to 0$. 
   \end{proof}

We define the following classes of functions with controlled
singularities at $(0,0)$.
\begin{defn} A function $u$ along $L_t$ is called \emph{admissible} if
  $|u| \leq C \Phi$ for a constant $C > 0$. 
\end{defn}

The following lemma gives a maximum principle for the linearized
equation.
\begin{lemma}\label{lem:maxprinc2}
  Suppose that $u$ is an admissible function along $L_t$, satisfying the
  differential inequality
  \[ (\partial_t - \Delta - |A|^2) u \leq 0. \]
 Then we have
  \[ \sup_{L_b} \frac{u_+}{\Phi} \leq \sup_{L_a} \frac{u_+}{\Phi}, \]
 for any $a < b$, where $u_+ = \max\{u, 0\}$. 
\end{lemma}
\begin{proof}
  For $\delta > 0$ let
  \[ W_\delta = u - \left(\sup_{L_a} \frac{u_+}{\Phi}\right) \Phi -
    \delta\Psi. \]
  Note that $W_\delta < 0$ at $t=a$. In addition we have that
  $W_\delta \to -\infty$ at $(0,0)$ by the property of $\Psi$. We
  claim that $W_\delta \leq 0$ for all $t\in [a,b]$. Otherwise there
  would be a first space-time point $(x_0, t_0)$ where $W_\delta = 0$,
  and at this point we would have
  \[ 0 \leq (\partial_t - \Delta - |A|^2) W_\delta < 0, \]
  which is a contradiction. It follows that we have
  \[ \sup_{L_b} \frac{u}{\Phi} \leq \sup_{L_a} \frac{u_+}{\Phi} + \delta
    \frac{\Psi}{\Phi}. \]
  Letting $\delta\to 0$ gives the required result. 
\end{proof}

As a consequence we have the following solvability of the
linearized equation.
\begin{lemma}\label{lem:solveinhom}
  Let $u_a$ be
  a smooth function on a time slice $L_a$ with $a\not=0$. There exists
  a unique admissible solution $u$ of the linearized equation
  \[ (\partial_t - \Delta - |A|^2) u = 0, \]
  for $t \geq a$, with initial condition $u(x,a) = u_a(x)$. The
  solution $u$ satisfies
  \[ \sup_{L_t} \frac{|u|}{\Phi} \leq \sup_{L_a} \frac{|u_a|}{\Phi}, \]
  for $t \geq a$. 
\end{lemma}
\begin{proof}
  The uniqueness of the solution, and the claimed estimate for $|u| /
  \Phi$ follows from applying Lemma~\ref{lem:maxprinc2} to $u$ and
  $-u$. To prove the existence of $u$ we argue by perturbing the
  flows, as follows.

  To simplify the notation we will assume that $a=-1$. For a given
  $u_{-1}$ on $L_{-1}$, and small $\epsilon > 0$ consider the flow
  $L^{\epsilon u_{-1}}_t$ whose initial condition is given by the
  graph of $\epsilon u_{-1}$ over $L_{-1}$. 
  For sufficiently small
  $\epsilon$ we can use the barriers from
  Proposition~\ref{prop:barrier} to control the flow $L^{\epsilon
    u_{-1}}_t$. We can first use a simple barrier given by the graph
  of $\epsilon \Lambda e^{\Lambda(t+1)}$ over $L_t$, for suitable large $\Lambda$, until $t=-\xi$, then use
  Proposition~\ref{prop:barrier} to bound the flow slightly past the
  singular time $t=0$, and then use a simpler barrier again until
  $t=1$. The conclusion is that we have a constant $C$, depending on
  $u_{-1}$ but not on $\epsilon$, such that the following holds: for
  any compact set $K$ away from $(0,0)$, once $\epsilon$ is
  sufficiently small (depending on $K$), we can write $L^{\epsilon
    u_{-1}}_t$ as the graph of a function $F_\epsilon(x,t)$ over
  $L_t$, and $|F_\epsilon(x,t)| \leq C\epsilon (1 + |H(x,t)|)$. It
  follows that as $\epsilon \to 0$, along a sequence the normalized functions
  $\epsilon^{-1} F_\epsilon$ converge locally smoothly, away from
  $(0,0)$, to a solution $u(x,t)$ of the linearized equation, with
  $u(x,-1) = u_{-1}(x)$. Moreover $u$ satisfies $|u(x,t)| \leq C (1 +
  |H|)$, so $u$ is admissible, as required. 
\end{proof}

Note that the same arguments apply to small perturbations of $L_t$,
away from their non-degenerate singularities. In particular for sufficiently small
$f \in \mathcal{B}$, we can consider the linearization of the mean curvature flow equation
along the flow $L^f_t$. For any function $u$ on $L^f_{-1}$ we 
can construct a unique admissible solution $\mathcal{P}^f_t(u)$ of the equation
\[ (\partial_t - \Delta - |A|^2) \mathcal{P}^f_t(u) = 0 \]
along the flow $L^f_t$, away from its singular point $(x_0,t_0)$, with
initial condition $u$.

For  $\mathcal{B}\subset
C^{2,\alpha}(L_{-1})$ sufficiently small, let us define the nonlinear map
\[ \mathcal{N} : \mathcal{B} \to L^\infty(L_1), \]
so that $L^f_{1}$ is the graph of  $\mathcal{N}(f)$ over $L_1$. The
construction of $u(x,t)$ in Lemma~\ref{lem:solveinhom} proves the
following. 
\begin{prop}
  The map $\mathcal{N}$ is (Gateaux) differentiable, and its derivative
  $D\mathcal{N}_f$ at $f\in \mathcal{B}$ is given by
  \[ \mathcal{P}^f_1 : C^{2,\alpha}(L_{-1}) \to L^\infty(L_1). \]
  Note that for sufficiently small $f$ the hypersurface $L^f_{-1}$
  (resp. $L^f_1$) is graphical over $L_{-1}$ (resp. $L_1$), and using
  this we can view functions on $L^f_{-1}$ (resp. $L^f_1$) as
  functions on $L_{-1}$ (resp. $L_1$). 
\end{prop}

We next want to show that the map $\mathcal{P}_1$ has dense image in
$C^0(L_1)$. First we recall the setting where no singularities occur
along the flow (see \cite[Theorem 5.2]{HV92} or \cite[Lemma
5.3]{SX22}).
\begin{lemma}\label{lem:smoothdenserange}
  For any $t_1 < 0$ the map $\mathcal{P}_{t_1}$ has dense range in
  $C^0(L_{t_1})$. 
\end{lemma}

Our goal is to extend this to the
case where the flow encounters a non-degenerate singularity. The proof
will also use the backwards uniqueness result shown in Section
\ref{sec:backwardsunique}, for the conjugate linearized equation
\[ \partial_t w= -\Delta w - |A|^2 w + H^2 w. \]
Note that if $u$ solves the equation $(\partial_t - \Delta - |A|^2)u =
0$ and $w$ solves the conjugate equation, then at smooth times $t$ we
have
\begin{equation} \label{eq:ddtinnerprod}
  \begin{aligned} \frac{d}{dt} \int_{L_t} uw\, d\mathcal{H}^n &=
    \int_{L_t}\left ( (\Delta u
    + |A|^2 u)w + u( - \Delta w - |A|^2 w + H^2 w) - H^2 uw\right)\,
  d\mathcal{H}^n \\
  &= 0. \end{aligned} \end{equation}

We will need the following barrier function for the conjugate linearized
equation.
\begin{lemma}\label{lem:conjbarrier}
  There are $B, r_0 > 0$ such that on the
  neighborhood $B_{r_0}(0) \times (0,r_0)$ after the singularity the
  function
  \[ G(x, t) = \exp \left( B\left| \log \left( |x|^2 + t |\log
          t|^2\right)\right|^2\right) \]
  is a subsolution of the conjugate linear equation in the sense that
  \begin{equation}\label{eq:dtG40} \partial_t G \leq - \Delta G -
    (|A|^2 - H^2) G.
    \end{equation}
  For sufficiently large $\Lambda$ the function $\tilde{G} =
  e^{-\Lambda t}G$ can be extended to a global strictly positive
  subsolution, for $t > 0$. 
\end{lemma}
\begin{proof}
  We use a direct calculation to prove \eqref{eq:dtG40}. To simplify the notation below, let us write
  \[ q=|x|^2 + t |\log t|^2. \]
  If $r_0$ is sufficiently small, then by Proposition \ref{positionest}, we have
    \[\label{ord_mag_bar} |x|^2 \geq C^{-1} t|\log t|, \qquad 1 \leq |\log q| \leq |\log t| \]
on our neighborhood. By Proposition \ref{prop:curvxest} and the Haslhofer-Kleiner curvature estimates, i.e. Proposition~\ref{prop:HK16}, we have
  \[\label{A_bdxx} |A| \leq C \frac{|\log |x| |}{|x|}, \]
  
  As a consequence, we claim that
  \[ \label{eq:A51} |A|^2 \leq C \frac{|\log q| |\log t|^2}{q}. \]
  Indeed, in the region where $|x| > t^{1/2}|\log t|$, $q$ is comparable to $|x|^2$, so 
  \[ \frac{|\log q| |\log t|^2}{q} \geq \frac{|\log q|^3}{q} \geq C^{-1}  \frac{|\log |x||^3}{|x|^2} \geq |A|^2.\]
  Similarly, on the region where $|x| \leq t^{1/2}|\log t|$, \eqref{A_bdxx} and \eqref{ord_mag_bar} gives
  \[ |A|^2 \leq C \frac{|\log t|^2}{|x|^2} \leq C \frac{|\log t|}{t}, \]
  using $|x|^2 \geq C^{-1}t|\log t|$. in that region
  $q$ is comparable to $t |\log t|^2$, so 
  \[ \frac{|\log q| |\log t|^2}{q} \geq \frac{|\log q|^3}{q} \geq C^{-1} \frac{|\log t|}{t}, \]
  and so \eqref{eq:A51} follows. From \eqref{A_bdxx} and \eqref{ord_mag_bar} we also have
  \[ \label{eq:A52} |\langle x, H\rangle| \leq C|\log|x|| \leq C|\log t|.\]

Along the mean curvature flow we have $\partial_t |x|^2 = 2\langle x, H\rangle$ and $\Delta |x|^2 = 2n + 2\langle x, H\rangle$. Therefore
\[ (\partial_t + \Delta) q = |\log t|^2 - 2|\log t| + 2n + 4\langle x, H\rangle \geq \frac{1}{2}|\log t|^2, \] 
if $r_0$ is small enough, 
and $|\nabla q|^2 = 4|x^T|^2$. By definition we have $G = \exp(B(\log q)^2)$, so 
\[ \frac{ (\partial_t + \Delta + |A|^2 - H^2)G}{G} &= \partial_t B(\log q)^2 + \Delta B(\log q)^2 +  B^2|\nabla (\log q)^2|^2 + (|A|^2 - H^2) \\
&= 2B(\log q)q^{-1}(\partial_t + \Delta)q + 2B|\nabla\log q|^2 + 2B|\log q| |\nabla \log q|^2\\
&\quad +4B^2 |\log q|^2 |\nabla\log q|^2 + |A|^2 - H^2 \\
&\leq -B|\log q| q^{-1}|\log t|^2 +32 B^2  |\log q|^2 q^{-2} |x^T|^2 + |A|^2.  \]
We use $|x^T|^2\leq q$, and \eqref{eq:A51} to get 
\[  \frac{ (\partial_t + \Delta + |A|^2 - H^2)G}{G} &\leq 
-B|\log q| q^{-1}|\log t|^2 + 32B^2 q^{-1} |\log q|^2 + C|\log q| q^{-1} |\log t|^2 \\
&\leq 0, \]
if $B$ is large enough, and $r_0$ is sufficiently small.

  Use a cutoff function to
  extend $G$ smoothly to all of $L_t$ for $t > 0$ so that $G$
  is strictly positive. Let $\Lambda > 0$ and define $\tilde{G} =
  e^{-\Lambda t}G$. Then we have
  \[\begin{aligned}
      (\partial_t + \Delta + |A|^2 - H^2)\tilde{G} &= e^{-\Lambda t}
      \left( (\partial_t + \Delta + |A|^2 -H^2)G - \Lambda G\right) \\
      &\leq 0, \end{aligned} \]
  if $\Lambda$ is sufficiently large. 
\end{proof}

We are now ready to prove our main claim about the density of the
range of $\mathcal{P}_1$. 
\begin{prop}\label{prop:11}
  Given $h$ on $L_1$ and $\kappa > 0$ we can find some $f \in
  C^{2,\alpha}(L_{-1})$ such that
  \[ |\mathcal{P}_1(f) - h| < \kappa. \]
\end{prop}
\begin{proof}
  Arguing by contradiction, let us suppose that the image of
  $\mathcal{P}_1$ is not dense in $C^0$. Then by duality there is
  a Radon measure $d\zeta$ on $L_1$ such that $\int_{L_1}
  \mathcal{P}_1(f) d\zeta = 0$ for all $f\in C^{2,\alpha}(L_{-1})$. 

  Consider the solution $w(x,t)$ of the conjugate
  linearized equation
  \[ \partial_t w  = -\Delta w - |A|^2 w + H^2 w \]
  for $t\in (0,1]$, with initial condition given by $w(x,1) =
  d\zeta(x)$ (more precisely $w(\cdot, t)\to d\zeta$ in the sense of
  measures as $t\to 1$). Note that for $t \in (0,1]$ the flow $L_t$ is
  smooth, and so by standard parabolic theory the function $w(x,t)$ is
  smooth for $t\in (0,1)$. Using the barrier $\tilde{G}$ constructed
  in Lemma~\ref{lem:conjbarrier}, we have an estimate
  \[ \label{eq:wbound3} |w(x,t)| \leq C \tilde{G}(x,t), \]
  for $t\in (0, \frac{1}{2}]$, for a suitable constant $C$. Since
  $\tilde{G}$ only blows up at the origin, using interior estimates we
  can extract a smooth limit $w(x,0)$ for $x\not=0$. 
  Using Theorem~\ref{thm:backwardsunique},
  we have that $w(x,0)$ is not identically zero. Note that we do not
  need to define $w(x,t)$ for $t < 0$. 

  Since $w(x,0)$ is not identically zero, we can find a smooth
  function $\xi$ on $L_0$, supported away from the origin, such that
  $\int_{L_0} \xi w\, d\mathcal{H}^n \not=0$. Define $r_0 > 0$ such
  that $\xi = 0$ on $L_0 \cap B_{4r_0}(0)$, and let $\epsilon > 0$. If $\epsilon$ is
  sufficiently small, then on the support of $\xi$ we can identify
  $L_0$ with $L_{-\epsilon}$ through graphicality, and so we can view
  $\xi$ as a function $\xi_{-\epsilon}$ on $L_{-\epsilon}$. In
  addition since the flow $L_t$ 
  is smooth on the time interval $[-1, -\epsilon]$, the solution map
 $\mathcal{P}_{-\epsilon}$  of the linearized operator has dense range in
 $C^0$ by  Lemma~\ref{lem:smoothdenserange}, so that we
 can find a function $u_\epsilon$ on $L_{-1}$ such that
 $|\mathcal{P}_{-\epsilon}(u_\epsilon) - \xi_{-\epsilon}| < e^{-1/\epsilon}$.

 We claim that if $\epsilon$ is sufficiently small,  then
 \[ \label{eq:51} \int_{L_\epsilon}
   \mathcal{P}_\epsilon(u_\epsilon)(x) w(x, \epsilon)\, d\mathcal{H}^n
   \not=0. \]
 Assuming this claim for the moment, note that by
 \eqref{eq:ddtinnerprod}, for $t\in (0,1)$ we have 
 \[ \frac{d}{dt} \int_{L_t} \mathcal{P}_t(u_\epsilon)(x) w(x, t)\,
   d\mathcal{H}^n = 0, \]
 and so as $t\to 1$, \eqref{eq:51} together with the weak convergence
 $w(\cdot, t) \to d\zeta$ implies that $\int_{L_1}
 \mathcal{P}_1(u_\epsilon) d\zeta \not=0$, contradicting the choice of
 $d\zeta$.

 It remains to show the claim \eqref{eq:51}. Let $U_\epsilon$ denote
 the admissible solution of the linearized operator for $t >
 -\epsilon$, with initial condition $U_\epsilon(x, -\epsilon) =
 \xi_{-\epsilon}$ on $L_{-\epsilon}$. The existence of this is given
 in Lemma~\ref{lem:solveinhom}. As $\xi$ is supported away from zero, $|\xi_{-\varepsilon}|/|\Phi|$ is uniformly bounded. Therefore, by the maximum principle of Lemma ~\ref{lem:maxprinc2}, there exist a uniform $C<\infty$ such that 
 \[\label{bd_point_phi_bd}
 U_{\varepsilon}(x,t) \leq C\Phi(x,t), \qquad t\in [-\varepsilon,1].
 \]
Consider the annulus 
$\mathcal{A} = B_{3r_0} \setminus B_{r_0}$ around $0$ and recall that  $\xi$ is supported
outside of $B_{3r_0}$. Since the coefficients of the linearized equation are bounded in $\mathcal{A}$, it follows from standard heat kernel estimates, using \eqref{bd_point_phi_bd} as boundary data, that
 \[\label{eq:Uepsbound3_bd} |U_\epsilon| (x, t) \leq C_1 \Phi(x,
   t) e^{-c_1/(t+\varepsilon)}, \] 
   for $x\in \partial B_{2r_0}$. Applying the maximum principle argument of Lemma \ref{lem:maxprinc2} in $B_{2r_0}$ (not globally) and using \eqref{eq:Uepsbound3_bd} for temporal boundary bound, we therefore get that
 \[\label{eq:Uepsbound3} |U_\epsilon| (x, \epsilon) \leq C_1 \Phi(x,
   \epsilon) e^{-c_1/(2\epsilon)}, \] 
   for $x\in B_{2r_0}$.

 Since outside of $B_{r_0}$ the flow is smooth, we have that $|U_\epsilon(x, \epsilon)
 -\xi_\epsilon| \to 0$ as $\epsilon\to 0$ on $L_\epsilon \setminus
 B_{r_0}$. At the same time, in the ball  $L_\epsilon \cap B_{2r_0}$ we have the bound
 in \eqref{eq:Uepsbound3} for $U_\epsilon$ and the bound
 \eqref{eq:wbound3} for $w$.  It follows that 
 \[ \left|\int_{L_\epsilon \cap B_{2r_0}} U_\epsilon(x, \epsilon) w(x,
     \epsilon)\, d\mathcal{H}^n\right| 
  &\leq C (\sup_{L_\epsilon} \Phi) (\sup_{L_\epsilon} \tilde{G})
  e^{-c_1/(2\epsilon)}. \]
 On $L_\epsilon$ we have $\tilde{G} \leq C \exp( C|\log \epsilon|^N)$,
 while $\Phi \leq C \epsilon^{-1/2} |\log \epsilon|^C$. It follows
 that
  \[ \left|\int_{L_\epsilon \cap B_{2r_0}} U_\epsilon(x, \epsilon) w(x, \epsilon)\, d\mathcal{H}^n\right|
    &\leq C e^{-c_1 / (2\epsilon)} \epsilon^{-1/2} |\log \epsilon|^C
    \exp (C |\log \epsilon|^N) \to 0,
    \]
 as $\epsilon\to 0$, since the dominant term is $e^{-c_1 / (2\epsilon)}$. 
 It follows from this, together with the choice of $\xi$, that
 \[ \label{eq:limintUeps}
   \lim_{\epsilon\to 0} \int_{L_\epsilon} U_\epsilon(x, \epsilon) w(x, \epsilon)\,
   d\mathcal{H}^n \not=0. \]
 To show the claim \eqref{eq:51}, it remains to compare $U_\epsilon(x,
 \epsilon)$ with $\mathcal{P}_\epsilon(u_\epsilon)$. Recall that by
 construction, at time $-\epsilon$ we have $U_\epsilon =
 \xi_{-\epsilon}$, and so $|\mathcal{P}_{-\epsilon}(u_\epsilon) -
 U_\epsilon| < e^{-1/\epsilon}$ on $L_{-\epsilon}$. Using the barrier
 $\Phi$ we have
 \[ |\mathcal{P}_\epsilon(u_\epsilon) - U_\epsilon(\cdot, \epsilon)|
   \leq C e^{-1/\epsilon} \Phi, \text{ on } L_\epsilon, \]
 with $C=(\inf_{t\in [-1.1]} \Phi)^{-1} <\infty$, and so
 \[ &\left| \int_{L_\epsilon} (U_\epsilon(x, \epsilon) w(x,\epsilon)-  \mathcal{P}_\epsilon(u_\epsilon)(x, \epsilon) w(x,\epsilon))\,
     d\mathcal{H}^n\right| \leq C e^{-1/\epsilon} \left|
     \int_{L_\epsilon} \Phi(x, \epsilon) \tilde{G}(x, \epsilon)\,
     d\mathcal{H}^n\right|\\ 
    &\qquad\qquad \leq C e^{-1/\epsilon} \epsilon^{-1/2} |\log \epsilon|^C e^{C|\log \epsilon|^2}. \]
 The right hand side converges to zero as $\epsilon \to 0$, so
 from \eqref{eq:limintUeps} we get
 \[ \lim_{\epsilon\to 0} \int_{L_\epsilon}
   \mathcal{P}_\epsilon(u_\epsilon) (x, \epsilon) w(x, \epsilon)\,
   d\mathcal{H}^n \not=0, \]
 and this shows the claim \eqref{eq:51} for sufficiently small
 $\epsilon$. 
\end{proof}

Let us suppose that $f$ is a function given by this proposition. We
will later choose $\kappa$ small. We have the following.
\begin{prop}\label{prop:Ndiff}
  Given $\epsilon > 0$, there exists an $a_0 > 0$ such that if $\Vert g\Vert_{C^{2,\alpha}} <
  a_0$ (also depending on $f$), then we have
  \[  |\mathcal{P}^g_1(f) - \mathcal{P}_1(f)| < \epsilon. \]
\end{prop}

\begin{proof}
  It is enough to show that if $g_i \to 0$ in $C^{2,\alpha}$, then
  $\mathcal{P}^{g_i}_1(f) \to \mathcal{P}^0_1(f)$ in $L^\infty$. Let
  us write $u^i(x, t) = \mathcal{P}^{g_i}_t(f)$, i.e. this is the
  solution of the linearized equation along the flow $L^{g_i}_t$, with
  initial data $f$. Let us denote by $Z_i$ the singular point of the
  flow $L^{g_i}_t$. By Theorem~\ref{thm:isolated} we have $Z_i \to
  (0,0)$ as $i\to\infty$. 

  We can apply Lemma~\ref{lem:solveinhom} to the solutions $u^i$. Note
  here that the barrier function $\Phi$ can be defined along the
  perturbed flows $L^{g_i}_t$ as well, and crucially in the estimate
  $\Phi \leq C(1 + |H|)$, the constant $C$ is uniform in $i$. Indeed, it is
  determined by the constant $\Lambda$ in \eqref{eq:Phidef31}, which
  in turn only depends on curvature bounds for the flows outside a
  mean convex neighborhood of the singularity, and these can be
  bounded uniformly.

  It follows that we have uniform bounds $|u^i| \leq C (1 + |H_i|)$,
  where we emphasize that the $u^i$ are defined along the different
  flows $L^{g_i}_t$, and $H_i$ are the corresponding mean
  curvatures. On any compact set $K$ away from $(0,0)$, this gives
  uniform bounds $|u^i| < C_K$, for sufficiently large $i >
  i(K)$. Using standard interior estimates, this implies that up to
  choosing a subsequence, the $u^i$ converge, locally smoothly away
  from $(0,0)$, to a solution $u$ of the linearized equation along
  $L_t$, with initial condition given by $f$. The upper bounds $|u^i|
  \leq C(1 + |H_i|)$ can be passed to the limit, implying that $u$ is
  admissible, and so by Lemma~\ref{lem:solveinhom} it is the unique
  admissible solution. Therefore $u = \mathcal{P}_t(f)$. As a
  consequence, on $L_1$ we have $u^i \to u$ in $C^\infty$ (identifying
  $L^{g_i}_1$ with $L_1$ using graphicality), and so in particular
  $\mathcal{P}^{g_i}_1(f) \to \mathcal{P}^{0}_1(f)$ in $L^\infty$ as
  required. 
\end{proof}

Given this result, we can compare $\mathcal{N}(af)$ for different
values of $a$:
\[ \begin{aligned}
    \mathcal{N}(a_2 f) - \mathcal{N}(a_1 f) &= \int_{a_1}^{a_2}
    D\mathcal{N}_{sf}(f) \, ds \\
    &= (a_2 - a_1) D\mathcal{N}_0(f) + \int_{a_1}^{a_2}
    (D\mathcal{N}_{sf}(f) - D\mathcal{N}_0(f)) \, ds
  \end{aligned} \]

Recall that $D\mathcal{N}_g = \mathcal{P}^g_1$, and by the choice of
$f$ we have $|D\mathcal{N}_0(f) - h| < \kappa$.  Therefore, using also
Proposition~\ref{prop:11}, we get
\[ |\mathcal{N}(a_2 f) - \mathcal{N}(a_1 f) - (a_2-a_1) h| \leq
  |a_2-a_1| (\kappa + \epsilon). \]
This completes the proof of Proposition~\ref{prop:10}.

Although we have focused on a flow that admits only a single
non-degenerate cylindrical singularity, the same arguments work if
there are finitely many such singularities. Indeed, the key barrier
function constructions can be easily extended to the case of multiple
singularities. We state the general result that we will be using in
the next section.

\begin{prop}\label{prop:initperturb11}
  Suppose that the mean curvature flow $L_t$ is such that $L_0$ and
  $L_T$ are smooth, and $L_t$ admits only non-degenerate cylindrical
  singularities for $t\in (0,T)$. Let $\kappa > 0$ and $h_1, \ldots,
  h_m \in C^\infty(L_T)$.

  There exists an $\epsilon > 0$, and $f_1,
  \ldots, f_m \in C^{2,\alpha}(L_0)$ with the following property. Let
  $\mathbf{a}=(a_1,\ldots, a_m)\in \mathbb{R}^m$ satisfy $|\mathbf{a}|
  < \epsilon$, and 
  define $L^{\mathbf{a}}_t$ to be the flow with initial condition
  given by the graph of $\sum a_j f_j$ over $L_0$. Then for any
  $|\mathbf{a}_1|, |\mathbf{a}_2| < \epsilon$ we can write
  $L^{\mathbf{a}_2}_T$ as the graph of $F$ over $L^{\mathbf{a}_1}_T$,
  and
  \[ \left| F - \sum_{j=1}^m (a_{2,j} - a_{1,j}) h_j \right| \leq
    \kappa |\mathbf{a}_2 - \mathbf{a}_1|, \]
  where the graphicality is used to view $F$ as a function over
  $L_T$. 
\end{prop}

Combining this result with Proposition~\ref{prop:genericperturb} has the following consequence,
analogous to \cite[Proposition 42]{Sz26_2}, and has the same proof.
\begin{prop}\label{prop:perturb23}
  There is an $\underline{r}_0 > 0$, and a function $\Phi(r) > 0$ depending on the dimension and
  entropy bounds such that we have the following. Suppose that the
  flow $L_t$ has an $\ell$-degenerate $(n,k)$-cylindrical singularity
  at the piont $(x_0, t_0)$, and this singularity is
  $\ell$-degenerate at scale $2r$ for some $r <
  \underline{r}_0$. Assume moreover that the time slices $L_0, L_{t_0
    - r^2}$ are smooth, and $L_t$ has only non-degenerate cylindrical
  singularities for $t\in (0, t_0 - r^2)$.

  Then there are arbitrarily small $C^{2,\alpha}$ perturbations $L'_0$
  of $L_0$ such that the corresponding flow $L'_t$ has no
  $\ell'$-degenerate $(n,k')$-cylindrical singularity in the parabolic
  neighborhood $Q_{\Phi(r)}(x_0, t_0)$, if either $k= k'$ and $\ell'
  \leq \ell$, or if $k' > k$.
\end{prop}

In other words we can ensure that after an arbitrarily small
perturbation of the initial condition, the only singularities that are
left in the fixed parabolic neighborhood $Q_{\Phi(r)}(x_0, t_0)$ are
``better'' than the original singularity at $(x_0, t_0)$ in terms of the
hierarchy defined in Definition~\ref{defn:singorder}.

Using this, we can prove our main result, 
Theorem~\ref{thm:main}. Recall that $N_t \subset\mathbb{R}^{n+1}$ is a
mean curvature flow of compact hypersurfaces, with only cylindrical
and spherical singularities. To prove the main result, we will prove
the following more general statement inductively. In this statement
we should think of the flow $L_t$ as being a small perturbation of the
original flow $N_t$.

\begin{prop}
   Suppose that the mean curvature flow $L_t$ satisfies the
   following:
   \begin{itemize}
     \item $L_t$ is smooth for $t\in [0, 2\underline{r}_0]$, for the
       constant $\underline{r}_0$ in
       Proposition~\ref{prop:perturb23}.
     \item For a degenerate singularity type $(k, \ell)$ (i.e. $0 < k
       < n$ and $0\leq \ell < k$), and $T > 0$ we have $\Sigma^>_{k,
         \ell}(L_t) = \emptyset$ for $t\in [0, T]$, i.e. the worst
       degenerate singularities appearing up to time $T$ are $\ell$-degenerate
       $(n,k)$-cylindrical (see Definition~\ref{defn:singorder}).
     \end{itemize}
    Then there are arbitrarily small $C^2$ perturbations $L'_0$ of
    $L_0$ such that the corresponding flow $L'_t$ has only
    non-degenerate cylindrical (and spherical) singularities up to time $T$. 
  \end{prop}
  \begin{proof}
    We will prove this statement by induction on the type of
    singularities that appear, similarly to \cite[Proposition
    46]{Sz26_2}. More precisely we will show that if
    $\Sigma^>_{k, \ell}(L_t) = \emptyset$ for $t\in [0, T]$, then
    we can find an arbitrarily small perturbation $L'_t$ for which we
    have $\Sigma_{k, \ell}(L'_t) \cup \Sigma^>_{k, \ell}(L_t) =
    \emptyset$ for $t\in [0,T]$. Applying this inductively implies the
    result.

    Suppose then that $\Sigma^>_{k, \ell}(L_t) = \emptyset$ for $t\in
    [0,T]$. From Proposition~\ref{prop:singularityorder} it follows
    that in this case $\bigcup_{t\in [0,T]}\Sigma_{k, \ell}(L_t)$ is a
    closed set. For every $X = (x,t) \in \Sigma_{k, \ell}(L_t)$ with
    $t\in [0, T]$ there is an $r_X > 0$ such that $X$ is
    $\ell$-degenerate at scale $r_X$. From
    Proposition~\ref{prop:nearbydegen} we find that for each such $X$
    there is an $r'_X > 0$ and an open set $U_X$, such that if $L'_t$
    is a sufficiently small perturbation of the flow $L_t$, then all
    $X' \in U_X\cap \Sigma_{k,\ell}(L_t')$ for $t\in [0,T]$ are
    $\ell$-degenerate at scale $r_X'$. We can cover $\Sigma_{k,\ell}(L_t)$
    over $t\in [0,T]$ with finitely many such open sets $U_{X_1},
    \ldots, U_{X_N}$, and we let $r_0 = \frac{1}{2}
    \min\{r_{X_i}'\}$. The conclusion is that every $\ell$-degenerate
    $(n,k)$-singularity that appears along a sufficiently close flow
    $L_t'$ for $t\in [0,T]$ is $\ell$-degenerate at scale $2r_0$. We
    will perform successive perturbations of $L_t$, but by choosing
    the perturbations sufficiently small, we ensure that
    every flow that we obtain still satisfies this property.

    Let us consider the flow $L_t$, and let $\rho_0 = \min\{\Phi(r_0 /
    4), r_0/2\}$ in terms of the function given in
    Proposition~\ref{prop:perturb23}. Note that using that proposition
    we will be able to find perturbations of our flow that eliminate
    the $\ell$-degenerate $(n,k)$-cylindrical singularities in
    parabolic $\rho_0$-neighborhoods.

    Let $s_1 \leq T$ denote the first time that a singularity in
    $\Sigma_{k,\ell}(L_t)$ appears (if there is no such singularity,
    then we are done with the inductive step). Note that by our
    assumptions we have $s_1 > 2\underline{r}_0$. We can apply our
    inductive hypothesis to the flow $L_t$ on the interval $[0, s_1 -
    (r_0/2)^2]$, since on this interval we have $\Sigma_{k,\ell}(L_t)
    = \emptyset$. In this way we obtain our first perturbation
    $L^{(1)}_t$ (by perturbing the initial value) which has only
    non-degenerate singularities for $t\in [0, s_1 - (r_0/2)^2]$.

    The perturbed flow $L^{(1)}_t$ still has
    $\Sigma^>_{k,\ell}(L^{(1)}_t) = \emptyset$ for $t \in [0, T]$, and
    by perturbing the flow slightly we can assume that $L^{(1)}_{s_1 -
      (r_0/2)^2}$ is smooth. We can now apply succesive perturbations
    of the initial condition $L^{(1)}_0$, using
    Proposition~\ref{prop:perturb23}, to eliminate all of the
    $\ell$-degenerate $(n,k)$-cylindrical singularities of $L^{(1)}_t$
    in the time interval $[s_1 - (\rho_0/2)^2, s_1 +
    (\rho_0/2)^2]$. Indeed, observe that if $(x_0, t_0)$ is such a
    singularity, with $t_0 \in [s_1 - (\rho_0/2)^2, s_1 +
    (\rho_0/2)^2]$, then
    \[ \frac{r_0^2}{8} \leq t_0 - (s_1 - (r_0/2)^2) \leq
      \frac{r_0^2}{2}. \]
    We apply Proposition~\ref{prop:perturb23},
    letting $r^2 = t_0 - (s_1 - (r_0/2)^2)$, and recalling that all
    the $\ell$-degenerate $(n,k)$-cylindrical singularities that
    appear are $\ell$-degenerate at scale $2r_0$. Using successively
    smaller perturbations of the initial data, we can eliminate these
    singularities in parabolic $\rho_0$-neighborhoods, and so finitely
    many such perturbations are enough to find a flow $L^{(2)}_t$
    which satisfies $\Sigma_{k,\ell}(L^{(2)}_t) = \emptyset$ for $t\in
    [s_1 - (\rho_0/2)^2, s_1 + (\rho_0/2)^2]$. Note that $L^{(2)}_t$
    can still be an arbitrarily small perturbation of $L_t$.

    Since $L^{(2)}_t$ no longer has $\ell$-degenerate $(n,k)$-cylindrical
    (or worse) singularities for  $t \in [0, s_1 + (\rho_0/2)^2]$, we can
    apply the inductive hypothesis again. This leads to a further
    perturbation $L^{(3)}_t$ of the flow which only has non-degenerate
    cylindrical (or spherical) singularities up to time $s_1 +
    (\rho_0/2)^2$. Note that all of this perturbed flow still
    satisfies $\Sigma^>_{k, \ell}(L^{(3)}_t) = \emptyset$ on the
    interval $[0,T]$ that we started with. Repeating the same argument
    finitely many times (note that $\rho_0$ is fixed), we find a
    perturbation $L'_t$ of the original flow $L_t$ which has only
    non-degenerate cylindrical singularities on $[0,T]$, as required. 
  \end{proof}

  This result quickly implies Theorem~\ref{thm:main}.
  \begin{proof}[Proof of Theorem~\ref{thm:main}]
    Let $T = T_{ext} + 1$, where $T_{ext}$ is the extinction time of
the flow $N_t$. In the Proposition above we take $(k, \ell) =
(n-1,0)$, and note that $\Sigma^>_{n-1, 0}(N_t)=\emptyset$ is
automatic. The conclusion is that there are arbitrarily small
perturbations $N'_0$ such that the corresponding flow $N'_t$ has only
non-degenerate cylindrical and spherical singularities. 
\end{proof}

\section{A conditional backwards uniqueness theorem}\label{sec:backwards1}

The main theorem of this section is the following conditional backwards uniqueness theorem for solutions of the conjugate linearized equation after a non-degenerate singularity. 
\begin{theorem}\label{main_uniqueness_thm}
Let $(N_t)_{t\in [0,T]}$ be a mean curvature flow of hypersurfaces in $\mathbb{R}^{n+1}$, smooth in $\mathbb{R}^{n+1}\times[0,T]-\{(0,0)\}$, encountering a non-degenerate neck-pinch singularity at $(0,0)$, and let $u:(N_t)_{t\in [0,T]}\rightarrow \mathbb{R}$ be a solution to the equation
\begin{equation}\label{conj_eg}
Lu +(|A|^2-H^2)u=0,
\end{equation}
where 
\[
Lu=\partial_t+\Delta.
\]
If there exists $\beta>0$  for which 
\begin{equation}\label{prox_ass}
\sup_{x\in N_t}\left(|u(x,t)|+|\nabla u(x,t)|\right) \exp(t^{-\beta})\rightarrow 0, \qquad \textrm{as}\; t\rightarrow 0,
\end{equation}
then $u\equiv 0$ for every $t$.
\end{theorem}

The proof of the conditional  backwards uniqueness theorem, Theorem \ref{main_uniqueness_thm}, is similar to the one in \cite{DH25}, which is in turn based on ideas in \cite{LP}. Our weights are the very same ones as the ones we  used in\cite{DH25}. There are two major differences:
\begin{enumerate}
    \item Unlike \cite{DH25}, we need to contend with the curvature not being type I (and not quickly decaying away from the origin). Many non-derivative terms are easily absorbable by the fact that the weight derivative produces an extra $t^{-\beta}$ term, while the deviation from a type I behavior is only logarithmic, in view of Lemma \ref{initialcurvatureest}.
    \item More importantly, in the analysis there is always a metric change gradient term. In \cite{DH25}, this was dealt with via splitting the estimate to one in the conical part (where one gets Type-I behavior with small constant) and the self-similar part, where one uses the rescaled flow (and for which the metric barely changes there) and additional Gaussian weights. While a similar splitting approach is feasible in this case too (due to Theorem \ref{quo_after2} and  Proposition \ref{curvdecay}), instead we use the  $AH$ lower bound estimate of Theorem  \ref{metric_changed_upper_bound} to control that gradient term directly. For this argument to work, the sharp bound (in the non-neck case) with sharp constant $AH\geq (-1/2-\varepsilon)t^{-1}$ is crucial.
\end{enumerate}

To set up the Carleman estimate, given parameters $\alpha,\beta>0$, consider the weight
\[
\mathcal{G}=\mathcal{G}_{\alpha,\beta}:=e^{\frac{\alpha}{t^\beta}}.
\]
and consider the weighted $L^2$ and $W^{1,2}$ space of functions $u:N_t\rightarrow \mathbb{R}$, whose (pseudo) norms are given by
\begin{equation}
\|u\|^2_{\mathcal{G}}=\int_{0}^T\int_{N_t}u^2\mathcal{G}d\mu_tdt,\qquad \|u\|^2_{\mathcal{G},1}=\int_{0}^T\int_{N_t}|\nabla u|^2\mathcal{G}d\mu_tdt.
\end{equation}

\begin{lemma}\label{first_lem}
Let $M_t$ be a  mean curvature flow on $[0,T_0]$, smooth outside of $(0,0)$, with a non-degenerate neck-pinch singularity at $(0,0)$. For every $\varepsilon>0$ there exists some $T_1\in(0,T_0)$ with the following significance: Let $T\in (0,T_1)$ and $u:(N_t)_{t\in [0,T]}\rightarrow \mathbb{R}$ be a  smooth function such that $u(\cdot,T)=0$ and $u(\cdot ,0)=0$, and suppose that there exists some $\beta_1>0$ such that  
\begin{equation}\label{u_asymptotics}
\sup_{N_t}\left(\left|u\exp\left(\frac{1}{t^{\beta_1}}\right)\right|+\left|\nabla u\exp\left(\frac{1}{t^{\beta_1}}\right)\right|\right) \leq D<\infty.
\end{equation}
Then for every $\beta\in (0,\beta_1)$ and $\alpha \geq 1$ 
\begin{align}
\|Lu\|^2_{\mathcal{G}} &\geq \frac{\alpha \beta(1+\beta/2)}{2} \|t^{-1-\beta/2} u\|^2_{\mathcal{G}}+\frac{\alpha \beta}{2} \|Ht^{-1/2-\beta/2} u\|^2_{\mathcal{G}} \\
&\qquad -\|H\nabla u\|^2_{\mathcal{G}}-(1+\varepsilon)\|t^{-1/2}\nabla u\|^2_{\mathcal{G}}
\end{align}
\end{lemma}

\begin{proof}
Setting 
\[
z:=\sqrt{\mathcal{G}}u=e^{\frac{\alpha}{2t^{\beta}}}u,
\]
we easily compute 
\[
Lu = \mathcal{G}^{-1/2}\frac{\partial z}{\partial t}+\mathcal{G}^{-1/2}\Delta_{g(t)}z+\frac{\alpha \beta}{2t^{1+\beta}}\mathcal{G}^{-1/2}z.
\] 
Thus  
\begin{equation}\label{to_estimate_glob_car_lemma}
\| Lu\|^2_{\mathcal{G}}\geq  2\left\langle \frac{\partial z}{\partial t}, \Delta_{g(t)}z \right\rangle + \alpha\beta \left\langle \frac{\partial z}{\partial t}, \frac{z}{t^{1+\beta}} \right\rangle,
\end{equation}
provided the right hand side is defined. We will integrate by parts the two terms on the right hand side to obtain expressions which do not involve a time derivative of $z$.

For the second term on the right hand side, observe 
\begin{align}\label{to_estimate_glob_car_lemma_1}
 \alpha\beta\left\langle \frac{\partial z}{\partial t}, \frac{z}{t^{1+\beta}} \right\rangle&= \frac{\alpha \beta}{2} \int\partial_t \left(\int \frac{z^2}{t^{1+\beta}}d\mu_t\right) dt+ \frac{\alpha \beta(1+\beta)}{2}\left\|\frac{z}{t^{1+\beta/2}}\right\|^2+ \frac{\alpha \beta}{2}\left\|H\frac{z}{t^{(1+\beta)/2}}\right\|^2 \nonumber \\
 &= \frac{\alpha \beta(1+\beta)}{2}\left\|\frac{z}{t^{1+\beta/2}}\right\|^2+\frac{\alpha \beta}{2}\left\|H\frac{z}{t^{(1+\beta)/2}}\right\|^2,
\end{align}
where the right hand side is well defined by \eqref{u_asymptotics}, and the second equality is valid since the temporal boundary values of $z^2/t^{1+\beta}$ are zero since $u(\cdot,T)=0$ and by \eqref{u_asymptotics} again. Note that the appearance of the involving $H$ is due to the time derivative of the volume form.

For the first term of the right hand side of \eqref{to_estimate_glob_car_lemma}, observe that
\begin{equation}\label{first_term_bla}
2\frac{\partial z}{\partial t}\Delta_{g(t)}z=2\mathrm{div}_{g(t)}\left(\frac{\partial z}{\partial t}\nabla_{g(t)}z\right)-2\left\langle\nabla_{g(t)}\frac{\partial z}{\partial t},\nabla_{g(t)}z\right\rangle.
\end{equation}
Now,
\begin{equation}
2\left\langle\nabla_{g(t)}\frac{dz}{dt},\nabla_{g(t)}z\right\rangle = 2g^{ij}\frac{\partial^2z}{\partial x_i \partial t}\frac{\partial z}{\partial x_j}=\frac{d}{dt}\left(g^{ij}\frac{\partial z}{\partial x_i}\frac{\partial z}{\partial x_j}\right)- \frac{dg^{ij}}{dt}\frac{\partial z}{\partial x_i}\frac{\partial z}{\partial x_j},
\end{equation}
which, after taking  into account the evolution of the metric, 
\begin{equation}
g'(t)=-2HA,\qquad
\end{equation}
yields 
\begin{equation}
2\left\langle\nabla_{g(t)}\frac{dz}{dt},\nabla_{g(t)}z\right\rangle = 2g^{ij}\frac{\partial^2z}{\partial x_i \partial_t}\frac{\partial z}{\partial x_j}=\frac{d}{dt}g(\nabla z,\nabla z)-2HA(\nabla z,\nabla z).
\end{equation}
Substituting this into \eqref{first_term_bla}, the  divergence theorem, the evolution of the volume form,  and the boundary conditions once more we have
\begin{equation}
2\left\langle \frac{\partial z}{\partial t}, \Delta_{g(t)}z \right\rangle =  \iint 2HA(\nabla_{g(t)}z,\nabla_{g(t)}z)d\mu_tdt-\iint H^2|\nabla z|^2d\mu_tdt.
\end{equation}

Combining these with the definition of $z$ and  \eqref{to_estimate_glob_car_lemma}, \eqref{to_estimate_glob_car_lemma_1}, along with Theorem  \ref{metric_changed_upper_bound} yields the desired estimate.
\end{proof}

\begin{theorem}[Global Carleman]\label{smooth_global_Carleman}
Under the same conditions and notations of  Lemma \ref{first_lem},  there exist $T_0>0$, $\delta>0$  and $\alpha_0<\infty$ such that if  $\alpha\geq \alpha_0$, $T\leq T_0$ 
\begin{align}
\|Lu\|^2_{\mathcal{G}} &\geq  \alpha \beta \delta\|t^{-1-\beta/2}u\|^2_{\mathcal{G}}.
\end{align}
\end{theorem}

\begin{proof}
 For this proof, we make use of the fact that by Lemma \ref{initialcurvatureest} and \cite{HK16}, for every $\delta$ and $T$ sufficiently small, we can safely estimate
\begin{equation}\label{junk_track}
t^{1/2+\delta}(|H|+|A|)+t^{3/2+\delta}|\partial_t H|+t^{1+\delta}|\nabla H|\leq 1,
\end{equation}
for all $t\in[0,T]$. These bounds will be used repeatedly.

\bigskip

We start by writing
\begin{equation}\label{start_eq_car2}
\left\langle \frac{H^2\sqrt{\mathcal{G}}}{t^{\beta/4}}u, \sqrt{\mathcal{G}}t^{\beta/4}Lu \right\rangle= \left\langle H^2e^{\frac{\alpha}{t^{\beta}}}u, \frac{\partial u}{\partial t}\right\rangle +\left\langle H^2e^{\frac{\alpha}{t^{\beta}}}u, \Delta_{g(t)}u\right\rangle.
\end{equation}
As before, using the boundary conditions we have
\[
\left\langle H^2e^{\frac{\alpha}{t^{\beta}}}u, \frac{\partial u}{\partial t}\right\rangle=\frac{\alpha\beta }{2}\iint  \frac{H^2\mathcal{G}}{t^{1+\beta}} u^2d\mu_tdt+ \frac{1}{2}\iint  \mathcal{G}(-2H\partial_t H+H^4) u^2d\mu_tdt,
\]
so, utilizing \eqref{junk_track} we have that for every $\delta>0$
\begin{equation}\label{global_car_tech_est1}
\left|\left\langle H^2e^{\frac{\alpha}{t^{\beta}}}u, \frac{\partial u}{\partial t}\right\rangle\right| \leq  \frac{1}{2} \alpha \beta \left\|\frac{Hu}{t^{(1+\beta)/2}}\right\|^2_{\mathcal{G}}+\left\|\frac{u}{t^{1+\beta/4}}\right\|^2_{\mathcal{G}},
\end{equation}
provided $T$ is sufficiently small. Using the divergence theorem on every time slice, and using \eqref{junk_track} again, we also get
\begin{align}\label{global_car_tech_est_2}
\left\langle H^2\mathcal{G}u, \Delta_{g(t)}u\right\rangle\\
=\iint\mathrm{div}_{g(t)}&\left(\mathcal{G} H^2u\nabla u\right)d\mu_tdt-\iint H^2\mathcal{G}|\nabla u|^2-2Hu \nabla H\cdot \nabla u \mathcal{G}d\mu_tdt \\
&\leq -\|H\nabla u\|_{\mathcal{G}}^2+\varepsilon \|ut^{-1-\beta/2}\|^2_{\mathcal{G}}+\|\nabla u H t^{\beta/4}\|^2_{\mathcal{G}}.
\end{align}

\bigskip

Combining \eqref{global_car_tech_est1} and \eqref{global_car_tech_est_2} we get
\begin{align}\label{first_fin_ineq}
 \|H\nabla u\|_{\mathcal{G}}^2 &\leq \left|\left\langle \frac{H^2\sqrt{\mathcal{G}}}{t^{\beta/4}}u, \sqrt{\mathcal{G}}t^{\beta/4}Lu \right\rangle\right|+ \frac{\alpha\beta}{2} \left\| H\frac{u}{t^{(1+\beta)/2}}\right\|^2_{\mathcal{G}}+\varepsilon \|ut^{-1-\beta/2}\|_{\mathcal{G}}^2+\|\nabla u H t^{\beta/4}\|^2_{\mathcal{G}}\\
&\leq T^{\beta/4}\|Lu\|_{\mathcal{G}}\left\|\frac{H^2u}{t^{\beta/4}}\right\|_{\mathcal{G}}+ \frac{\alpha\beta}{2} \left\|\frac{Hu}{t^{(1+\beta)/2}}\right\|^2_{\mathcal{G}}+\varepsilon \|ut^{-1-\beta/2}\|^2_{\mathcal{G}}+\|\nabla u H t^{\beta/4}\|^2_{\mathcal{G}}.
\end{align}

Observe that repeating the above process, but starting with \begin{equation}\label{start_eq_car2var1}
\left\langle \frac{t^{-1}\sqrt{\mathcal{G}}}{t^{\beta/4}}u, \sqrt{\mathcal{G}}t^{\beta/4}Lu \right\rangle= \left\langle \frac{1}{t}e^{\frac{\alpha}{t^{\beta}}}u, \frac{\partial u}{\partial t}\right\rangle +\left\langle \frac{1}{t}e^{\frac{\alpha}{t^{\beta}}}u, \Delta_{g(t)}u\right\rangle.
\end{equation}
we similarly obtain (note that there will be no term corresponding to the $\nabla H\cdot \nabla u$ term above)
\begin{align}\label{come_back_term}
 \left\|\frac{1}{t^{1/2}}\nabla u\right\|_{\mathcal{G}}^2 &\leq  T^{\beta/4}\|Lu\|_{\mathcal{G}}\left\|\frac{u}{t^{1+\beta/4}}\right\|_{\mathcal{G}}+ \frac{\alpha\beta}{2} \left\|\frac{u}{t^{1+\beta/2}}\right\|^2_{\mathcal{G}}+\varepsilon \|ut^{-1-\beta/2}\|^2_{\mathcal{G}}.
\end{align}

\bigskip

Substituting \eqref{first_fin_ineq} into the  previous lemma we get
\begin{align}
\|Lu\|^2_{\mathcal{G}} &\geq \frac{\alpha \beta(1+\beta/2)}{2} \|t^{-1-\beta/2} u\|^2_{\mathcal{G}}+\frac{\alpha \beta}{2} \|Ht^{-1/2-\beta/2} u\|^2_{\mathcal{G}}-(1+\varepsilon)\|t^{-1/2}\nabla u \|^2_{\mathcal{G}}\\
&-T^{\beta/4}\|Lu\|_{\mathcal{G}}\left\|\frac{H^2u}{t^{\beta/4}}\right\|_{\mathcal{G}}- \frac{\alpha\beta}{2} \left\|\frac{Hu}{t^{(1+\beta)/2}}\right\|^2_{\mathcal{G}}- \varepsilon \|ut^{-1-\beta/2}\|^2_{\mathcal{G}}- \|\nabla u H t^{\beta/4}\|^2_{\mathcal{G}}\\
\end{align} 
Note that the third negative term cancels with the second positive term, the fourth negative term is absorbed into the first positive term (if $\alpha$ is large enough) and the second negative term are absorbed into the LHS and the first positive term. Finally, the last negative term is dominated by the first negative term. Thus, we have

\begin{equation}
\|Lu\|^2_{\mathcal{G}} \geq \frac{\alpha \beta(1+\beta/4)}{2} \|t^{-1-\beta/2} u\|^2_{\mathcal{G}}-(1+\varepsilon)\|t^{-1/2}\nabla u  \|^2_{\mathcal{G}},
\end{equation}
when $\alpha$ is large enough and $T$ is small enough. Substituting  \eqref{come_back_term} into this we get
\begin{align}
\|Lu\|^2_{\mathcal{G}} &\geq \frac{\alpha \beta(1+\beta/4)}{2} \|t^{-1-\beta/2} u\|^2_{\mathcal{G}}\\
&-(1+\varepsilon)\left(T^{\beta/4}\|Lu\|_{\mathcal{G}}\left\|\frac{u}{t^{1+\beta/4}}\right\|_{\mathcal{G}}+ \frac{\alpha\beta}{2} \left\|\frac{u}{t^{1+\beta/2}}\right\|^2_{\mathcal{G}}+\varepsilon \|ut^{-1-\beta/2}\|^2_{\mathcal{G}}\right).
\end{align}
Note the the second negative term is absorbed into the positive one, provided $\varepsilon$ (and in turn $T$) were sufficiently small. The other two terms are absorbed into the positive term and the left hand side.  
\end{proof}

\begin{proof}[Proof of Theorem \ref{main_uniqueness_thm}]
It suffices to prove the result when $T$ is small. To extend to larger times, one uses existing smooth backward uniqueness results for parabolic equations with smooth coefficients.  For the purpose of this proof, It's convenient to identify each $N_t$ with the final $N=N_T$, with evolving metric.

Take any $0<t_2<t_1<T$  and let $\varphi$ be a function which is $1$ for $t\in [0,t_2]$, decreasing between $t_2$ and $t_1$ and zero to $t\geq t_1$. Letting $\Omega_1$ be the domain $N\times [t_2,t_1]$ and $\Omega_2$ be the domain $N\times [0,t_2]$, setting $v=\varphi u$,  choosing $\beta\in (0,\beta_1)$ and considering the corresponding $\mathcal{G}$, we have

\begin{align}\label{tech_proof_of_tech_global_car}
\|Lv\|_{\Omega_1,\mathcal{G}}^2+\|Lu\|_{\Omega_2,\mathcal{G}}^2&=\|Lv\|^2_{\mathcal{G}}\geq   \|t^{-1-\beta/2}v\|^2_{\mathcal{G}} \geq  \|t^{-1-\beta/2}u\|^2_{\Omega_2,\mathcal{G}},
\end{align} 
where the first inequality uses Theorem \ref{smooth_global_Carleman} (Global Carleman).

In light of \eqref{conj_eg} and \eqref{junk_track}, if $T$ is sufficiently small  we can absorb the second term in the left hand side of \eqref{tech_proof_of_tech_global_car} into the right to obtain:

\begin{equation}\label{bd_norm}
\|Lv\|_{\Omega_1,\mathcal{G}}^2  \geq  \frac{1}{2}\|t^{-1-\beta/2}u\|^2_{\Omega_2,\mathcal{G}}.
\end{equation}

Now, taking any $t_3\in (0,t_2)$ and setting $\Omega_3=N\times [0,t_3]$ we can safely use \eqref{bd_norm} to obtain
\begin{equation}\label{bd_norm_2}
\|Lv\|_{\Omega_1,\mathcal{G}}^2  \geq  \frac{1}{2}\|t^{-1-\beta/2}u\|^2_{\Omega_2,\mathcal{G}} \geq  \frac{1}{2}\|t^{-1-\beta/2}u\|^2_{\Omega_3,\mathcal{G}} \geq \|u\|^2_{\Omega_3,\mathcal{G}},
\end{equation}
provided $T$ is small enough.

Note however that 
\begin{equation}
q:= \frac{\sup_{t\in [t_2,t_1]} e^{1/t^{\beta}}}{\inf_{t\in [0,t_3]} e^{1/t^{\beta}}}<1,
\end{equation}
so recalling recalling the definition of $\mathcal{G}$, \eqref{bd_norm_2} implies
\begin{equation}
\|u\|_{\Omega_3}^2 \leq q^{\alpha}\|Lv\|_{\Omega_1}^2\rightarrow 0, 
\end{equation}
as $\alpha \rightarrow \infty$. This completes the proof.
\end{proof}

\section{The backwards uniqueness theorem} \label{sec:backwardsunique}
Our goal in this section is to verify the hypotheses of the conditional backwards uniqueness result, Theorem~\ref{main_uniqueness_thm}. The argument will have several stages, following the approach of \cite{DH25}. The main result of this section will then be the following unconditional backwards uniqueness theorem, which we restate from the introduction (Theorem \ref{thm:linearbackwarduniqueness}) for the reader's convenience.

\begin{theorem}\label{thm:backwardsunique}
  Suppose that $N_t\subset \mathbb{R}^{n+1}$ is a mean curvature flow
  that has only non-degenerate cylindrical singularities at time
  $t=t_0$, and which is smooth for $t\in (t_0, t_1]$. Suppose that
  $u(x,t)$ solves the conjugate linearized equation
  \begin{equation} \label{eq:b13} 
  (\partial_t + \Delta + |A|^2 - H^2) u = 0, \end{equation}
  for $t\in (t_0, t_1]$, and as $t\to t_0$ we have $u(x,t)\to 0$
  locally smoothly away from the singular set. Then $u\equiv 0$. 
\end{theorem}

Throughout this section, we will always assume without loss of generality that $t_0=0$. Owing to Theorem \ref{main_uniqueness_thm}, in order to prove Theorem \ref{thm:backwardsunique},  it is sufficient to show that under the assumptions $u$ above $u$ has the fast rate of convergence to zero \eqref{u_asymptotics} as $t\to 0$.

\bigskip 

The basic strategy of the proof of \eqref{u_asymptotics} is similar to the one \cite[Section 4]{DH25}. We first derive a rate of convergence at the regularity scale which (naturally) is almost meaningless around the singularity. We then derive a Carleman estimate, similar to the conical Carleman estimate of \cite{Wang16} and \cite[Sec 4.2]{DH25}, which allows us to push \eqref{u_asymptotics} to points in $N_t$ of  distance at least $K\sqrt{t|\log t|}$ from the tip. Treatment of the flow as approximately conical (and not as cylindrical) after the non-degenerate singularity is facilitated by Lemma \ref{prop:curvxest}, which provides a conical curvature decay, up to a logarithmic factor, and Theorem \ref{lem:almostconical}, which shows that the deviation from  being conical is small, at the curvature scale, at points which are $K\sqrt{t|\log t|}$ away from from the tip.   Once this is done, an iterative `propagation of smallness' argument is used to push \eqref{u_asymptotics} all the way in. More details about each of the steps, including a discussion of the similarity and differences from \cite{DH25} could be found at the beginning of each of the relevant subsections.

\bigskip 

Before continuing, note that as $N_t$ is compact, it follows from Lemma \ref{lem:conjbarrier} that there exists some $D<\infty$ such that 
\begin{equation}
    |u(x,t)| \leq D\tilde{G}(x,t).
\end{equation}
for all $t>0$, where $\tilde{G}$ is in given in that theorem. By Proposition \ref{positionest} and the precise form of $\mathcal{G}$, we see that there exists some $D<\infty$ such that
\begin{equation}\label{w_init_est}
|u(x,t)|\leq D\exp\left(\log^4(|x|)\right),
\end{equation}
for all $t\in (0,t_1]$ and $x\in N_t$. We will assume that bound throughout. 

\subsection{Rate of convergence at the regularity scale}\label{sec:regularity}

In this section we show a decay rate of $u$ at the regularity scale. The main theorem of this section is the following decay estimate, particularly useful at points away from 0.

\begin{theorem}\label{rate_smooth}
Suppose $u$ and $N_t$ are as in Theorem \ref{thm:backwardsunique}. There exist constants $C, c > 0$ with the following significance.
For every point $p\in \mathbb{R}^{n+1}$ we set
\begin{equation}
\rho(p)=\frac{|p|}{\sqrt{|\log(|p|)|}}.
\end{equation}
Then, for every $p\in N_0\cap B(0,r_0)$, $t\in [0,c\rho^2(p)]$ and $x\in N_t\cap B(p,c\rho(p))$, we have the estimate
\begin{equation}
|u(x,t)|+\rho(x)|\nabla u(x,t)| \leq C\exp\left(\log(|x|)^{4}-\frac{c|x|^2}{t|\log |x||}\right).
\end{equation}

\end{theorem}

As in \cite[Theorem 27]{DH25}, this estimate follows from the Carleman inequality of \cite[Theorem 4]{EF03}, similarly to \cite[Lemma 2.1]{Nguyen2010} and \cite[Lemma 3.6]{Wang2014}. The basic PDE Lemma, which appears in various guises in all of the references above, is the following. 

\begin{lemma}\label{EF_smallness}
Given a constant $\Lambda$ there exist some $C<\infty$ and $\delta>0$ with the following significance.
Suppose that $u:B_1\times[0,1]\rightarrow  \mathbb{R}$ satisfies
\begin{equation}\label{assumed_v_bds_car_EF}
 \left|
 \partial_t u
 +\partial_i\!\left(a^{ij}(y,s)\partial_j u\right)
 \right|
 \leq \Lambda \bigl(|u|+|\nabla u|\bigr),
\end{equation}
where the $\Lambda^{-1} I \leq a^{ij} \leq \Lambda I$ and where the $\|a_{ij}\|_{C^4}+\|b_i\|_{C^4} \leq \Lambda$. 
Assume that 
\[
        u(\cdot,0)=0
\]
and that 
\begin{equation}
        \sup_{B_1\times[0,1]}
        \bigl(|u|+|\nabla u|\bigr)
        \le M.
\end{equation}
Then 
\begin{equation}\label{EF_outcome}
        |u(0,t)|+|\nabla u(0,t)|
        \leq 
        C M\exp\left(-\frac{\delta}{t}\right).
\end{equation}
for every $t\in [0,1]$.
\end{lemma}

\begin{remark}
The proof is almost identical to \cite[Theorem 27]{DH25} with the following two small distinctions:
\begin{itemize}
\item  Here $u$ is assumed to  satisfy \eqref{assumed_v_bds_car_EF}, so one does not need to derive \cite[Eq (128)]{DH25}).
\item Here $|u|+|\nabla u| \leq M$ and  does not converge to zero as in \cite{DH25}. This results in a factor of $M^2$ which multiplies the right hand side of \cite[Eq. (136)]{DH25}.  This $M^2$ factor stays at the right hand side of \cite[Eq (139), (140)]{DH25}, from which \eqref{EF_outcome} follows.
\end{itemize}
\end{remark}

\begin{proof}[Proof of Theorem \ref{rate_smooth}]
Consider the parabolically scaled flow
\[
\tilde{N}_t=\lambda (N_{\lambda^{-2}t}-p),\qquad \lambda=c^{-1}\rho(p)^{-1}.
\]
If $c$ is small then, by pseudolocality and Theorem \ref{thm:SWXMain}, we can write the flow $\tilde{N}_t$ as a graph over $B(0,2)\times [0,2]$. Moreover, the function $w(x,t)=u(p+\lambda^{-1}x,\lambda^{-2}t)$ satisfies a (backwards) parabolic PDE with uniform (depending on dimension) parabolicity and coefficients derivative bounds. Standard interior estimates, together with \eqref{w_init_est} yield that on $B(0,1)\times [0,1]$ we have 
\[
|\nabla w| \leq D \exp\left(\log^{4}(|p|)\right).
\]
Applying Lemma \ref{EF_smallness} to $v$ therefore yields
\begin{equation}\label{EF_outcome_app}
        |w(0,t)|+|\nabla w(0,t)|
        \leq 
        C\exp\left(\log^{4}(|p|)\right) \exp\left(-\frac{\delta}{t}\right).
\end{equation}
Scaling this estimate back to $u$ and noticing that $\rho(x) \cong \rho(p)$ for  $x\in B(p,c\rho(p))$, the asserted estimate follows. 
\end{proof}

\subsection{A Carleman estimate at the almost conical neck region}
In this section, we derive a Carleman estimate which is suitable for pushing smallness information from physical distance $r>0$ to physical distance $K\sqrt{t|\log t|}$ from the origin. This Carleman estimate, Theorem \ref{main_car}, and its proof, are the post non-degenerate neck pinch analogue of Theorem \cite[Theorem 30]{DH25}, which was applicable for flows with an AC singularity. The similarity between these two seemingly unrelated regimes,  follows from Proposition \ref{prop:curvxest} and Lemma \ref{lem:almostconical} both providing an aspect of an approximate conical behavior in the domain of interest.

We recall the following result, which is essentially due to Wang~\cite{Wang2014}.

\begin{lemma}[Lemma 3.3 in \cite{Wang2014}]\label{gen_carl}
Let $(N_t)_{t\in [0,T]}$ be a smooth MCF on $\mathbb{R}^{n+1}-\{0\}$. Let $R>0$ and let $\mathcal{G}$ be a smooth positive function in 
\[
Q_{R,T}=\{(x,t)\;|x\in N_t-B(0,R),\;t\in [0,T]\}.
\]
Assume $v,\nabla_{N_t}v\in C^{0}_c(Q_{R,T})$ with $v(\cdot,0)=v(\cdot,T)=0$ smooth in the interior of $Q_{R,T}$. Setting 
\begin{equation}\label{F_def}
\mathcal{F}=\frac{\partial_t\mathcal{G}-\Delta_{N_t}\mathcal{G}}{\mathcal{G}},
\end{equation}
we have the inequality
\begin{align}
\int _0^{T}&\int_{N_t} \left(2\mathrm{Hess}\log \mathcal{G}+2HA-H^2g\right)(\nabla_{N_t}v,\nabla_{N_t}v)\mathcal{G} d\mu_tdt\\
&+\int _0^{T}\int_{N_t}\frac{1}{2}\left(\frac{d\mathcal{F}}{dt}+\Delta_{N_t}\mathcal{F}-H^2\mathcal{F}\right)v^2 \mathcal{G} d\mu_tdt\\
&\leq \int _0^{T}\int_{N_t}\left(\frac{dv}{dt}+\Delta_{N_t}v\right)^2\mathcal{G} d\mu_tdt.
\end{align}
\end{lemma}
\begin{remark}
As is pointed out in \cite[Lemma 29]{DH25}, while Wang had proved this in the case of self-shrinkers, only the mean curvature flow equation was used. Also note that some signs are different from \cite{Wang2014}. The sign of $H^2$ is different because we are dealing with the forward mean curvature flow, rather than the backwards one. The sign of the term $2HA$ is opposite to the sign of $H^2$, as \cite{Wang2014} had implicitly used a non-standard sign convention under which $H=-\sum A(e_i,e_i)$ (which is apparent from equation (3.14)).
\end{remark}

We apply Lemma \ref{gen_carl} on Type-I rescalings of $N_t$: For every time  $t_0\in [0,t_1)$ set
\begin{equation}\label{nTDef}
\tilde{N}_t=\tilde{N}^{t_0}_t:=\frac{1}{t_0^{1/2}}N_{t_0^{-1}t}.
\end{equation}
For parameters $\alpha,K,T,\delta$, set
\begin{equation}\label{G_def}
\mathcal{G}=\mathcal{G}_{\alpha,K,T,\delta,t_0}=e^{\varphi},
\end{equation}
and denote 
\begin{equation}\label{phi_def}
\varphi=\varphi_{\alpha,K,T,\delta,t_0}:= 2\alpha(T-t)\left(|x|^{1+\delta}-R^{1+\delta}\right),
\end{equation}
and 
\begin{equation}\label{R_def}
R=R_{K,t_0}:=K\sqrt{|\log t_0|}.
\end{equation}

\begin{theorem}\label{main_car}
For every $\delta\in (0,1/2)$ and $T\geq 1$ there exist constants $C,K<\infty$ and $r,\mathcal{T}>0$, depending on the original MCF, $N_t$, with the following significance: 
Suppose $t_0\in (0,\mathcal{T})$, suppose $\alpha\geq \log^6(t_0)$ and let $\tilde{N}_t=\tilde{N}_t^{t_0}$ be as in \eqref{nTDef}. Suppose  $v$ is a smooth function on $\tilde{N}_t$, satisfying $v(\cdot,0)=v(\cdot,T)=0$, and whose support is contained in the region 
\begin{equation}\label{omega_def}
\Omega_{R,K,T}=\left\{(x,t)\;|\;R \leq |x| \leq \frac{2r}{\sqrt{|t_0|}},\;\alpha(T-t)\geq KR^{1-\delta}\right\},
\end{equation}
where $R=R_{K,t_0}$ is given by \eqref{R_def}. Then for $\mathcal{G}=\mathcal{G}_{\alpha,K,T,\delta,t_0}$ as in \eqref{G_def} we have
\begin{equation}\label{spec_car}
\begin{aligned}
\iint &\left(\alpha (T-t) |x|^{\delta-1} |\nabla v|^2+\alpha^2(T-t)|x|^{2\delta}v^2\right)\mathcal{G}d\mu_tdt \\
&\qquad \leq C \iint \left( \partial_t v+\Delta v+(|A|^2-H^2)v\right)^2\mathcal{G}d\mu_tdt.
\end{aligned}
\end{equation}
\end{theorem}

\begin{remark}
Both the weight $\mathcal{G}$ and the assumed support of $u$ appearing here are similar, but not identical to the one in \cite[Theorem 30]{DH25} (see also \cite{Wang16}). The biggest difference (apparent after tracing the constants) is that $R$ is not fixed in $t_0$. The time support restriction requiring $\alpha(T-t)\geq KR^{1-\delta}$ comes from a cosmetic tradeoff of not including an extra  free $|x|^2$ term in $\varphi$, as in \cite{DH25}.
\end{remark}

\begin{proof}
We will first apply Lemma \ref{gen_carl} on $\mathcal{G}$ and in particular estimate the coefficient on the left hand side. We will then show that the additional terms in the right hand side of \eqref{spec_car} are absorbable.

The first step is very similar to the proof of Theorem \cite[Theorem 30]{DH25}, which is itself an adaptation of \cite{Wang2014}. While our singularity is cylindrical, rather than asymptotically conical, we will exploit the fact that on the support of $u$, the evolving hypersurface is ``almost conical'' in the sense that $|\langle \nu ,x \rangle A| \leq \varepsilon$, and in that $|A|$ itself (as well as higher derivative estimates) decays conically in this domain. 

\bigskip

First observe that by Lemma~\ref{lem:almostconical}, taking $K$ sufficiently large, we have that 
\begin{equation}\label{Gabor_bd}
|\langle \nu ,x \rangle A| \leq \varepsilon. 
\end{equation}
comparing the Euclidean and hypersurface Hessians, this implies that 
\begin{equation}
\mathrm{Hess}_{N_t}|x|^2\geq 2g-\varepsilon.
\end{equation}
As
\begin{equation}
\mathrm{Hess}_{N_t} |x|^{1+\delta}=\frac{1+\delta}{2} |x|^{\delta-1} \mathrm{Hess}_{N_t}|x|^2 +\frac{1+\delta}{2}\frac{\delta-1}{2}|x|^{\delta-3}d|x|^2\otimes d|x|^2.
\end{equation}
and
\begin{equation}
\frac{1+\delta}{2}\frac{\delta-1}{2}|x|^{\delta-3}d|x|^2\otimes d|x|^2 \geq -|x|^{\delta-1}g,
\end{equation}
if follows that there exists some $c = c(\delta)>0$ such that
\begin{equation}
\mathrm{Hess}_{N_t} |x|^{1+\delta} \geq c|x|^{\delta-1}g \geq 0,
\end{equation}
so
\begin{equation}\label{Hess_bd_car}
\mathrm{Hess}\, \varphi \geq c\alpha(T-t)|x|^{\delta-1}g. 
\end{equation}

\bigskip

We claim that there exist some $C<\infty$ such that for $|x|\geq K\sqrt{t_0|\log(t_0)|}$ and for every $t\in (0,Tt_0]$ we have 
\begin{equation}\label{AlmostCon_bd}
|A_{N_t}(x)|\leq \frac{\sqrt{|\log t_0|}}{|x|}.
\end{equation}
Indeed, by the tip location estimate of Proposition \ref{positionest} we know that if $K$ is large enough there exists some time $t_1>t_0$ such that $|p_{t_1}(v)|=|x|/2$. This and the curvature decay estimate of Proposition \ref{curvdecay} imply
\begin{align}
|A_{N_t}(x,t_1)| \leq &C\sqrt{\frac{-\log t_1}{t_1}}\left(|x|\sqrt{\frac{-\log t_1}{t_1}}\right)^{-1/2}\\ 
&=  C\sqrt{|x|\sqrt{\frac{-\log t_1}{t_1}}}\frac{1}{|x|} \leq \frac{C\sqrt{|\log t_1|}}{|x|} \leq \frac{C\sqrt{|\log t_0|}}{|x|}.
\end{align}
As the flow for times $t\in [0,t_1]$ is cylindrical around $x$ (with smaller curvature at previous times), the very same estimate holds for all time $t\in[0,t_1]$. Finally, as $K$ can be taken much larger than $T$, the flow continues to evolve as cylinder around $x$ up to time $Tt_0$, so \eqref{AlmostCon_bd} indeed holds (possibly giving away a multiplicative constant).

Scaling to $\tilde{N}_t$ and recalling the definition of $R$ and the fact that $u$ is supported on $\Omega_{R,K,T}$, we have that
\begin{equation}\label{A_bd_car}
|A_{\tilde{N}_t}(x)|\leq  C\frac{R}{|x|},\qquad \alpha(T-t) \geq KR^{1-\delta}, \qquad |x|\geq R.
\end{equation}
on the support of $v$. 

Combining \eqref{A_bd_car}, \eqref{Hess_bd_car}, we get that if $K$ is large enough, then  over the support of $u$
\begin{equation}
2\mathrm{Hess}\log \mathcal{G}+2HA-H^2g \geq  c\alpha(T-t)|x|^{\delta-1}g.
\end{equation}

For the second estimate, calculating as in \cite[Theorem 30]{DH25} we have 
\begin{align}\label{F_form}
\mathcal{F}&=-2\alpha(|x|^{1+\delta}-R^{1+\delta})+2\alpha(1+\delta)(T-t)|x|^{\delta-1}H\langle x ,\nu \rangle\nonumber \\
&-2\alpha(T-t)\left[(\delta+1)(\delta-1)|x|^{\delta-3}|x^{\top}|^2+(1+\delta)|x|^{\delta-1}(n-\langle x,\nu \rangle H)\right]\\
&-4\alpha^2(1+\delta)^2(T-t)^2|x|^{2\delta-2}|x^{\top}|^2 \nonumber
\end{align}

In light of \eqref{Gabor_bd} and  \eqref{AlmostCon_bd}, and as $|x^{\top}|$ is almost $|x|$, then exactly is in \cite[Theorem 30]{DH25} we have
\begin{equation}
\mathcal{F} \geq -16(T-t)^2\alpha^2|x|^{2\delta}-20(T-t)\alpha|x|^{1+\delta},
\end{equation} 
and so 
\begin{equation}\label{HsqF}
H^2\mathcal{F} \geq  -D|\log(t_0)|\left((T-t)^2\alpha^2|x|^{2\delta-2}+(T-t)\alpha|x|^{\delta-1}\right)
\end{equation}
Similarly, by the structure of $\mathcal{F}$ and by \eqref{AlmostCon_bd} and by the Haslhofer-Kleiner gradient estimate \cite{HK17}, each space derivative of $\mathcal{F}$ contributes at most a factor of $\frac{\sqrt{|\log t_0|}}{|x|}$, yielding
\begin{equation}\label{DeltaF}
\Delta_{\tilde{N}_t} \mathcal{F} \geq -D|\log(t_0)|\left((T-t)^2\alpha^2|x|^{2\delta-2}+(T-t)\alpha|x|^{\delta-1}\right)
\end{equation}
Similarly, each time derivative of a geometric quantity in $\mathcal{F}$ yields a factor of, at most, $\frac{|\log t_0|}{|x|^2}$, but the derivative of the explicit $t$ appearing therein also contributes. Therefore we obtain
\begin{equation}
\frac{d}{dt} \mathcal{F} \geq 2\alpha^2(T-t)|x|^{2\delta}-D|\log(t_0)|\left((T-t)^2\alpha^2|x|^{2\delta-2}+(T+1-t)\alpha|x|^{\delta-1}\right),
\end{equation}
which combined with \eqref{HsqF} and \eqref{DeltaF} yields
\begin{align}
\frac{d}{dt} \mathcal{F}+\Delta_{\tilde{N}_t} \mathcal{F} -H^2\mathcal{F}  \geq &2\alpha^2(T-t)|x|^{2\delta} -D|\log(t_0)|(T-t)^2\alpha^2|x|^{2\delta-2}\\
&-D|\log(t_0)|(T+1-t)\alpha|x|^{\delta-1}.
\end{align}

Recall that
\begin{equation}\label{struct_ass}
|x|\geq R=K\sqrt{|\log t_0|}, \qquad \alpha(T-t) \geq KR^{1-\delta}\qquad \textrm{over the support of $u$}. 
\end{equation}
Taking $K$ sufficiently large  the first term on the right hand side dominates, so we obtain
\begin{equation}
\frac{d}{dt} \mathcal{F}+\Delta_{\tilde{N}_t} \mathcal{F} -H^2\mathcal{F} \geq \alpha^2(T-t)|x|^{2\delta}
\end{equation}
on the support of $u$. Applying Theorem \ref{gen_carl} we get 
\begin{equation}\label{almost_car}
\iint \left(\alpha (T-t)|x|^{\delta-1}|\nabla v|^2+\alpha^2(T-t)|x|^{2\delta}v^2\right)\mathcal{G}d\mu_tdt \\ \leq C \iint \left( \partial_t v+\Delta v\right)^2\mathcal{G}d\mu_tdt.
\end{equation}
\bigskip

Finally,  by \eqref{struct_ass} and \eqref{AlmostCon_bd}, and the assumption that $\alpha \geq \log^6(t_0)$, the left hand side of \eqref{almost_car} already dominates  $|A|^4v^2$, so the desired result follows.
\end{proof}

\subsection{Decay rate at the almost conical neck region}
In this section, we will get improved decay of $u$ at at points of distance at least $K\sqrt{t|\log(|t|)}$ for $K$ sufficiently large, by employing the Carleman estimate of the last section. This is very similar to \cite[Sec 4.3]{DH25}.
We assume that $u$ solves the equation
\[\label{ev_eq} \partial_t u + \Delta u + (|A|^2 - H^2)u = 0, \]
with initial condition $u(x, 0) = 0$ on $N_0\setminus \{0\}$, and the assumptions of Theorem~\ref{rate_smooth}. 

\begin{theorem}\label{con_cone}
There exist some $K<\infty$,$\mathcal{T}>0$ and $\beta\in (0,1/2)$  with the following significance: For every  $t<\mathcal{T}$ and every point in $N_t-B(0,K\sqrt{t|\log t|})$ we have
\begin{equation}\label{rate_outside_C_ball}
|u(x,t)|+|\nabla u(x,t)| \leq e^{-1/t^{\beta}} 
\end{equation} 
\end{theorem}

\begin{proof}
Set $T=10$, take any $\delta\in (0,1/2)$ and let $K,C,r,\mathcal{T}$ be as in Theorem \ref{main_car}. Thus, choosing $t_0<\mathcal{T}$ and defining $\tilde{N}_t=\tilde{N}_t^{t_0}$ by \eqref{nTDef}, so that the Carleman estimate \eqref{spec_car} holds for $\tilde{N}_t$, any $\alpha \geq \log^6(t_0)$, and for functions supported on $\Omega_{R,K,T}$ (given by \eqref{omega_def}), where $R=K\sqrt{|\log(t_0)|}$.

Take a spatial, radial, cutoff function $\varphi:\mathbb{R}_{+}\rightarrow [0,1]$ which is $0$ for $|x|\leq R$ and for $|x|\geq 2r/\sqrt{t_0}$ and $1$ for $|x|\in [2R, r/\sqrt{t_0}]$ with derivatives estimated by 
\begin{equation}\label{cut_spa_bound}
R|D\varphi|+R^2|D^2\varphi| \leq C_1.
\end{equation}
Let us further choose a temporal cutoff function $\psi=\psi_{\alpha}$ which is $1$ for $t\in [2\varepsilon,T-2KR^{1-\delta}\alpha^{-1}]$ and $0$ for $t\leq \varepsilon$ and $t\geq T-KR^{1-\delta}\alpha^{-1}$ with 
\begin{equation}\label{cut_temp_bound}
|\psi'(t)| \leq 
\begin{cases}
2\varepsilon^{-1}, & \textrm{if } t\in [\varepsilon,2\varepsilon], \\
\frac{2\alpha}{KR^{1-\delta}}, & \textrm{if } t \in [T-2KR^{1-\delta}\alpha^{-1},T-KR^{1-\delta}\alpha^{-1}]:=I_{R,\alpha,\delta}.
\end{cases}
\end{equation}

Observe that if $\varphi=\varphi(|x|)$ satisfies an ambient Hessian bound as in \eqref{cut_spa_bound} then it also satisfies it intrinsically on $\tilde{N}_t$, as over its support 
\begin{equation}
\mathrm{Hess}_{\tilde{N}_t}\varphi(|x|)=\varphi''dr\otimes dr+\varphi'\mathrm{Hess}_{\mathrm{Euc}}|x| +\varphi'A\langle \nabla|x|, \nu \rangle \leq (2+\varepsilon)\frac{C_1}{R^2}g
\end{equation}
by \eqref{cut_spa_bound} and Lemma \ref{lem:almostconical}.

Considering the function 
\[
v(x,t)=u(t_0^{1/2}x,t_0t),
\]
on $\tilde{N}_t$, $v$ still satisfies \eqref{ev_eq}. Define

\begin{equation}
w=v\psi{\varphi}.
\end{equation}
Using \eqref{cut_spa_bound} and \eqref{cut_temp_bound}, and \eqref{ev_eq}, there exists some $C_2=C_2(R)<\infty$, such that 
\begin{equation}\label{w_err_est}
|\partial_t w+\Delta w+|A|^2w-H^2w|^2 \leq \frac{C_2}{R^2}(|v|^2+|\nabla v|^2)\chi_{E_1}+\frac{C_2v^2}{\varepsilon^2}\chi_{E_2}+\frac{C_2\alpha^2v^2}{R^{2-2\delta}}\chi_{E_3},
\end{equation}
where 
\begin{equation}\label{E1_def}
E_1 = \left(\left(B(0,2r/\sqrt{t_0})\setminus B(0,r/\sqrt{t_0})\right) \cup \left(B(0,2R) \setminus B(0,R)\right)\right)\times [0,T],
\end{equation}
\begin{equation}\label{E2_def}
E_2= \left(B(0,2r/\sqrt{t_0})\setminus B(0,R)\right)\times [\varepsilon,2\varepsilon],
\end{equation}
and 
\begin{equation}\label{E3_def}
E_3= \left(B(0,2r/\sqrt{t_0})\setminus B(0,R)\right)\times I_{R,\alpha,\delta}.
\end{equation}

Note that by Theorem \ref{rate_smooth} and the dominated convergence theorem
 \[
 \frac{C_3}{\varepsilon^2}\int_{\varepsilon}^{2\varepsilon} \int_{\tilde{N}_t\cap  \left(B(0,2r/\sqrt{t_0})\setminus B(0,R)\right)}v^2\mathcal{G}d\mu_tdt\rightarrow 0,\qquad\varepsilon\rightarrow 0
 \]
Therefore, applying Theorem \ref{main_car} on $w$,  passing to the limit as $\varepsilon\rightarrow 0$ and redefining $w=v\phi\psi$ without the cutoff around time $0$ (so $\psi=1$ for $t\in [0,1]$) we have

\begin{align}\label{Lu_Carl_app}
\iint& \left(\alpha (T-t) |x|^{\delta-1}|\nabla w|^2+\alpha^2(T-t)|x|^{2\delta}w^2\right)\mathcal{G}d\mu_tdt \\
&\quad \leq \frac{C_3}{R^2}\int_{0}^{T}\int_{\tilde{N}_t\cap  \left(B(0,2R) \setminus B(0,R)\right)}(v^2+|\nabla v|^2)\mathcal{G}d\mu_tdt  \nonumber \\ 
&\qquad +\frac{C_3}{R^2}\int_{0}^{T}\int_{\tilde{N}_t\cap  \left(B(0,2r/\sqrt{t_0})\setminus B(0,r/\sqrt{t_0})\right)}(v^2+|\nabla v|^2)\mathcal{G}d\mu_tdt \\
&\qquad + \frac{C_3\alpha^2}{R^{2-2\delta}}\int_{I_{R,\alpha,\delta}}\int_{\tilde{N}_{t}}v^2 \mathcal{G}d\mu_tdt.
\end{align}
for some $C_3=C_3(C_2)$. 

\bigskip

Observe that for $(x,t)\in \left(\mathbb{R}^{n+1}\setminus B(0,17R)\right) \times (0,T/2]$ we have  
\begin{equation}\label{G_estimate_T/2}
\mathcal{G}\geq  \exp\left(\frac{\alpha T}{2}(17^{1+\delta}-1)R^{1+\delta}\right),
\end{equation}
while for $(x,t)\in \left(B(0,2R)\setminus B(0,R)\right)\times (0,T]$ we similarly have
\begin{equation}\label{G_estimate_anu}
\mathcal{G}\leq  \exp\left(2\alpha T(2^{1+\delta}-1)R^{1+\delta}\right),
\end{equation}
and for $(x,t)\in B(0,2r/\sqrt{t_0})\setminus B(0,r/\sqrt{t_0})\times (0,T]$ we have 
\begin{equation}\label{G_estimate_away}
\mathcal{G}\leq  \exp\left(2\alpha T (2r/\sqrt{t_0})^{1+\delta}\right).
\end{equation}
For the last weight estimate, for $t\in I_{R,\alpha,\delta}$ we have 
\begin{equation}\label{G_estimate_T}
\mathcal{G} \leq \exp\left(4KR^{1-\delta}|x|^{1+\delta} \right).
\end{equation}

\bigskip

By Theorem \ref{rate_smooth}, there exist $\eta>0$, independent of $t_0$,  such that in $\Omega_{R,K,T}$ we have 
\begin{equation}\label{vw_estimate_smooth}
|v|^2+|\nabla v|^2+|w|^2+|\nabla w|^2 \leq D\exp\left(\log^6(t_0)-\frac{\eta|x|^2}{t|\log(t_0)|}\right),
\end{equation}
for some $D<\infty$ independent of $t_0$.

Combining \eqref{vw_estimate_smooth} and \eqref{G_estimate_T}, and noticing that for $a>0$ the function $f(s)=Ca^{1/2-\delta/2}s^{1+\delta}-cs^2/a$ is bounded by a quantity comparable to $a^{2/(1-\delta)}$, we obtain
\begin{align}\label{all_estimate_T}
\int_{I_{R,\alpha,\delta}}\int_{\tilde{N}_{t}}|v|^2 \mathcal{G}d\mu_tdt \leq \int_{I_{R,\alpha,\delta}}\int_{\tilde{N}_t} \exp\left( \log^6(t_0)+D|\log(t_0)|^{2/(1-\delta)}\right)d\mu_tdt \\
\leq \exp\left(\log^8(t_0)\right) . 
\end{align}

Similarly, combining \eqref{vw_estimate_smooth} with \eqref{G_estimate_away}, and recalling $\alpha\geq \log^6(t_0)$, we have 
\begin{align}\label{all_estimate_away}
\int_{0}^{T}\int_{\tilde{N}_t\cap  \left(B(0,2r/\sqrt{t_0})\setminus B(0,r/\sqrt{t_0})\right)}&(v^2+|\nabla v|^2)\mathcal{G}d\mu_tdt \\
&\leq D\exp\left(4\alpha T \left(\frac{2r}{\sqrt{t_0}}\right)^{1+\delta}-\frac{cr^2}{t_0 |\log(t_0)|}\right).
\end{align}
\bigskip

Substituting \eqref{all_estimate_T} and \eqref{all_estimate_away} into the last two terms on  the right-hand side of \eqref{Lu_Carl_app}, plugging the $\mathcal{G}$ bounds  \eqref{G_estimate_T/2} and \eqref{G_estimate_anu} for the remaining two terms and using \eqref{vw_estimate_smooth} to estimate $v^2+|\nabla v|^2\leq D\exp(\log^6(t_0))$ on $B(0,2R)\setminus B(0,R)$ we thus get 
\begin{align}\label{exp_both_sides_calc}
&\alpha^2R^{2\delta} \exp\left(\frac{\alpha T}{2}(17^{1+\delta}-1)R^{1+\delta}\right) \int_{0}^{T/2}\int_{\tilde{N}_t \cap \left(B(0,r/\sqrt{t_0})\setminus B(0,17R)\right)}|v|^2d\mu_tdt \\
&\qquad \leq \frac{C_4}{R^2} \exp\left(4\alpha T(2^{1+\delta}-1)R^{1+\delta}\right)\\
&\qquad+\frac{C_4\alpha^2}{R^{2-2\delta}}\exp\left(\log^6(t_0)\right)+\frac{C_4}{R^2}\exp\left(4\alpha T \left(\frac{2r}{\sqrt{t_0}}\right)^{1+\delta}-\frac{cr^2}{t_0 |\log(t_0)|}\right).
\end{align}
Note that the exponent on the left-hand side of \eqref{exp_both_sides_calc} is larger than the first exponent in the right-hand side by a factor of more than $2$, so we obtain 
\begin{align}
& \int_{0}^{T/2}\int_{\tilde{N}_t \cap \left(B(0,r/\sqrt{t_0})-B(0,17R)\right)}|v|^2d\mu_tdt\\
&\qquad\leq C_4\exp\left(-\alpha TR^{1+\delta}\right) \Bigg(1+ \exp\left(\alpha T \left(\frac{2r}{\sqrt{t_0}}\right)^{1+\delta}-\frac{cr^2}{t_0 |\log t_0|}\right)\\
&\qquad\qquad\qquad\qquad\qquad\qquad\qquad+\exp\left(\log^8(t_0)\right)\Bigg).\nonumber
\end{align}
Finally, set
\[
\alpha=\frac{1}{2}\left(\frac{2r}{\sqrt{t_0}}\right)^{1-2\delta}.
\]
Note that with this choice we indeed have $\alpha\geq \log^6(t_0)$ so the above calculation are justified.  We also have
\[ \alpha T \left(\frac{2r}{\sqrt{t_0}}\right)^{1+\delta}-\frac{cr^2}{t_0 |\log t_0|} < 0,\]
and 
\[- \alpha TR^{1+\delta} + \log^8 t_0 < -\frac{1}{2}\alpha TR^{1+\delta} < -\frac{\theta}{t_0^{(1-2\delta)/2}},\]
for some $\theta > 0$ if $t_0$ is small enough. 
We  obtain that 
\begin{equation}
\int_{0}^{T/2}\int_{\tilde{N}_t \cap \left(B(0,r/\sqrt{t_0})\setminus B(0,17R)\right)}|v|^2d\mu_tdt \leq C_5\exp\left(-\theta/{t_0^{(1-2\delta)/2}}\right),
\end{equation}
for some $\theta>0$ and some $C_5<\infty$. Converting this integral estimate into pointwise $C^1$ estimate follows again from standard parabolic estimates and scaling back to $N_t$
\end{proof}

\subsection{Decay rate around neck-pinch center}

In this section, we will push the estimate from the previous section, valid on point on $\Omega_t=N_t \cap \{|x|\geq K\sqrt{t|\log t|}\}$ for some $K<\infty$ into the ``inner points'', themselves of distance $\geq c_1 \sqrt{t|\log t|}$ from the origin with $c_1 \ll K$. The basic idea is similar to the one in \cite[Sec 4.4]{DH25}, where the smallness from outside the shrinker core is pulled inside.  

An important difference is that while in \cite{DH25} a single application of a ``propagation of smallness'' result was used, in our case one would have to iterate it many times to get from $\Omega_t$ to the tips, and one needs to make sure that these iterations do not overpower the initial decay. This requires keeping track of how the curvature bound improves as we move away from the tip.  Relatedly, we would need a propagation of smallness result that is not wasteful in time, i.e., which spits out smallness on the same time interval as its input.  

\subsubsection{Finding good balls for the iteration mechanism}

Our first task is to find good balls over which we will iterate the propagation of smallness. To that end, we first record the following estimate, which follows from Proposition \ref{curvdecay} and the Haslhofer-Kleiner curvature estimates, Proposition~\ref{prop:HK16}. Recall from \cite[Theorem 1.1 (vii)]{SWX25} that if the flow encounters a non-degenerate $S^{n-k}\times \mathbb{R}^k$ singularity, then after the singular time, in a neighborhood of the singularity the flow is graphical over the dual cylinder $R^{n-k+1}\times S^{k-1}$. For each direction $v\in S^{k-1}(1)$, we have the ``tip'' 
\[ 
p_t(v):=\mathrm{Argmin}_{\{x\in M_t \;|\; \frac{\pi(x)}{|\pi(x)|} =v\}}|x|,
\]
where $\pi:\mathbb{R}^{n-k+1}\times (\mathbb{R}^k-\{0\})\to S^{k-1}(1)$ denotes the normalized projection map to the last $k$ co-ordinates.

\begin{lemma}\label{Decay_lemma_real}
Suppose the non-degenerate neck pinch at the origin has tangent flow $S^{n-k}\times \mathbb{R}^k$. 
For every $K<\infty$ there exist constants $D_1,D_2<\infty$ such that for ever $x\in N_t \cap B(0,K\sqrt{t |\log t|})$ we have the estimate
\begin{equation}
|A_{N_t}(x)| \leq D_1\frac{\sqrt{|\log(t)|}}{\sqrt{t}}{\left[\sqrt{1+D_2\frac{\sqrt{|\log t|}}{\sqrt{t}}\left\|x-p_t\left(\frac{\pi(x)}{|\pi(x)|}\right)\right\|}\right]^{-1}}.
\end{equation} 
\end{lemma}

\begin{lemma}\label{lem:pichoices}
For every $1 \ll \Lambda<\infty$ there exist $D<\infty$ and $t_0>0$ such that for every $t\in (0,t_0)$ and $p\in N_t\cap B(0,K\sqrt{t|\log t|})$, there exists a sequence of of points $\{p_i\}_{i=0}^{I(p)}\subseteq M_t$ with the following significance:
\begin{enumerate}
\item $p_0=p$
\item $p_{I(p)} \in N_t - B(0,(K\sqrt{t|\log t|})$.
\item $p_{i-1}\in B(p_i, r_i)$, where 
\begin{equation}\label{radii_def}
r_i= \frac{1}{100\Lambda|A_{M_t}(p_i)|}, 
\end{equation}
\item $I(p) \leq D \sqrt{|\log(t)|}$.
\end{enumerate}
\end{lemma}
\begin{proof}
Take $p_0=p$, set $v=\pi(p_0)$ and define $r_0$ by \eqref{radii_def}. Now, define $p_{i+1}$ inductively by
\begin{equation}\label{next_point_choice}
p_{i+1}=\mathrm{Argmax}_{\{x\in N_t\cap B(p_i,7r_i/16) \;|\; \frac{\pi(x)}{|\pi(x)|}=v\}}|x|
\end{equation}
Setting $r_{i+1}$ by \eqref{radii_def}, \cite[Theorem 4.2]{SWX25} implies that 
\begin{equation}\label{radi_non_dec}
2r_i \geq r_{i+1}\geq r_i/2,
\end{equation} 
 if $\Lambda$ is large enough, as we are are far below the curvature scale,
as well as the existence of some $\eta>0$ such that 
\begin{equation}\label{points_far}
|p_{i+1}-p_t(v)|\geq |p_{i}-p_t(v)|+\eta r_{i+1}
\end{equation}

As 
\begin{equation}
|A_{N_{t}}| \leq \frac{C\sqrt{|\log(t)|}}{\sqrt{t}},
\end{equation}
\eqref{points_far} implies that this process will depart from $B(0,K\sqrt{t |\log(t)|})$ after some $I(p)$ steps, at which point we stop. By \eqref{next_point_choice} and \eqref{radi_non_dec}, the above constructed points $\{p_i\}_{i=0}^{I(p)}$ satisfy conditions (1)-(3).

\bigskip

Note that by \eqref{points_far} and Lemma \ref{Decay_lemma_real}, we have the recursion inequality
\begin{equation}\label{ind_step_ineq}
r_{n+1} \geq \sqrt{\beta\sum_{i=0}^nr_i},
\end{equation}
with 
\[
\beta=\varepsilon \sqrt{\frac{t}{|\log t |}},
\]
for some $\varepsilon>0$  sufficiently small. Note that, if $\varepsilon>0$ is taken sufficiently small, then $r_0\geq \beta/4$. From \eqref{ind_step_ineq}, we have
\[
r_n\geq \frac{\beta n}{4},
\]
so
\begin{equation}
\sum_{i=0}^n r_i \geq \frac{\beta n^2}{8}.
\end{equation}
Thus
\begin{equation}
\frac{1}{8} I(p)^2 \eta \varepsilon \sqrt{\frac{t}{|\log t |}} \leq |p_{I(p)}-p_t(v)| \leq 2K\sqrt{t|\log t|}. 
\end{equation}
This implies assertion (iv) and completes the proof of the theorem.
\end{proof}

\begin{remark}
    It is worth noting the effect of the improved curvature bound of Lemma~\ref{Decay_lemma_real}, which decays as the point moves away from the tip. If instead we only used a uniform curvature bound of order $t^{-1/2} |\log t|^{1/2}$, then in Lemma~\ref{lem:pichoices} the radii of the balls would all be comparable to $t^{1/2} |\log t|^{-1/2}$, and so we would need $\sim |\log t|$ balls to reach distance $K t^{1/2} |\log t|^{1/2}$, instead of the $|\log t|^{1/2}$ we got above. In the following `propagation of smallness' argument, below, the accumulation of errors would be too great for the argument to close. 
\end{remark}

The following is an immediate corollary from the above lemma, and its proof.
\begin{lemma}
Let  $p\in N_t$ and  $\{p_i\}_{i=1}^{I(p)}$ as in the previous lemma. for every $i\in \{1,\ldots I(p)\}$ there exists a rigid motion $O_i$ of $\mathbb{R}^{n+1}$ mapping $p_i$ to $0$  such that for every $s\in  [t,t+\Lambda^2 r_i^2]$:
\begin{enumerate}
\item The image set $O_i(N_s\cap B(p_i,8r_i))$ contains a graph over $B^n(0,4r_i)$, and 
\item Denoting by $\pi$ the projection of $\mathbb{R}^{n+1}$ to its first $n$ coordinates,
\begin{equation}
\pi\left(O_i(N_s\cap B(p_{i-1},r_{i-1})\right)) \subseteq B^n(0,r_i). 
\end{equation}
\end{enumerate}
\end{lemma}

\begin{proof}
Write $N_t$ as graph over its tangent space around $p_i$, take ${O}_i$ to be the transformation that sends it to the plane $\mathbb{R}^n\times \{0\}$ at $0$.
Since $r_{i+1} \geq r_i/2$, both are far below the regularity scale, even at times and distance of scales $\Lambda$, and as 
\[
|p_i-p_{i+1}|=\frac{7}{16}r_i,
\]
the result holds. 
\end{proof}

\subsubsection{Iterative propagation of smallness}

We record the following theorem from \cite{Vessella2008}. 

\begin{theorem}[Theorem 3.3.3 in \cite{Vessella2008}]\label{Vessella}
For every $\lambda<\infty$ there exist constants $C$ and $R<\infty$  and $\theta\in (0,1)$ with the following significance:
Suppose that $a^{ij}$ are smooth functions on $B^n(0,R)\times [0,R^2]$ with
\begin{equation}\label{coeff_bds}
\|a^{ij}\|_{C^2(B(0,R)\times [0,R^2])} \leq \lambda, \qquad \lambda^{-1}I \leq a^{ij} \leq \lambda I,
\end{equation}
and that $u$ is a smooth function on $B^n(0,R)\times [0,R^2]$ satisfying

\begin{equation}\label{eq_bd}
\left|\partial_tu+\partial_i\left(a^{ij}\partial_j u\right)\right|\leq \lambda (|u(x,t)|+|\nabla u(x,t)|).
\end{equation}
Then
\begin{equation}
\int_{B(0,4)} u^2(\cdot ,0) \leq C\left(\iint_{B(0,R)\times [0,R^2]} u^2\right)^{1-\theta}\left(\int_{B(0,1)}u^2(\cdot,0)\right)^{\theta}.
\end{equation}
\end{theorem}

\begin{remark}
The main difference, in terms of the applicability, between Theorem  \ref{Vessella} and \cite[Theorem 3.3.10]{Isakov1998}, which was used in \cite{DH25} is that while the latter translates smallness at some time interval to smallness elsewhere on a sub-interval, Theorem \ref{Vessella} smallness at a given final time to smallness elsewhere at the same time, provided one has additional boundedness at all times everywhere. This no time loss is crucial for our current application: as we will be iterating this result, losing time would  mean more iterations, overpowering the initial exterior decay.
\end{remark}

 Using the previously selected points and Theorem \ref{Vessella}, we will push the decay established in Theorem \ref{con_cone} into the remaining region, obtaining a decay which will allow for the use of the conditional uniqueness theorem, Theorem \ref{main_uniqueness_thm}. This will complete the proof of the backwards uniqueness theorem. 

\begin{proof}[Proof of Theorem \ref{thm:linearbackwarduniqueness}]
As always, we assume that $u$ is defined on a time interval $[0,t_1$]. By Theorem \ref{con_cone} we  know that there exist some $K<\infty$ and $\beta>0$ such that
\[
|u(x,t)|+|\nabla u(x,t)|<e^{-1/t^{\beta}}, \label{dec_out}
\]
for all $x\in N_t-B(0,(K-1)\sqrt{t\log t})$, and $t$ sufficiently small. In order to apply Theorem \ref{main_uniqueness_thm} we need to show that \eqref{dec_out} holds true, with a smaller $\beta$, inside $B(0,(K-1)\sqrt{t\log t})$ as well.

\bigskip

Considering the radii choice formula \eqref{radii_def}, for every $\varepsilon>0$ we can choose $\Gamma < \infty$ such that for every point $p$ and the corresponding points $\{p_i\}_{i=0}^{I(p)}$ and radii $r_i$ as in the previous section with that choice of $\Gamma$, we have that the flow
\begin{equation}
N^i_s:=\frac{1}{r_i}(N_{t+s/r_i^2}-p_i)
\end{equation}
is a graph of a function $h(x,t)$ over $T_{0}N^i_0$ inside $B(0,\sqrt{\Lambda})$ for times $s\in [0,\Lambda]$, with
\begin{equation}
\|h\|_{C^4(B(0,\sqrt{\Lambda})\times [0,\Lambda]}) \leq \varepsilon.
\end{equation}  
Therefore, if $\Gamma$ is chosen large enough, setting 
\[
u^i(x,s)=u(p_i+r_ix, t+r_i^2s)
\]
and thinking of $N^i_s$ in the graphical gauge, we have that $u^i$ satisfy the assumption of Theorem \ref{Vessella} with $\lambda=2$. Possibly enlarging $\Gamma$ we therefore get that
\begin{equation}
\int_{B(0,4)} u_i^2(\cdot ,0) \leq C\left(\iint_{B(0,R)\times [0,R^2]} u_i^2\right)^{1-\theta}\left(\int_{B(0,1)}u_i^2(\cdot,0)\right)^{\theta}.
\end{equation}
for some uniform $R$.

As $|u|\leq C\exp(|\log(t)|^{4})$ by \eqref{w_init_est} we therefore get that
\begin{equation}
\int_{B\left(p_{i-1},r_{i-1}\right)} u^2(x,t) \leq C\exp\left(|\log(t)|^{2A_0}\right)  \left(\int_{B\left(p_i,r_i\right)}u^2(x,t)\right)^{\theta},
\end{equation}
for some $A_0<\infty$. Denoting $I=I(p)$ and iterating this gives
\begin{equation}
\int_{B\left(p_{0},r_0\right)} u^2(x,t) \leq C^{I}\exp\left(I|\log(t)|^{2A_0}\right)\left(\int_{B\left(p_I,r_I\right)}u^2(x,t)\right)^{\theta^I}.
\end{equation}

Note that $B(p_I, r_I)$ is in the region where \eqref{dec_out} is valid, and also $I\leq D\sqrt{|\log(t)|}$. Therefore
\begin{equation}
\left(\int_{B\left(p_I,r_I\right)}u^2(x,t)\right)^{\theta^I} \leq \exp\left(-\frac{\theta^{D\sqrt{|\log(t)|}}}{t^{\beta}}\right)\leq \exp\left(-t^{-\beta/2}\right)
\end{equation}
for $t$ sufficiently small.  Similarly 
\[
C^{I}\exp\left(I|\log(t)|^{2A_0}\right) \leq \exp(t^{-\beta/4}),
\]
so we get 
\[
\int_{B\left(p_{0},r_0\right)} u^2(x,t) \leq \exp(-t^{-\beta/4}),
\]
for $t$ sufficiently small. Using standard parabolic estimates, we get that 
\[
|u(x,t)|+|\nabla u(x,t)|<e^{-1/t^{\beta/5}}, \label{dec_out2},
\]
for sufficiently small $t$. That $u\equiv 0$ now follows from Theorem \ref{main_uniqueness_thm}.
\end{proof}
\bibliography{mybib}{}
\bibliographystyle{alpha}

\end{document}